\documentclass[12pt]{amsart}

\usepackage{tikz}
\usetikzlibrary{snakes}
\usepackage[colorlinks=true, linkcolor=blue, anchorcolor=blue, citecolor=blue, filecolor=blue, menucolor= blue, urlcolor=blue]{hyperref}
\usepackage{amsmath,amssymb,amsthm,amscd}
\usepackage{verbatim}
\usepackage{comment}
\usepackage{multirow}
\usepackage{mathtools}
\usepackage{enumitem}
\usepackage{pgf,tikz}
\usetikzlibrary{arrows}
\usepackage[normalem]{ulem}

\usepackage[margin=1in]{geometry} 

\numberwithin{equation}{section}

\newtheorem{theorem}{Theorem}[section]
\newtheorem{proposition}[theorem]{Proposition}
\newtheorem{lemma}[theorem]{Lemma}
\newtheorem{corollary}[theorem]{Corollary}
\newtheorem{problem}[theorem]{Problem}
\newtheorem{question}[theorem]{Question}
\newtheorem{conjecture}[theorem]{Conjecture}
\newtheorem*{theorem*}{Theorem}

\theoremstyle{definition}
\newtheorem{definition}[theorem]{Definition}

\newtheorem{example}[theorem]{Example}
\newtheorem{remark}[theorem]{Remark}

\newcommand{\PP}{ \ensuremath{\mathbb{P}}}

\newcommand{\cE}{\mathcal{E}}
\newcommand{\mA}{\mathcal{A}}

\newcommand{\mD}{\mathcal{D}}

\newcommand{\mI}{\mathcal{I}}

\newcommand{\mO}{\mathcal{O}}

\DeclareMathOperator{\Jac}{Jac}
\DeclareMathOperator{\coker}{coker}

\definecolor{MyDarkGreen}{cmyk}{0.7,0,1,0}

\def\cocoa{{\hbox{\rm C\kern-.13em o\kern-.07em C\kern-.13em o\kern-.15em A}}}

\begin{document}

\title[Line arrangements]{Line arrangements and configurations of points with an unusual geometric property}

\author{D.\ Cook II}
\address{Department of Mathematics and Computer Science\\
Eastern Illinois University\\
600 Lincoln Avenue\\
Charleston, IL 61920-3099 USA 
\newline 
\emph{Current address}: {Google, Inc., 111 Eighth Avenue, $4^{\rm th}$ Floor, New York, NY 10011}}
\email{dcook.math@gmail.com}

\author{B.\ Harbourne}
\address{Department of Mathematics\\
University of Nebraska\\
Lincoln, NE 68588-0130 USA}
\email{bharbour@math.unl.edu}

\author{J.\ Migliore} 
\address{Department of Mathematics \\
University of Notre Dame \\
Notre Dame, IN 46556 USA}
 \email{migliore.1@nd.edu}

\author{U.\ Nagel}
\address{Department of Mathematics\\
University of Kentucky\\
715 Patterson Office Tower\\
Lexington, KY 40506-0027 USA}
\email{uwe.nagel@uky.edu}

\begin{abstract} 
    The SHGH conjecture proposes a solution to the question of how many conditions a general union of fat points imposes on
    the complete linear system of curves in $\mathbb P^2$ of fixed degree $d$, and it is known to be true in many cases. We
    propose a new problem, namely to understand the number of conditions imposed by a general union of fat points on the
    incomplete linear system defined by the condition of passing through a given finite set of points $Z$ (not general). Motivated
    by work of Di Gennaro-Ilardi-Vall\`es and Faenzi-Vall\`es, we give a careful analysis for the case where there
    is a single general fat point, which has multiplicity $d-1$. There is an expected number of conditions imposed by this fat point, and we
    study those $Z$ for which this expected value is not achieved. We show, for instance, that if $Z$ is in linear general position
    then such unexpected curves do not exist.  We give criteria for the occurrence of such unexpected curves and describe the range
of values of $d$ for which they occur. Unexpected curves have a very particular structure, which we describe, and they are often unique for a given set of points. In particular, we give a criterion for when they are irreducible, and we exhibit examples both where they are reducible and where they are irreducible. Furthermore, we relate properties
    of $Z$ to properties of the arrangement  of lines dual to the points of $Z$. In particular, we obtain a new interpretation of
    the splitting type of a line arrangement. Finally, we use our results to establish a Lefschetz-like criterion for Terao's
    conjecture on the freeness of line arrangements.
\end{abstract}

\date{February 22, 2017}

\thanks{
{\bf Acknowledgements}: Harbourne was partially supported by NSA grant number  H98230-13-1-0213.
Migliore was partially supported by NSA grant number H98230-12-1-0204 and by Simons Foundation grant \#309556.
Nagel was partially supported by NSA grant number H98230-12-1-0247 and by Simons Foundation grant \#317096.
We thank I.\ Dolgachev, G. Ilardi, A.\ Langer, H.\ Schenck, A.\ Seceleanu, T.\ Szemberg, J.\ Szpond, S.\ Tohaneanu and J.\ Vall\`es 
for their comments on this paper.
We also thank the Mathematisches Forschunginstitut Oberwolfach and the Banff International Research Station for supporting the workshops 
which gave us the opportunity to discuss and present our results.}



\maketitle

%\tableofcontents

%%%%%%%%%%%%%%%%%%%%%%%%%%%%%%%%%%%%%%%%%%%%%%%%%%%%%%%%%%%%%%

\section{Introduction}

A fundamental problem in algebraic geometry is the study of the dimension of linear systems on projective varieties, and many tools have been developed by researchers to this end (e.g. the different versions of the Riemann-Roch theorem). It is usually the case that there is an {\em expected} dimension (or codimension), given by naively counting constants; understanding the  {\em special} linear systems, i.e., those whose actual dimensions are greater than the expected ones, is a subtle problem of substantial interest.

For example, consider the complete linear system $\mathcal V=\mathcal L_j$ of plane curves of degree $j$; its (projective) dimension is  $\binom{j+2}{2}-1$. For $j\geq m$,  the requirement that the curves all have multiplicity at least $m$ at a fixed point $P$ imposes $\binom{m+1}{2}$ linear conditions, and the   linear subsystem of all such curves indeed has codimension $\binom{m+1}{2}$ in $\mathcal V$,  so the actual and expected codimensions coincide. We will refer to this as the  linear subsystem of curves passing through a {\em fat point of multiplicity $m$}  supported at $P$. It is a very well-studied (but still open) problem to compute the dimension of the linear subsystem of $\mathcal L_j$ of curves of degree $j$ passing through a {\em general} set of $r$ fat points $P_1,\ldots,P_r$ with multiplicities $m_1,\dots, m_r$. The still open {\em SHGH Conjecture} gives a putative solution to this problem; we will recall  this conjecture in more detail below. 
When $m_1=\ldots=m_r=2$, results of Alexander-Hirschowitz not only confirm the SHGH Conjecture for those cases, but also solve the 
corresponding problem for double points in projective spaces of dimension greater than~2; however, little is known for fat points with arbitrary multiplicity in higher dimensions. 

Motivated by results in this paper described below, we propose a refinement of the above problem. That is, rather than beginning with $\mathcal V = \mathcal L_j$, we propose to begin with the linear system $\mathcal V= \mathcal L_{Z,j}$ of all plane curves of degree $j$  containing a fixed, reduced 0-dimensional scheme $Z$. We then impose the passage through a general set $X$ of fat points and ask for the dimension of the resulting linear subsystem.  The expected dimension depends only on the dimension of the homogeneous component $[I_Z]_j$ of degree $j$ of the ideal of $Z$ and the number of points of $X$, counted with multiplicity: each point of multiplicity $m$ is expected to impose $\binom{m+1}{2}$ independent conditions, as long as the expected dimension of the linear system is non-negative. 

The problem in this generality is currently inaccessible; the case that $X$ is an arbitrary finite general set of fat points and $Z=\varnothing$, for example, has only a conjectural solution, given by the still open SHGH Conjecture. So for this paper we begin a study of this problem by focusing on the first nontrivial case at the other extreme, namely, $X$ a single fat point  of multiplicity $j-1$ and $Z$  an arbitrary finite reduced set of points.
It is surprising (as the example of \cite{DIV} in the next paragraph shows) that already in this case, it is no longer true that the expected dimension is necessarily achieved, as it was when we began with $\mathcal V= \mathcal L_j$ (i.e., when $X$ is one fat point and $Z=\varnothing$). Since $Z$ is {\em not} assumed to be a general set of points, the problem obtains a new and central aspect, namely to understand how the geometry of $Z$ can affect the desired dimension. 
In this paper we  carefully analyze  this surprising behavior. Furthermore, we show that our results have interesting connections to the study of line arrangements. In particular, they give  new perspectives on Terao's freeness conjecture, including a generalization to non-free arrangements.

Our original inspiration came in two ways, from a paper of 
Di Gennaro, Ilardi, and Vall\`es  \cite{DIV}. 
The first was by an example of \cite{DIV}, in which
 they observe that the set of nine points in $\PP^2$  dual to 
 the so-called B3 arrangement has an unusual geometric property \cite[Proposition 7.3]{DIV}: For every point $P$ of the plane, there is a degree four curve passing through these nine points and vanishing to order three at  $P$. This is surprising because a naive dimension count suggests that the  linear system of curves of degree 4 containing the nine points and $3P$ should be empty except for a special locus of points $P$, but in fact it is nonempty for a general point~$P$. 

This led us to study finite sets of points 
$Z$ in the plane for which, for some integer $j$, the dimension of the linear system of plane
curves of degree $j+1$ that pass  through the points of $Z$ and have multiplicity $j$ at a general point $P$ is unexpectedly large. In this case, we say that $Z$ {\em admits (or has) an unexpected curve} of degree $j+1$ (see Definition \ref{def:unexpected curve}). We establish a numerical criterion for the occurrence of unexpected curves. It involves two invariants. The first one, which arose already in the work of Faenzi and Vall\'es \cite{FV2}, we call the \emph{multiplicity index}  $m_Z$ of $Z$. It  is the least integer $j$ such that the linear system of degree $j+1$ forms vanishing at $Z + jP$ (the scheme defined by the ideal $I_P^j\cap I_Z$) is not empty (see Definition \ref{def:mult speciality index}). 
The second invariant, which is new,  is $t_Z := \min \big\{j \ge 0 \; : \; h^0(\mathcal I_Z(j+1)) - \binom{j+1}{2} > 0 \big\}$ (see Definition \ref{def: tZ}). It depends only on the Hilbert function of $Z$.

It turns out that a set $Z$ of points can have unexpected curves of various degrees. To understand this range of degrees we introduce  another new invariant, $u_Z$, called the \emph{speciality index} of $Z$,  as the least integer $j$ such that the scheme $Z + j P$, where $P$ is a general point, imposes independent conditions on forms of degree $j+1$ (see Definition \ref{def:mult speciality index}). Our first main result (see Theorem~\ref{u_ZTheorem}) is: 

\begin{theorem}
     \label{mainThm1} 
$Z$ admits an unexpected curve if and only if $m_Z < t_Z$. Furthermore, in this case $Z$ has an 
unexpected curve of degree $j$ if and only if $m_Z < j \le u_Z$. 
\end{theorem}

In particular, the existence of an unexpected curve forces $m_Z < u_Z$. 
The converse is almost but not quite true. Example \ref{ctrex to DIV} gives a counterexample to the converse. 
It has $m_Z < u_Z$ and admits no unexpected curve. However, $Z$ has a subset of at least $m_Z+2$ collinear points. 
This led us to the following more geometric version of Theorem \ref{mainThm1} (see Corollary \ref{cor:geomVersion}):

\begin{theorem}
   \label{thm:geomVersion}
Let $Z \subset \PP^2$ be a finite set of points. Then 
$Z$ admits an unexpected curve if and only if
$2m_Z+2<|Z|$ but no subset of $m_Z+2$ (or more) of the points is collinear.
In this case, $Z$ has an 
unexpected curve of degree $j$ if and only if $m_Z < j \le |Z|-m_Z-2$. 
\end{theorem}

As we will show (see Lemma \ref{m, t and u}(c)), $2m_Z+2<|Z|$ is equivalent to $m_Z<u_Z$.
Thus $m_Z<u_Z$ together with there being no large collinear subsets of $Z$
implies the occurrence of unexpected curves.

We also show that unexpected curves have a very particular structure. If $Z$ has any unexpected  
curve, then the unexpected curve of degree $m_Z  + 1$ is uniquely determined by $Z$ and the 
general point $P$. Denote it by $C_P (Z)$. Any other unexpected curve of $Z$ associated to $P$ contains $C_P (Z)$ (see 
Proposition \ref{prop:vector space decomposition}). Moreover, the curve $C_P (Z)$ is either  irreducible or 
it is the union of a reduced irreducible curve unexpected with respect to a proper subset $Z' \neq \emptyset$ of $Z$ and 
the $|Z \setminus Z'|$ lines through $P$ and a point of $Z \setminus Z'$ (see 
Theorem \ref{thm:unexp curve structure}). The curve $C_P (Z')$ is rational, and we give  a 
parametrization of it (see Proposition \ref{paramprop1}).

One conclusion that can be made from the aforementioned results is that understanding unexpected curves
reduces to understanding irreducible ones, since whenever $Z$ gives an unexpected
curve, then $Z$ uniquely determines a subset $Z'$ which gives an irreducible unexpected curve,
and $Z$ arises from $Z'$ in a prescribed way (see Remark \ref{constructing unexpected curves}).

By Theorem \ref{mainThm1}, checking for the existence of unexpected curves  requires computing $m_Z$ and $t_Z$.
Since $t_Z$ depends only on the fixed reduced scheme $Z$, it is typically easy to compute. 
In contrast $m_Z$ is much harder to compute rigorously (although one can get experimental evidence for its value
using randomly selected points $P$). Work of Faenzi and Vall\`es \cite{FV2} relates $m_Z$ to properties of
the arrangement of lines $\mA_Z$ dual to the points of $Z$.

Recall that associated to any line arrangement $\mA_Z$ is a locally free sheaf $\mD_Z$ of rank 
two, called the derivation bundle. Restricted to a general line $L$, it splits as 
$\mO_L (-a_Z) \oplus \mO_L (-b_Z)$ with $a_Z + b_Z = |Z| -1$. The pair $(a_Z, b_Z)$, where $a_Z \le b_Z$, is called the \emph{splitting type} of $\mD_Z$ or $\mA_Z$. Theorem 4.3 in  \cite{FV2}  
shows that the number $a_Z$ is equal to the multiplicity index $m_Z$. We observe that  
$b_Z=u_Z+1$ (see Lemma \ref{m, t and u}). This allows us to translate our results about finite sets 
of points into statements on line arrangements. In the other direction, we use methods for studying 
line arrangements to determine multiplicity indices of sets of points. For example, we determine the 
multiplicity index and the speciality index of a set of points in linearly general position and conclude 
that such a set does not admit any unexpected curves (see Corollary \ref{cor:dZ lin gen position}).  
We also show that the set of points dual to a Fermat configuration of $3 t \ge 15$ lines admits unexpected curves of degrees $t+2.\ldots,2t-3$, and that the unexpected curve of degree $t+2$ is irreducible (see Proposition \ref{FermatProp}). 
Furthermore, we exhibit a family of free line arrangements, defined over the rational numbers, with the property that any of the dual sets of points admits a unique unexpected curve which is in fact  irreducible (see Proposition~\ref{prop:unexpected irr}). This relies on new stability criteria for derivation bundles (see Lemma~\ref{hal thm}).\smallskip 

The second way that \cite{DIV} inspired us relates to a fundamental open problem in the study of 
hyperplane arrangements, namely, Terao's conjecture, which is open even for line arrangements. A 
line arrangement $\mA = \mA (f)$ is said to be \emph{free} if the Jacobian ideal of $f$ is saturated, 
where $f$ is the product of linear forms defining the lines in $\mA$. Terao conjectured that freeness 
is a combinatorial property, that is, it depends only on the incidence lattice of the lines in $\mA$. In 
\cite{DIV}, the authors give an equivalent version of Terao's conjecture in terms of Lefschetz properties. 
In trying to understand their proof we realized that some of the results used in \cite{DIV} to derive 
the claimed equivalence are not quite true as stated. We use our results on points to clarify and to 
adjust the needed results. For example, in Theorem~\ref{SLP condition} we show that the existence 
of an unexpected curve is equivalent to the failure of a certain Lefschetz property. We also 
 establish that Terao's conjecture is equivalent to a Lefschetz-like condition (see Proposition \ref{prop:Terao equiv}). This allows us to show that the (adjusted) Lefschetz condition given in \cite{DIV} implies Terao's conjecture (see Corollary \ref{cor:suff Terao}). We do not know if this condition is also necessary. We observe that the condition suggests that, for a set of points, having maximal multiplicity index is a combinatorial property. If that is true, then Terao's conjecture for line arrangements is a consequence (see Corollary \ref{cor:points to Terao}). 
%we do establish that Terao's conjecture follows if one can show that the splitting type of a free line arrangement is a combinatorial property (see Corollary \ref{cor:type of free comb}). We wonder, if, in fact, the splitting type of an arbitrary line arrangement is determined by its incidence lattice. 
\smallskip

We end the introduction with the more detailed discussion of the SHGH Conjecture which we promised above
in the context of the larger problem which frames the work we are doing here. 
Let $V = [R]_j$ be the vector space of degree $j$ forms in three variables, let $\mathcal L_j$ be its projectivization, and let $X = m_1 P_1 + \cdots m_r P_r$ be a fat point scheme supported on a set of $r$ points $P_1,\ldots,P_r$. Thus $X$ is defined by 
\[
I_X = I_{P_1}^{m_1} \cap \cdots \cap I_{P_r}^{m_r}.
\]
We say that $X$ {\em fails to impose the expected number of conditions} on $V$ (or on $\mathcal L_j$) if
\[
\dim_K [I_X]_j > \max \left \{0, \ \dim_K V - \sum_i \binom{m_i+1}{2} \right \} = 
\max \left \{0, \ \binom{j+2}{2}  - \sum_i \binom{m_i+1}{2} \right \}. 
\]

If the points $P_i$ are general, it is a well-known and difficult 
open problem to classify all $m_i$ and $j$ such that the subscheme
$X$ fails to impose the expected number of conditions on $V$, but a conjectural 
answer is given by the SHGH Conjecture \cite{segre, Ha1, G, Hi}. Segre's
version of the conjecture, which ostensibly gives only
a necessary criterion, is as follows.
\smallskip

\begin{conjecture}[SHGH Conjecture]\label{SHGHConj} For $X = m_1 P_1 + \cdots + m_r P_r$ 
with  general points $P_i$, $X$ fails to impose the expected number 
of conditions on $V$ only if $[I_X]_j\neq 0$ but
the base locus of $[I_X]_j$ contains a non-reduced curve.\smallskip
\end{conjecture}

In fact, the SHGH Conjecture as stated above is equivalent to
versions \cite{Ha1, G, Hi} that not only provide an explicit and complete list of all 
$(m_1,\ldots,m_r)$ and $j$ for which $[I_X]_j$ conjecturally fails to 
impose independent conditions on $V$ but which also conjecturally
determine the extent to which the conditions fail to be independent.
Although we will not discuss the details here, we note that it took 
40 years \cite{refCM} to recognize that the partial characterization 
as given in Conjecture \ref{SHGHConj} above actually provides a full quantitative 
conjectural solution. 

Similarly, our focus here will be on identifying failures
of independence in a generalized context, with a long term
goal of obtaining a more complete characterization. The generalized 
context is that we consider the case that $V$ is a subspace of $R_j$, 
in particular, $V= [I_Z]_j$, where $Z$ is a fat point subscheme. Then 
the overall problem becomes:

\smallskip

\begin{problem}
  \label{prob:intro}
 Characterize and then classify all  triples $(Z,X,j)$
where $Z=c_1 Q_1+\cdots +c_s Q_s$ for distinct points $Q_i$, $X = m_1 P_1 + \cdots m_r P_r$
for general points $P_i$, such that $X$ fails to impose the expected number of conditions on $V = [I_Z]_j$.  
\end{problem}

If $Z$ is the empty set, then $V=[R]_j$, so this is addressed by the SHGH Conjecture. If $Z$ is reduced,  $r = 1$, and $j = m_1 + 1$, this becomes the problem of deciding the existence of an unexpected curve of degree $j$. Our results give the following answer (see Theorem \ref{thm:introequiv}): 

\begin{theorem}
   \label{thm:intro}
Let $Z \subset \PP^2$ be a finite set of points whose dual is a line arrangement with splitting type $(a_Z, b_Z)$.  Let $P$ be a general point. Then the subscheme $X = m P$ fails to impose the expected number of conditions on  $V = [I_Z]_{m+1}$ if and only if 
\begin{itemize}
\item[(i)]  $a_Z \le m \le b_Z-2$; \quad and 

\item[(ii)] $h^1 (\mI_{Z} (t_Z)) = 0$. 
\end{itemize} 
\end{theorem} 

\noindent
Notice that Condition (ii) is equivalent to the assumption $\dim_K [R/I_Z]_{t_Z} = |Z|$.
Theorem \ref{thm:intro} is essentially a restatement of Theorem \ref{mainThm1}, but with
$h^1 (\mI_{Z} (t_Z)) = 0$ replacing $m_Z<t_Z$. The connection between these two conditions is given
in Lemma \ref{a_Z & t_Z} and Theorem \ref{t_Z & unexp}.

Our results give criteria for when a general fat point $mP$ fails to impose the expected number of conditions on $[I_Z]_{m+1}$ for a reduced point scheme $Z$.
It would be interesting to understand exactly for which sets $Z$ 
such failures occur. 
Furthermore, our results strongly suggest that finding  answers to Problem \ref{prob:intro} in other cases is worth investigating. 
\smallskip 

Our paper is organized as follows: In Section \ref{sec: tZ} we introduce unexpected curves and the invariant $t_Z$, and we establish properties of this invariant. Section \ref{sec:criteria} is entirely geared towards establishing our criteria for the existence of unexpected curves. A key ingredient of the argument is shown in Section \ref{sec:proof of 3.7}. The structure of unexpected curves is described in Section \ref{structure}. In Section \ref {sec:examples} we use line arrangements to show that points in linearly general position do not admit unexpected curves and to exhibit configurations of points that do have unexpected curves. 
 The relation of the Lefschetz properties to the existence of unexpected curves and  to Terao's freeness conjecture is described in Section \ref{sec:connections}.

%%%%%%%%%%%%%%%%%%%%%%%%%%%%%%%%%%%%%%%%%%%%%%%%%%%%%%%%%%%%%%

\section{Unexpected curves and the invariant $t_Z$} 
\label{sec: tZ}

In this section we formally define the notion of an unexpected curve. Our main results on when such curves exist will require understanding a certain invariant, which we denote by $t_Z$. Here we also derive the elementary geometric properties of this invariant.

Let $K$ be an arbitrary infinite field (when necessary we will add assumptions) and let $Z=P_1+\cdots+P_d$ be a reduced subscheme of $\PP^2_K$ 
consisting of $d>0$ distinct points $P_i$, with homogeneous ideal $I_Z \subset K[\PP^2] = K[x,y,z] = R$. (In particular, $Z$ will always be nonempty.) For a general point $P$ we denote  by 
$X=Z+jP$ the scheme defined by the ideal $I_X=I_P^j\cap I_Z$. Throughout this paper, ``dimension" refers to the vector space dimension over $K$.
For any $j$ and a fixed $Z$, by semicontinuity there is a Zariski open subset of points $P$ on which 
the dimension of $\dim[I_{Z+jP}]_{j+1}$ takes its minimum value. Thus it makes sense to talk about the 
number of conditions imposed on $[I_{Z}]_{j+1}$ by a general fat point $jP$.

In each degree $t$, note that $\dim_K [I_X]_t \geq \dim_K [I_Z]_t-\binom{j+1}{2}$;
i.e., the forms in $[I_X]_t$ are obtained from those of $[I_Z]_t$ by imposing at most
$\binom{j+1}{2}$ linear conditions coming from $jP$. Typically, if $\dim_K [I_X]_t > \dim_K [I_Z]_t-\binom{j+1}{2}$ 
(i.e., if $jP$ imposes fewer than $\binom{j+1}{2}$ conditions on $[I_Z]_t$)
for a general point $P$, it is because $\dim_K [I_Z]_t<\binom{j+1}{2}$ and $\dim_K [I_X]_t = 0$.
For special choices of $Z$, however, it can happen that $jP$ imposes fewer than
$\binom{j+1}{2}$ conditions even though $P$ is general and $\dim_K [I_X]_t >0$.
We are interested in exploring this situation when the degree $t$ is $j+1$.
This motivates the following definition, where we denote the sheafification
of a homogeneous ideal $I$ by $\mathcal I$. Also, given a sheaf $\mathcal F$ on $\PP^2$, 
we will usually write $h^0(\PP^2, {\mathcal F})$ simply as $h^0({\mathcal F})$.
Thus, for example, $\mathcal I_Z \otimes \mathcal I_P^j=\mathcal I_X = \mathcal I_{Z+jP}$, $  h^0(\PP^2, {\mathcal I}_Z(t))=h^0({\mathcal I}_Z(t))=\dim_K [I_Z]_t$ and 
$h^0(\PP^2,(\mathcal I_Z \otimes \mathcal I_P^j)(t)) = h^0(\mathcal I_{Z+jP}(t)) = h^0(\mathcal I_X(t))=\dim_K [I_X]_t = \dim_K [I_{Z+jP}]_t$. From now on we will suppress the subscript $K$ in the dimension notation.

\begin{definition}
    \label{def:unexpected curve}
We say that a reduced finite set of points $Z \subset \PP^2$  {\em admits an unexpected curve} of degree $j+1$ if there is an integer $j>0$ such that, 
for a general point $P$, $jP$ fails to impose the expected number of conditions on the linear system of curves of 
degree $j+1$ containing $Z$. That is, $Z$ admits an unexpected curve of degree $j+1$ if 
\begin{equation}
  \label{eq:unexpected}
  h^0(\mathcal I_{Z+jP}(j+1))  > \max \left  \{ h^0(\mathcal I_Z(j+1)) - \binom{j+1}{2}, 0 \right \}.
\end{equation}
\end{definition}

\begin{remark}
While it certainly can be of interest to ask when different kinds of non-reduced schemes admit ``unexpected curves" of 
this sort, in this paper we are concerned only with the case where $Z$ is reduced, 
of degree at least two; i.e., $|Z|\geq 2$.
\end{remark}

\begin{remark} \label{interesting}
If $0\leq j\leq1$ and $P$ is general, then $h^0((\mathcal I_Z \otimes \mathcal I_P^j)(j+1)) > 0$ implies
$h^0((\mathcal I_Z \otimes \mathcal I_P^j)(j+1)) = h^0(\mathcal I_Z(j+1)) - \binom{j+1}{2}\geq 0$.
Thus unexpected curves must have degree at least 3. 
\end{remark}

\begin{example}\label{FanoExample}
By Remark \ref{interesting}, the least degree for which an unexpected curve can occur is 3. 
We now reprise an example of Serre (see \cite[Exercise III.10.7]{Hrt}) to show that unexpected curves of degree 3 can occur. 
Although the occurrence of unexpected curves is not purely a characteristic $p>0$ phenomenon
(later we will give examples in characteristic 0), we believe that 2 is the only characteristic for which
an unexpected curve of degree 3 can occur. In any case, for this example
assume $K$ has characteristic 2 and take $Z$ to be the seven points whose homogeneous coordinates 
$[a : b : c]$ consist of just zeros and ones. We now show that \eqref{eq:unexpected} holds with $j + 1 = 3$ and with the right hand side of \eqref{eq:unexpected} being 0.
Note that the seven points are the points of the Fano plane and that any line through two of them goes through a third.
There are only seven such lines, and they are projectively dual to the seven points.
Let $P=[\alpha:\beta:\gamma]\in \PP^2$ be a general point. One can check that
$Z$ imposes independent conditions on cubics (in fact, 
$I_Z=(yz(y + z),xz(x + z),xy(x + y))$). Since $Z+2P$ imposes 10 conditions,
one would expect that there would not be a cubic containing $Z$ having a double point at $P$.
But the conditions are not independent: one can easily check
that $F=\alpha^2yz(y + z)+\beta^2xz(x + z)+\gamma^2xy(x + y)$ defines a curve $C$ 
(reduced and irreducible in fact) which
is singular at $P$ and hence $C$ is an unexpected curve of degree 3 for $Z$.
\end{example}

Note for $j\geq 0$ that it is always true that 
\begin{equation} \label{always}
\begin{array}{rcl}
\dim [I_{Z+jP}]_{j+1} & \geq & \dim [I_Z]_{j+1} - \binom{j+1}{2} \\ \\
& = & \binom{j+3}{2} - h_Z(j+1) - \binom{j+1}{2} \\  \\
& \geq & \binom{j+3}{2} - |Z| - \binom{j+1}{2} 
\end{array}
\end{equation}
and
%\begin{equation} \label{useful}
%\binom{j+3}{2}-\binom{j+1}{2}-|Z|=(j-a_Z+1)+(j-b_Z+1),
%\end{equation}
%the latter because  $a_Z+b_Z = d-1$. It is also easy to see that
\begin{equation} \label{useful2}
\dim [I_Z]_{j+1} - \binom{j+1}{2} = 2j+3 - h_Z(j+1),
\end{equation}
where $h_Z (j) = \dim [R/I_Z]_j = \binom{j+2}{2} - \dim [I_Z]_j$ is the {\em Hilbert function} of $Z$.

The definition of an unexpected curve already suggests the importance of the following invariant.

\begin{definition}
      \label{def: tZ}
We define $t_Z$ to be the least $j$ such that $\dim [I_Z]_{j+1}>\binom{j+1}{2}$. 
\end{definition}

\begin{remark}
One sees immediately that $t_Z$ depends only on the Hilbert function of $Z$. However, the existence of an unexpected curve does not depend only on the Hilbert function, as one can see from easy examples.
\end{remark}

\begin{lemma} \label{bound for t}
\begin{itemize}

\item[(a)] $\displaystyle 0 \leq  t_Z \leq \left \lfloor \frac{|Z|-1}{2} \right \rfloor$

\item[(b)] $\displaystyle t_Z = \left \lfloor \frac{|Z|-1}{2} \right \rfloor$ if and only if $\displaystyle
\left \{
\begin{array}{ll}
h_Z(t_Z) = |Z| & \hbox{if } |Z| \hbox{ is odd} \\
h_Z(t_Z) \geq |Z|-1 & \hbox{if $|Z|$ is even}
\end{array}
\right.
$

\end{itemize}
\end{lemma}

\begin{proof}
The fact that $0 \leq t_Z$ follows from the definition. From (\ref{useful2}) we have
\[
\dim [I_Z]_{j+1} - \binom{j+1}{2} = 2j+3 - h_Z(j+1) \geq 2j+3-|Z|
\]
with equality if and only if $h_Z(j+1) = |Z|$. Thus
\[
\begin{array}{rcl}
t_Z & = & \min \{ j \ | \ 2j+3 - h_Z(j+1) > 0 \} \\ \\
& \leq & \min \{ j \ | \ 2j+3 - |Z| > 0 \} \\ \\
& = & \min \{ j \ | \ j \geq \frac{|Z|-2}{2} \} \\ \\
& = &  \lfloor \frac{|Z|-1}{2}  \rfloor.
\end{array}
\]
For (b), suppose first that $t_Z = \lfloor \frac{|Z|-1}{2} \rfloor$. Since $2(t_Z-1) +3-h_Z(t_Z) \leq 0$, we have 
\[
2  \left \lfloor \frac{|Z|-1}{2} \right \rfloor +1 - h_Z(t_Z) \leq 0.
\]
Recalling that $h_Z(j) \leq |Z|$ for all $j$, this gives:

\medskip

\begin{itemize}

\item If $|Z|$ is odd and $\displaystyle  t_Z = \left \lfloor \frac{|Z|-1}{2} \right \rfloor \hbox{ then } h_Z(t_Z) = |Z|.$ 

\medskip

\item If $|Z|$ is even and $\displaystyle  t_Z =  \left \lfloor \frac{|Z|-1}{2} \right \rfloor \hbox{ then } h_Z(t_Z) \geq |Z| -1 $.

\end{itemize}

\noindent For the converse, assume that the parity condition holds. 
In both cases, $h_Z(t_Z+1) = |Z|$ since $h_Z$ is strictly increasing until it reaches the value $|Z|$. So the only inequality in the calculation in (a) is an equality, and we are done.
\end{proof}

\begin{remark}
If $|Z|$ is even, 
both $h_Z(t_Z) = |Z|$ and $h_Z(t_Z) = |Z|-1$ are possible. For example, take $|Z|=6$ and choose $Z$ to be a set of 6 
general points versus a set of 6 points on a smooth conic. 
In both cases $t_Z = 2$, but $h_Z(t_Z) = |Z|$ for 6 general points, while $h_Z(t_Z) = |Z|-1$ for 6 points on the conic.
\end{remark}

\begin{example}
   \label{rem:tZ}
Here we evaluate $t_Z$ exactly when $Z$ lies on a curve of low degree.
\begin{enumerate}
\item[(i)] The definition immediately gives that $t_Z = 0$ if and only if the points of $Z$ are collinear,
so in this case $t_Z$ is as small as possible.
\item[(ii)] If $Z$ lies on an irreducible conic, then  it is not hard to check that $t_Z = \left \lfloor \frac{|Z| -1}{2} \right \rfloor$,
so in this case $t_Z$ is as large as possible. 
%Moreover, we again have $m_Z=t_Z$. (By Bezout's Theorem,
%any form of degree $t_Z$ vanishing on $Z$ must be divisible by the form defining the conic. But the quotient
%then has degree $t_Z-2$, which therefore cannot vanish to order $t_Z-1$ at a general point $P$. Thus $m_Z>t_Z-1$.)
\end{enumerate}
\end{example}

\begin{proposition}\label{prop:small tZ}
Let $Z \subset \PP^2$ be a reduced scheme consisting of a finite set of points.
Then the following conditions are equivalent: 
\begin{itemize}
	\item[(a)] $h_Z (t_Z) < |Z|$; 
	\item[(b)] 
\begin{itemize}
	\item[(i)] the scheme $Z$ is a complete intersection cut out by two curves
meeting transversely, of degree $2$ and $t_Z + 1$ respectively, with $t_Z>0$;  or 
	\item[(ii)] there is a line that contains precisely  $|Z| - t_Z \ge t_Z + 2$ points of $Z$. 
\end{itemize}
\end{itemize}
Furthermore, in case (b)(i) we have $t_Z  = \frac{|Z| -2}{2}$,
while for case (b)(ii) we have $t_Z \le \frac{|Z| -2}{2}$.
\end{proposition} 

\begin{proof}
To simplify notation, put $t = t_Z$. 
We  use $\Delta h_Z$ to denote the first difference of the Hilbert function of $Z$; that is, 
$\Delta h_Z (j) = h_Z (j) - h_Z (j-1)$. 

First assume $t =0$. By Example \ref{rem:tZ}, the points of $Z$ are collinear, so
(a) holds if and only if $|Z|>1$, and (b) holds if and only if $|Z|\geq 2$, so (a) and (b)
are equivalent, and clearly $0=t \le \frac{|Z| -2}{2}$ for $|Z|>1$. Thus it is now enough to consider the case that 
$t \ge 1$, that is, that $Z$ is not collinear. 

Assume (a) holds. By the definition of $t_Z$, equation (\ref{useful2}) and the fact that $h_Z$ is strictly increasing
until it stabilizes at the value $|Z|$, this forces 
\[
2 t + 1 \le h_Z (t) < h_Z (t + 1) \le 2 t + 2,  
\]
and thus $h_Z (t + 1) = 2 t + 2 = 1 + h_Z (t)$. In particular, $\Delta h_Z (t + 1) = 1$. 
By standard results (see, for example, \cite[Proposition 3.9]{DGM}), this 
implies that the values of $\Delta h_Z$ are as follows (where $s$ is the regularity of $I_Z$): 
\begin{equation} 
   \label{eq:difference Hilb}
\begin{array}{ccccccccc}
j                      & :  & 0 & 1 & \ldots & t+1 & \ldots & s-1 & s\\
\Delta h_Z(j) & :  & 1 & 2 & \ldots & 1    & \ldots & 1 & 0
\end{array}
\end{equation}
Thus, $h_Z (t + 1) = 2 t + 2$ implies
\[
|Z| = 2 t + 2 + (s - t - 2) = s + t. 
\]
Using $h_Z (t + 1) = 2 t + 2 \le |Z|$, we conclude that $s = |Z| - t \ge t + 2$.

Now we consider two cases: 

\emph{Case 1}: Assume $Z$ does not lie on a  conic, that is, $\Delta h_Z (2) = 3$. Hence 
\[
2 t + 2 =  h_Z (t + 1) = \sum_{j=0}^{t+1} \Delta h_Z (j)
\]
forces $\Delta h_Z (t) = 1$, and thus $\Delta h_Z (s-2) = \Delta h_Z (s-1) = 1 > \Delta h_Z (s)$. 
By \cite[(2.3)]{Da} (or by applying results of \cite{BGM}), it follows that $[I_Z]_{s-1}$ has a linear form $\ell$ 
as a common divisor. Since $I_Z$ has a minimal set of homogeneous generators all of whose degrees are
at most $s$, there must be a generator $f$ of degree $s$ and by
\cite[Theorem 2.1]{Camp} there is only one generator of degree $s$ in a minimal set of homogeneous generators.
Moreover, since $Z$ is reduced, the curves defined by $f$ and $\ell$ must intersect transversely.
Thus the ideal $(\ell, f)$ defines a subset of $s$ collinear points of $Z$ and clearly $\ell$ vanishes at no point of $Z$
other than these $s$. Therefore, condition (ii) is satisfied. 

\emph{Case 2}: Assume $Z$ is contained in a conic, defined, say, by a homogeneous form $q$. 
Again taking into account $h_Z (t + 1) = 2t + 2$, we get 
\begin{equation}
  \label{eq:h-vector}
\Delta h_Z (j) = 
\begin{cases}
1 & \text{if $j=0$ or $t+1 \le j < s$} \\
2 & \text{if } 1 \le j \le t \\
0 & \text{otherwise.}
\end{cases}
\end{equation}
If $t=1$, then $\Delta h_Z=(1,2,1,\ldots,1)$. Thus $Z$ is either 4 general points
(i.e., a complete intersection) or $Z$ consists of 3 or more collinear points
and one point off the line; in both cases it is easy to check that the assertions hold.
Now assume that $t>1$.
It follows that $q$ is a common factor for $[I_Z]_j$ for $j\leq t$, but not for $j=t+1$,
so any minimal set of homogeneous generators for $I_Z$ must contain $q$ and a generator
$g$ of degree $t+1$. If $q$ and $g$ are coprime, then $|Z|\leq \deg(q)\deg(g)=2t+2$.
Since $|Z| = t+s$ and $s\geq t+2$, this means $s=t+2$ and $|Z|=2t+2$, so
$Z$ is a complete intersection as claimed in (i).
Otherwise, $q$ and $g$ have a linear common factor $\ell$ and $I_Z$ 
has another minimal generator $f$ of degree $s$. As in Case 1 we conclude that 
the line defined by $\ell$ contains precisely $s$ points of $Z$, and so condition (ii) is met. 
\smallskip

Conversely, assume one of the conditions in  (b) is true. Thus, $|Z| - t \ge t +2$
(and hence $t \le \frac{|Z| - 2}{2}$), by hypothesis for part (ii) and using the fact
that $Z$ is a transverse complete intersection of a conic with a curve of degree $t+1$
for part (i). Again, we consider two cases: 

If (i) is true, then  $h_Z (t)  = 2 t + 1 < 2 t + 2 = |Z|$, as desired. Moreover, here we have $t = \frac{|Z| - 2}{2}$.

Finally, assume (ii) is true, let $Y \subset Z$ be a subset of $|Z|-t$ collinear points
and let $U$ be the complement of $Y$ in $Z$. Then $t=|U|$ and $U$ is reduced, so
$U$ imposes independent conditions on forms of degree $t-1$; i.e.,
$h_U(t-1)=t$ and thus $\dim [I_U]_{t-1}=\binom{t+1}{2}-t$. But the linear form $\ell$ 
vanishing on $Y$ is, by Bezout's Theorem,
a common divisor of $[I_Z]_t$, so $\dim [I_Z]_t=\dim [I_U]_{t-1}$, and we have $h_Z(t)=2t+1<2t+2 \leq |Z|$. 
\end{proof} 

As a consequence we show that adding a point to $Z$ will change the invariant $t_Z$ by at most one. 

\begin{corollary}
     \label{cor:change of t_Z}
Let $Z \subset \PP^2$ be a finite  reduced scheme. If $Q \notin Z$ is any other point of $\PP^2$, then 
\[
t_Z \le t_{Z+Q} \le t_Z + 1. 
\] 
\end{corollary}

\begin{proof} 
By definition, we clearly have $t_Z \le t_{Z+Q}$. It remains to show the second inequality. Suppose $t_{Z+Q} \ge t_Z + 2$. Then the definition gives $h_Z (t_Z + 1) \le 2 t_Z + 2$ and 
$h_{Z+Q} (t_Z + 1) \ge 2 t_Z + 3$. Since the Hilbert functions of $Z$ and $Z + Q$ differ at most by 
one in each degree, we conclude that 
\[
h_{Z+Q} (j) = h_Z (j) + 1 \quad \text{whenever } j > t_Z,  
\]
and, in particular, $h_Z (t_Z + 1) = 2 t_Z + 2$.
Considering degree $t_Z + 2 \le t_{Z+Q}$, we get \linebreak $h_{Z+Q} (t_Z + 2) \ge 2 t_{Z} + 5$, which implies 
\[
|Z| \ge h_Z (t_Z + 2) \ge 2 t_Z + 4 = h_Z (t_Z + 1) +2. 
\]
It follows that $h_Z (t_Z) < |Z|$ and $t_Z \le \frac{|Z|-4}{2}$. Hence, 
Proposition \ref{prop:small tZ}  shows that $|Z| - t_Z$ of the points in $Z$ are 
collinear. Denote by $Y$ this subset of $Z$, and so $|Y| = |Z| - t_Z \ge t_Z + 4$. Now, using that 
the points in $Y$ are collinear, we obtain 
\[
h_{Z+Q} (t_Z +2) \le h_Y (t_Z+2) + |Z+Q-Y| = t_Z + 3 + t_Z + 1 = 2 t_Z + 4,
\]
contradicting our estimate above that $h_{Z+Q} (t_Z + 2) \ge 2 t_{Z} + 5$.
%This implies $t_{Z+Q} \le t_Z + 1$, a contradiction to our assumption. 
\end{proof}

%%%%%%%%%%%%%%%%%%%%%%%%%%%%%%%%%%%%%%%%%%%%%%%%%%%%%%%%%%%%%%

\section{Line arrangements and a criterion for unexpected curves}
\label{sec:criteria}

The following are the additional invariants that we will need.

\begin{definition}
      \label{def:mult speciality index}
Let $Z$ be a reduced 0-dimensional subscheme of $\PP^2$. 

\begin{itemize}

\item[(a)]{\cite[Definition 4.1]{FV2}} Given a point $P \notin Z$, we call 
\[
m_{Z, P} = \min\{ j \ge 0 \ |  \dim [I_{Z+jP}]_{j+1} > 0 \}
\]    
the \emph{multiplicity index of $Z$ with respect to $P$}. 
We define the {\em multiplicity index}, $m_Z$, to be 
\[
m_Z = \min \{ j \in \mathbb Z \ |  \dim [I_{Z+jP}]_{j+1} > 0\}
\]
for a general point $P$.

\item[(b)] Let $P \in \mathbb P^2$ be a general point. We define the {\em speciality index}, $u_Z$, to be the least $j$ such that $Z + jP$ imposes independent conditions on plane curves of degree $j+1$, i.e. the least $j$ such that 
\[
\dim [I_{Z+jP}]_{j+1}=\binom{j+3}{2}-\binom{j+1}{2}-|Z|.
\]
\end{itemize}
\end{definition}

\begin{remark} 
     \label{rem: u_z}
\begin{itemize}

\item[(i)] We note that $m_{Z,P}$ exists for each point $P\notin Z$, since
it is easy to see that  $\dim [I_{Z+jP}]_{j+1}  > 0$ holds
for all $j\geq |Z|$ (pick $j$ lines through $P$ which also 
go through the $|Z|$ points of $Z$), and hence $m_{Z,P}\leq |Z|$, so also $m_Z \leq |Z|$.
We also note that $\dim [I_{Z+jP}]_{j+1}$ is a nondecreasing function of $j$,
since we have an injection $[I_Z \cap I_P^j]_{j+1}\to [I_Z \cap I_P^{j+1}]_{j+2}$
given by multiplication by any linear form $\ell$ vanishing at $P$.
Thus if $\dim [I_{Z+jP}]_{j+1} =0$ then $m_{Z,P}>j$, hence $m_Z>j$ by semicontinuity.

\item[(ii)] Observe that  $u_Z$ can equivalently be defined as
\[
u_Z = \min \{ j \ | \ h^1(\mathbb P^2, \mathcal I_{Z+jP}(j+1)) = 0 \}.
\] 

\item[(iii)] We also note that $u_Z$ exists, and in fact 
$u_Z \leq |Z|-2$. Indeed, if  $Z$ is a set of $d$ points in $\PP^2$, and  $P \in \PP^2$ is a point that is not on any line 
through two of the points of $Z$, then we will show that 
$h^1(\PP^2,\mathcal I_{Z+(d-2)P}(d-1)) = 0$. 

To see this, we have to show that $Z + (d-2)P$ imposes $d + \binom{d-1}{2}$ conditions to the linear system of plane curves of degree $d-1$. Clearly $(d-2)P$ imposes $\binom{d-1}{2}$ conditions (since the regularity of $(d-2)P$ is $d-1$), so we want to show that the points of $Z$ impose $d$ independent conditions on the linear system, $\mathcal L$, of plane curves of degree $d-1$ vanishing to order $d-2$ at $P$. It is enough to show that given any point $Q$ of $Z$ there is a curve of degree $d-1$ vanishing to order $d-2$ at the general point $P$ and vanishing at each point of $Z \backslash \{Q\}$, but not vanishing at $Q$. This can be done (for instance) with a suitable union of $d-1$ lines, each joining $P$ and a point of $Z \backslash \{Q\}$.

\end{itemize}
\end{remark}

Next we bring in an important tool derived from a result of Faenzi and Vall\`es.
We continue with the assumption that $Z=P_1+\cdots+P_d$ is a reduced subscheme of $\PP^2_K$ consisting of distinct points $P_i$.
Let $\ell_i$ be the corresponding linear form dual to $P_i$ and $L_i$ the line defined by $\ell_i$, and define $f$  to be the product $f =\ell_1\cdots\ell_d$ (so $f$ is square free). We denote by $\mathcal A (f)$, or simply $\mathcal A$, the line arrangement in $\PP^2$ defined by $f$. In most cases we will not need to use different sets of variables for $Z$ and for $f$.

Note that when $\operatorname{char}(K)$ does not divide $d = \deg(f)$, then
$xf_x+yf_y+zf_z = d f$ is a non-zero scalar multiple of $f$. In the case when $\operatorname{char}(K)$ does divide $\deg(f)$, Euler's theorem gives $xf_x+yf_y+zf_z = 0$, i.e. a degree one syzygy on $f_x,f_y,f_z$. In this case it is not necessarily true that $f $ is in the 
ideal $\Jac(f)=(f_x, f_y, f_z)$ generated by its first partial derivatives, although it can happen. For instance, let
\[
F=xyz(x+y)=(x^2y+xy^2)z \ \hbox{ with } \  \hbox{char}(K) = 2.
\] 
Then $F$ is in $\Jac(F)=(y^2z, x^2z, x^2y+xy^2)$ since $F$ is $z$ times $(x^2y+xy^2)$. In fact, whenever there is only one factor with a $z$ in it, and that factor is $z$, we get this. In this situation it follows that $Z$ consists of all but one  of the points on a line. (We do not know if this, up to change of variables, is the only situation in which this behavior can happen.)

Let $J' =  \Jac(f) = (f_x,f_y,f_z)$. Let $J = (J',f)$. If $\operatorname{char}(K)$ divides $d$, we have seen that it may or may not happen that $J = J'$, and in any case $J'$ has a degree 1 syzygy coming from the Euler relation that does not occur when $\operatorname{char}(K)$ does not divide $d$. Nevertheless, it turns out that the issue of whether or not $\operatorname{char}(K)$ divides $d$ is less crucial than these considerations might lead one to expect, and in fact until section \ref{structure} we will make no assumption on the characteristic. The justification of this omission, and the role of the characteristic, seems to be known at least to the experts, but since we are not aware of a detailed reference in the literature, we include it as an appendix.

Define the submodule 
$D(Z)\subset R\frac{\partial}{\partial x}\oplus R\frac{\partial}{\partial y}\oplus R\frac{\partial}{\partial z}\cong R^3$ 
to be the $K$-linear derivations $\delta$ such that $\delta(f)\in Rf$.
In particular,  $D(Z)$ contains the Euler derivation 
$\delta_E =x\frac{\partial}{\partial x}+y\frac{\partial}{\partial y}+z\frac{\partial}{\partial z}$,
and $\delta_E$ generates a submodule $R\delta_E \cong R(-1)$. 
We can now define the quotient $D_0(Z)=D(Z)/R\delta_E$. Let $\mathcal D_Z$ be the sheafification of $D_0(Z)$, which we call the {\em derivation bundle} of $Z$. 

The following facts are shown in the appendix and will be used freely throughout this paper.

\begin{itemize}
\item $\mathcal D_Z$ is locally free of rank 2.

\item When $\operatorname{char}(K)$ does not divide $d$, $\mathcal D_Z$ is isomorphic to the syzygy bundle (suitably twisted) of $J'$.

\item In any case  $D(Z)$ is isomorphic to the syzygy module of $J$.

\item The restriction of $\mathcal D_Z$ to a general line splits as a direct sum $\mathcal O_{\PP^1}(-a_Z) \oplus \mathcal O_{\PP^1}(-b_Z)$ for positive integers $a_Z, b_Z$ satisfying $a_Z + b_Z = |Z|-1 = d-1$. We call the ordered pair $(a_Z, b_Z)$, with $a_Z \leq b_Z$, the {\em splitting type} of $Z$.
\end{itemize}

\begin{lemma}\label{FVlemma} 
Let $Z$ be a reduced 0-dimensional subscheme of $\PP^2$ and let $P$ be a general point.
Then one has, for each integer $j$,
$$
\dim [I_{Z+jP}]_{j+1} = \max \{0,j-a_Z+1 \}+\max \{0,j-b_Z+1\}.
$$
\end{lemma}

\begin{proof}
Let $q:Y\to \PP^2$ be the blow up of $\PP^2$ at the point $P$.
Let $H$ be the pullback of a line and $E=q^{-1}(P)$ the exceptional curve
coming from $P$. The proper transforms of the lines through $P$
gives the linear system $|H-E|$, which gives a morphism $p:Y\to L=\PP^1$
with fibers the elements of $|H-E|$, making $Y$ a $\PP^1$-bundle over 
$L$. (In fact, $Y$ is just the Hirzebruch surface $H_1$.)

Let $\mI_Z$ be the sheaf of ideals of $Z$ on $\PP^2$. 
Since $P$ is general (and hence the $t_{Z,y}$ appearing in \cite[Theorem 4.3]{FV2} is 0), 
we have the isomorphism
$p_*(q^*(\mI_Z(1)))\cong \mO_L(-a_Z)\oplus\mO_L(-b_Z)$
from  \cite[Theorem 4.3]{FV2}.
Since $P\not\in Z$, $q$ is an isomorphism on an open set containing $Z$,
so we can regard $Z$ as being on $\PP^2$ or on $Y$,
hence there is a natural identification of $\mI_Z$ with $q^*(\mI_Z)$. Under this identification
we can regard $q^*(\mI_Z(1))$ as being the sheaf $\mI_Z\otimes\mO_Y(H)$.
Thus we have $p_*(\mI_Z\otimes\mO_Y(H))\cong \mO_L(-a_Z)\oplus\mO_L(-b_Z)$.
Now tensor through by $\mO_L(j)$ to get 
\begin{align*}
p_*(\mI_Z\otimes\mO_Y((j+1)H-jE)) & \cong p_*(\mI_Z\otimes\mO_Y(H)\otimes p^*(\mO_L(j))) \cong p_*(\mI_Z\otimes\mO_Y(H))\otimes \mO_L(j)\\
 & \cong \mO_L(j-a_Z)\oplus\mO_L(j-b_Z).
\end{align*}
Since $p_*$ preserves global sections, taking global sections gives
\begin{align*}
[I_{Z+jP}]_{j+1}& \cong\Gamma(\PP^2, \mI_{Z+jP}\otimes\mO_{\PP^2}((j+1)H))\\ 
 & \cong\Gamma(\PP^2, \mI_Z\otimes \mI_{jP}\otimes\mO_{\PP^2}((j+1)H))\\ 
 & \cong\Gamma(Y, \mI_Z\otimes\mO_Y((j+1)H-jE)))\\ 
 & \cong \Gamma(L,p_*(\mI_Z\otimes\mO_Y((j+1)H-jE)))\\
 & \cong\Gamma(L,\mO_L(j-a_Z)\oplus\mO_L(j-b_Z))\\
 & \cong\Gamma(L,\mO_L(j-a_Z))\oplus \Gamma(L,\mO_L(j-b_Z))
\end{align*}
The result now follows by taking dimensions.
\end{proof}

\begin{remark} \label{1 or 2}

\begin{itemize}

\item[(i)]
From Lemma \ref{FVlemma} it follows immediately that $\dim [I_{Z+a_Z P}]_{a_Z+1}$ is either equal to 1 or to 2, and the latter holds if and only if $a_Z = b_Z$. 

\item[(ii)] It also follows immediately from Lemma \ref{FVlemma} and the fact that $a_Z + b_Z + 1 = |Z|$ that if $Z$ is a set of points with splitting type $(a_Z,b_Z)$ and $Q$ is a general point then $Z \cup Q$ has splitting type $(a_Z+1,b_Z)$ (noting that if $a_Z = b_Z$ then this should be written $(a_Z,b_Z+1)$ to preserve the proper inequality). 

\end{itemize}
\end{remark}

We record some immediate consequences of these observations.

\begin{lemma}
          \label{m, t and u}
Let $Z$ be a reduced set of points in $\mathbb P^2$.

\begin{itemize}
\item[(a)] $m_Z = a_Z$. 

\item[(b)]  $m_Z = 0$ if and only if the points of $Z$ lie on a line. 

\item[(c)]  $u_Z = b_Z-1$ (hence $m_Z-1\leq u_Z=|Z|-m_Z-2$).

\item[(d)] $\displaystyle m_Z\leq t_Z \leq \left  \lfloor \frac{|Z|-1}{2} \right \rfloor$.

\item[(e)] If $m_Z<t_Z$, then $t_Z\leq u_Z$.

\end{itemize}
\end{lemma}

\begin{proof}
Part (a) follows immediately from Lemma \ref{FVlemma}, 
while (b) follows from the definition of $m_Z$. 
For (c), note that for any $j$,
\begin{equation} \label{useful}
\binom{j+3}{2} - \binom{j+1}{2} - |Z| = 2j+3-|Z| = (j-a_Z+1) + (j-b_Z+1),
\end{equation}
the latter since $a_Z + b_Z = |Z|-1$. Because $a_Z \leq b_Z$, the result follows from the definition of $u_Z$, Lemma \ref{FVlemma} and $a_Z+b_Z=|Z|-1$.
Part (d) comes from Lemma \ref{bound for t} and the definitions.

For (e), assume that $m_Z = a_Z < t_Z$. It is enough to prove that $a_Z \leq j < t_Z$ implies $j < u_Z$. But $a_Z \leq j < t_Z$ implies that $Z + jP$ does not impose independent conditions on forms of degree $j+1$, so $j < u_Z$ and we are done.
\end{proof}

\begin{lemma}\label{a_Z & t_Z}
Let $Z$ be a reduced 0-dimensional subscheme of $\PP^2$  and let $h_Z$ be its Hilbert function.
If $h_Z(t_Z)=|Z|$ and $m_Z<u_Z$, then $m_Z<t_Z$.
\end{lemma}

\begin{proof}
From Lemma \ref{m, t and u} we have $a_Z\leq t_Z$ and $u_Z = b_Z-1$.   Since $a_Z <  b_Z-1$, 
applying Lemma \ref{FVlemma} and (\ref{useful}) with $j=a_Z$ we get 
\[
1=\dim [I_{Z+a_ZP}]_{a_Z+1} > \binom{a_Z+3}{2}-|Z|-\binom{a_Z+1}{2}.
\]
Now suppose that $a_Z = t_Z$. 
Since $h_Z(t_Z)=|Z|$, the points of $Z$ impose independent conditions 
on curves of degree $t_Z = a_Z$ and hence also on curves of degree $a_Z+1$,
$\binom{a_Z+3}{2}-|Z|-\binom{a_Z+1}{2}=\dim [I_Z]_{a_Z+1} - \binom{a_Z+1}{2}>0$ (by definition of $t_Z$). Combining with the previous inequality, we obtain  
\[
1 >\binom{a_Z+3}{2}-|Z|-\binom{a_Z+1}{2} > 0,
\]
which is impossible since the middle expression  is an integer. Thus $m_Z = a_Z<t_Z$.
\end{proof}

The definition of unexpected curves implies already that if $Z$ admits an unexpected curve of degree $j+1$ then $Z+jP$ fails to impose independent conditions on plane curves of degree $j+1$. We will see that the converse is false. The following result is critical for our main theorems, but the proof is rather involved,
so we put off addressing it until the next section where we prove a stronger result
of which Theorem \ref{t_Z & unexp} is an immediate consequence.

\begin{theorem}\label{t_Z & unexp}
Let $Z$ be a reduced 0-dimensional subscheme of $\PP^2$.
If $h_Z(t_Z)<|Z|$, then $Z$ admits no unexpected curves.
\end{theorem}

\begin{proof}
See Theorem \ref{thm:tZ = dZ}.
\end{proof}

\begin{remark}
Definition \ref{def:unexpected curve} leaves open the possibility that the points of $Z$ themselves do not impose independent conditions on curves of some degree $j+1$, and moreover the addition of the general fat point $jP$ fails to impose the expected number of conditions on the linear system defined by $[I_Z]_{j+1}$ (i.e. there is still an unexpected curve). Theorem \ref{t_Z & unexp} gives the surprising result that this is impossible.
\end{remark}

The following result restates Theorem \ref{mainThm1}.

\begin{theorem}
     \label{u_ZTheorem} 
Let $Z$ be a reduced 0-dimensional subscheme of $\PP^2$.
Then $Z$ admits an unexpected curve if and only if $m_Z < t_Z$. 
Furthermore, in this case $Z$ has an unexpected curve of degree $j+1$ if and only if $m_Z \leq  j < u_Z $. 
\end{theorem}

\begin{proof}
Assume $Z$ admits an unexpected curve. Then for a general point $P$ there is an $m\geq a_Z = m_Z$ such that
\[
\dim [I_{Z+mP}]_{m+1} > \max \left \{0, \ \dim [I_Z]_{m+1} - \binom{m+1}{2} \right \}.
\]
By Theorem \ref{t_Z & unexp} we have $h_Z(t_Z)=|Z|$. 
We now claim that $m_Z < u_Z$. Indeed, if it were true that $u_Z \leq m_Z \leq m$ then 
{\small \[
\begin{array}{rclll}
\displaystyle \max \left \{0, \ \dim [I_Z]_{m+1} - \binom{m+1}{2} \right \} & \geq & \displaystyle  \binom{m+3}{2}-\binom{m+1}{2}-|Z| & & \hbox{(by (\ref{always}))} \\ 
& = & \displaystyle  \dim [I_{Z+mP}]_{m+1} && \displaystyle  \hbox{(since $m \geq u_Z$)} \\ 
& > & \displaystyle  \max \left \{0, \ \dim [I_Z]_{m+1} - \binom{m+1}{2} \right \} && \hbox{(by choice of $m$).}
\end{array}
\]}
Lemma \ref{a_Z & t_Z} now implies $m_Z<t_Z$.

Conversely, if $m_Z<t_Z$, then 
\[
\dim [I_{Z+m_ZP}]_{m_Z+1} > 0 = \max \left \{0, \ \dim [I_Z]_{m_Z+1} - \binom{m_Z+1}{2} \right \},
\]
and so $Z$ admits an unexpected curve of degree $m_Z+1$.

For the rest, assume $m_Z<t_Z$ and $h_Z(t_Z)=|Z|$. Then $t_Z\leq u_Z$ by Lemma \ref{m, t and u} (e). If $m_Z\leq j<t_Z$, we have $\dim [I_{Z+jP}]_{j+1} > 0 = \max \left \{0, \ \dim [I_Z]_{j+1} - \binom{j+1}{2} \right \}$, and 
so there are unexpected curves for each such $j$.

Now assume that $t_Z\leq j<u_Z = b_Z-1$ (Lemma \ref{m, t and u} (c)). Since $h_Z(t_Z) = |Z|$, we know that $Z$ imposes independent conditions on curves of degree $j$. Then 
using Lemma~\ref{FVlemma},
we have 
\begin{align*}
\dim [I_{Z+jP}]_{j+1} &=j-m_Z+1 \\
& > (j+1-a_Z)+(j-b_Z +1) \\
&= \binom{j+3}{2}-|Z|-\binom{j+1}{2}\\
&=\dim [I_Z]_{j+1} - \binom{j+1}{2}\\
&=\max \left \{0, \ \dim [I_Z]_{j+1} - \binom{j+1}{2} \right \},
\end{align*}
hence there are unexpected curves for each such $j$. 

Finally, if $j\geq u_Z$, we have 
\[
\dim [I_{Z+jP}]_{j+1}=\binom{j+3}{2}-|Z|-\binom{j+1}{2}=\max \left \{0, \ \dim [I_Z]_{j+1} - \binom{j+1}{2} \right \},
\]
so there are no unexpected curves of any such degree $j+1$.
\end{proof}

Later we show that unexpected curves of degree greater than $m_Z +1$ are always reducible (see Corollary \ref{cor:unexpected curves only in initial deg}). 

An alternative characterization of the occurrence of unexpected curves is given by the 
following theorem (which is equivalent to Theorem \ref{thm:intro}).

\begin{theorem}
   \label{thm:introequiv}
Let $Z \subset \PP^2$ be a finite set of points whose dual is a line arrangement with splitting type $(a_Z, b_Z)$.  
Let $P$ be a general point. Then the subscheme $X = m P$ fails to impose the expected number of 
conditions on  $V = [I_Z]_{m+1}$ if and only if 
\begin{itemize}
\item[(i)]  $a_Z \le m < b_Z-1$; \quad and 

\item[(ii)] $h_Z(t_Z) = |Z|$. 
\end{itemize} 
\end{theorem} 

\begin{proof}
Assume $X$ fails to impose the expected number of conditions on $V = [I_Z]_{m+1}$;
i.e., $Z$ has an unexpected curve of degree $m+1$. Then $a_Z \le m < u_Z$ by Theorem \ref{u_ZTheorem}
and $h_Z(t_Z) = |Z|$ by Theorem \ref{t_Z & unexp}.

Conversely, if $a_Z<u_Z$ and $h_Z(t_Z)=|Z|$, then $a_Z<t_Z$ by Lemma \ref{a_Z & t_Z}, and hence
by Theorem~\ref{u_ZTheorem} there are unexpected curves in degrees $m+1$ for each
$m$ in the range $a_Z \le m < u_Z$.
\end{proof}

\begin{remark}
By Proposition \ref{prop:small tZ}, condition (ii) of the previous theorem imposes a very weak restriction.
\end{remark}

%%%%%%%%%%%%%%%%%%%%%%%%%%%%%%%%%%%%%%%%%%%%%%%%%%%%%%

\section{The proof of Theorem \ref{t_Z & unexp}}
\label{sec:proof of 3.7}

\begin{lemma}
  \label{lem:rewrite rhs} For each integer $j \ge 0$ we have
\[
h^1(\mathcal I_{Z+j P} (j +1)) = h^0(\mathcal I_{Z+j P} (j+1)) + |Z| - (2j +3).
\]
\end{lemma}

\begin{proof}
This follows from the exact sequence
\[
\begin{array}{ll}
0 \rightarrow H^0(\mathcal I_{Z+j P} (j+1)) \rightarrow H^0(\mathcal O_{\mathbb P^2} (j +1)) \rightarrow H^0(\mathcal O_{Z + j P}(j+1)) \\ \\
\hspace{4in} \rightarrow H^1 (\mathcal I_{Z + j P}(j+1)) \rightarrow 0.
 \end{array}
\]
\end{proof}

\begin{lemma} \label{lem:initial degree}
Let $Z  \subset \PP^2$ be a reduced scheme  consisting of a finite set of points. Then, for each general point $P \in \PP^2$,    
\[
\dim [ I_{Z+ jP} ]_j  > 0 \text{ if and only if } j \ge |Z|. 
\]  
In this case we have, furthermore, 
\[
\dim [I_{Z+jP}]_j 
 =  \dim [I_Z]_j  - \binom{j+1}{2} = j+1-|Z| 
\]
and
$h^1(\mathcal I_Z(j-1)) = h^1 (\mathcal I_{Z+jP}(j) )
 = 0$ for $j\geq |Z|$. 
\end{lemma}

\begin{proof} 
%Since vector space dimensions do not change when extending scalars we may assume that the base field is algebraically closed. 
%Thus, a form of degree $j$ with multiplicity $j$ at a point $P$ is a product of $j$ 
%linear forms corresponding to  a set of $j$ lines concurrent at $P$. If 
%$P$ is general, no line through two distinct points of $Z$ passes  through $P$, so
%there is a set of $j$ lines concurrent at $P$ which vanish on $Z$ if and only if $j\geq |Z|$, 
%and the lines are uniquely determined if $j=|Z|$. 
%Thus $\dim [I_{Z+jP}]_j = 0$ if $j<|Z|$ and
%$\dim [I_{Z+|Z|\cdot P}]_{|Z|}  = 1$. In particular, by adding 
%suitable lines through $P$ we obtain the first assertion.

If $f \in [I_{Z+jP}]_j$ then any line joining $P$ to a point of $Z$ is a component of $f$, since the restriction of $f$ to a line is either zero or has at most $j$ roots up to multiplicity. If $P$ is general, any such line contains no other points of $Z$. Hence $\dim [I_{Z+jP}]_j = 0$ if $j<|Z|$ and $\dim [I_{Z+|Z|\cdot P}]_{|Z|}  = 1$. In particular, by adding suitable lines through $P$ we obtain the first assertion.

Now, 
$1= \dim [I_{Z+ |Z| \cdot P}]_{|Z|}  \geq  \dim [I_{Z}]_{|Z|} 
-\binom{|Z|+1}{2}\geq \binom{|Z|+2}{2}-|Z|-\binom{|Z|+1}{2} = 1$,
hence $Z+|Z| \cdot P$ (and thus $Z$) imposes independent conditions on forms of degree $|Z|$.
This means $h^1(\mathcal I_Z(j)) = 0$ for $j=|Z|$ (and hence for $j\geq |Z|$), and it means
$h^1(\mathcal I_{Z+jP} (j)) = 0$ for $j = |Z|$.
Replacing $Z$ by $Z+Q$ for any point $Q\notin Z$, we now get $h^1(\mathcal I_{Z+Q+jP}(j)) = 0$ for $j = |Z+Q|=|Z|+1$ and hence
$Z+Q+jP$ imposes independent conditions on forms of degree $j=|Z|+1$, and therefore
$Z+jP$ also imposes independent conditions on forms of degree $j=|Z|+1$.
Continuing in this way, we see that for any $j \geq |Z|$, $Z + jP$ imposes independent 
conditions on forms of degree $j$; hence for such $j$ we have $h^1(\mathcal I_{Z+jP}(j)) = 0$. 
Thus $\dim [I_{Z+jP}]_j = \binom{j+2}{2}-|Z|-\binom{j+1}{2}=j+1-|Z|$  as asserted. 
Since $h^1(\mathcal I_Z(j)) = 0$ for $j\geq|Z|$, we also have
$\dim [I_Z]_j = \binom{j+2}{2}-|Z|$, so $\dim [I_{Z+jP}]_j$
can in addition be written as $\dim [I_Z]_j -\binom{j+1}{2}$.

Since we have already shown that $h^1(\mathcal I_{Z+jP}(j)) = 0$ for 
$j \geq |Z|$, it remains only to prove that $h^1(\mathcal I_Z(j-1)) = 0$ for $j \geq |Z|$. 
But this is true for any finite set of points, so we are done.
\end{proof} 

\begin{theorem} \label{thm:tZ = dZ}
Let $Z \subset \PP^2$ be a reduced scheme  consisting of a finite set of points
such that $h_Z (t_Z) < |Z|$
and let $P\in\PP^2$ be a general point.
Then 
\[
m_Z = t_Z < \frac{|Z| - 1}{2}
\]
and $\dim  [I_{Z+m_Z P}]_{m_Z+1} = 1$.
Furthermore, 
\[
\dim [I_{Z+jP}]_{j+1}
 = \dim [I_Z]_{j+1} - \binom{j+1}{2}
\]
for all $j\geq m_Z$ (hence $Z$ admits no unexpected curves).
\end{theorem}

\begin{proof}
If $t_Z=0$, then the points of $Z$ are collinear, in which case it's not hard to check that the claims hold.
So we may assume $t_Z>0$.
By Proposition \ref{prop:small tZ} we have to consider two cases.

\emph{Case 1}: Assume $Z$ is defined by an ideal $I_Z = (q, g)$, where $q, g\in R$ 
are forms of degree $2$ and  $t_Z+1$, respectively. We have (from the proof of Proposition \ref{prop:small tZ})
\[
h_Z(t_Z) = 2t_Z+1, \ h_Z(t_Z+1) = 2t_Z +2 = |Z|.
\]
Then for $j+1 < t_Z+1$ we get, 
\[
[I_Z \cap I_P^j]_{j+1} = [(q)\cap I_P^j]_{j+1} = q \cdot [I_P^j]_{j-1} = 0, 
\]
which implies $m_Z = t_Z$ by Lemma \ref{m, t and u} (d). Since $a_Z + b_Z = |Z|-1 = 2t_Z +1$, we also obtain $b_Z = u_Z +1 = t_Z+1$. Then Lemma \ref{FVlemma} gives $\dim [I_{Z+m_Z P}]_{m_Z+1} = 1$ as desired.
And since $|Z|=2t_Z+2$, we have $t_Z < \frac{|Z| - 1}{2}$. 

To show $\dim [I_{Z+jP}]_{j+1} = \dim [I_Z]_{j+1} - \binom{j+1}{2}$ for $j\geq m_Z$, first note that we have
\[
\begin{array}{rcl}
1 & = &  \dim [I_{Z+m_Z P}]_{m_Z+1} \\ 
& \geq &     \dim [I_Z]_{m_Z+1} -\binom{m_Z+1}{2} \\
& \geq & \binom{t_Z+3}{2}-(2t_Z+2)-\binom{t_Z+1}{2} \\
& = & 1, 
\end{array}
\]
hence $Z+m_ZP$ imposes independent conditions on
forms of degree $m_Z+1$. 
This also means that the points of $Z$ impose independent conditions on $[I_P^{m_Z}]_{m_Z+1}$. By adding lines through $P$, it is then clear that the 
points of $Z$ also impose independent conditions on $[I_P^{m_Z+k}]_{m_Z+1+k}$, 
which implies $\dim [I_{Z+jP}]_{j+1} = \dim [I_Z]_{j+1} - \binom{j+1}{2}$   
for all $j \geq m_Z$ as desired.

\bigskip

\emph{Case 2}: Assume a line defined by a linear form $\ell \in R$ contains precisely 
$|Z| - t_Z \ge t_Z + 2$ points of $Z$ (and hence $t_Z < \frac{|Z| - 1}{2}$) 
and let $Y$ be the set of these points. Let $U \subset Z$ be the subset of the other $t_Z$ points. 
Then $I_Y = (\ell, f)$ for some form $f$, where $\deg f \ge t_Z + 2$. Thus, for 
each integer $j$ and any general point $P \in \PP^2$, we get
\[
[I_Z \cap I_P^j]_{j+1} = [(\ell, f) \cap I_U \cap I_P^j]_{j+1}. 
\]
Since $\deg f \ge t_Z + 2$, it follows for $j+1 \le t_Z + 1$ that
\[
[I_{Z+jP}]_{j+1} = [I_Z \cap I_P^j]_{j+1} = [(\ell) \cap I_U \cap I_P^j]_{j+1} = \ell \cdot [I_U \cap I_P^j]_{j}
\]
because $\ell$ does not vanish at $P$ or at any of the points in $U$. 
Since $|U| = t_Z$, Lemma \ref{lem:initial degree} gives  $\dim [I_{U+jP}]_j  \leq 1$ for $j\leq t_Z$, with equality exactly when $j=t_Z$. 
Thus $m_Z = t_Z$ and  $\dim [I_{Z+m_Z P}]_{m_Z+1} = 1$.

Now assume $j \geq m_Z = t_Z$. Using Equation \ref{eq:difference Hilb},  we have 
\[
h_Z(j+1) = \min \{ t_Z+j+2, |Z| \} = \min \{ m_Z+j+2, |Z|\}. 
\]
Hence
\[
\dim [I_Z]_{j+1} = \max \left \{ \binom{j+3}{2} -(m_Z+j+2), \binom{j+3}{2} - |Z| \right \}.
\]
Then
\[
\begin{array}{rcl} 
\displaystyle \dim [I_Z]_{j+1} - \binom{j+1}{2} & = & \displaystyle \max \{ 2j+3 - (m_Z+j+2), 2j+3 - |Z| \}\\ 
& = & \max \{ j+1-m_Z, (j+1-m_Z) + (j+1-b_Z) \} \\
& = & \dim [I_{Z+jP}]_{j+1},
\end{array}
\]
the latter thanks to Lemma \ref{FVlemma}. 
\end{proof}

\section{The structure of unexpected curves and relation to syzygies} 
\label{structure} 

We  now give a rather detailed description of unexpected curves of a finite set of points $Z \subset \PP^2$. It turns out that any such curve has exactly one irreducible component of degree greater than one, and that this irreducible curve is rational and is an unexpected curve of a subset of $Z$ (which can be equal to $Z$). 

We begin with a description of curves whose existence is guaranteed by the definition of the multiplicity index $m_Z$. When $P$ is a general point
and $\dim [I_{Z+m_ZP}]_{m_Z+1}=1$ we will for later use denote the unique curve defined by $[I_{Z+m_ZP}]_{m_Z+1}$ by $C_P(Z)$. 
Thus when $\dim [I_{Z+m_ZP}]_{m_Z+1}=1$, by the next result there is an open set of points $P$ such that for each $P$ there is a subset $Z''_P$ of $Z$ such that
$C_P(Z)$ is the union of the $|Z''_P|$ lines through $P$ and each point of $Z''_P$, together with an irreducible curve $C_P(Z'_P)$ of degree
$m_Z+1-|Z''_P|$ containing $Z'_P$, where $Z'_P=Z\setminus Z''_P$. 
Therefore, by semicontinuity applied to $\dim[I_{Y+(m_Z-|Y|)P}]_{m_Z+1-|Y|}$ for the various subsets $Y$ of $Z$,
there is a single subset $Z''$ of $Z$ such that for a nonempty open set of points $P$ we have $Z''_P=Z''$.
I.e., when $\dim[I_{Z+m_ZP}]_{m_Z+1}=1$, it makes sense to talk about the components of $C_P(Z)$ for a general point $P$.

\begin{lemma}
     \label{lem:initial deg curves}
Let $Z$ be a finite set of points of $\PP^2$, and let $P \in \PP^2$ be a general point.  
If  $C$ is a curve of degree $m_Z+1$ containing $Z$, with multiplicity $m_Z$ at a general point $P
$, then it is reduced and  either it is  irreducible, or it is a union of lines through $P$ and an irreducible curve $C'$ whose multiplicity at $P$ is $-1 + \deg C'$. The curve $C'$ is rational 
and smooth away from $P$. 

Furthermore, the set $Z' = Z \cap C'$ has multiplicity index $m_{Z'} = m_Z - |Z''|$, where $Z'' = Z - Z'$, and each of the 
components of $C$ other than $C'$ passes through exactly one of the points of $Z''$. In particular, $\deg C' = \deg C - |Z''|=m_{Z'}+1$. 
\end{lemma}

\begin{proof}
Note that 
the multiplicity of an irreducible curve at a point is at most the degree of the curve and that  
the multiplicity is equal to the degree if and only if the curve is a line. 
Since the degree of $C$ at $P$ is precisely one more than its multiplicity at $P$, it follows that $C$ has a unique irreducible component $C'$ whose multiplicity at $P$ is $-1 + \deg C'$ and that this component has multiplicity one. Thus, $f = f' \cdot \ell_1 \cdots \ell_k$, where $f$ and $f'$ define the curves $C$ and $C'$, respectively, $k \ge 0$, and each $\ell_i$ is a linear form in $I_P$, so  $\deg(C')=\deg(C)-k=m_Z+1-k$. The genus formula implies that $C'$ is rational and smooth at all points other than $P$.

Put $Z' = Z \cap C'$. 
Since $P$ is general, each of the $k$ components of $C$ other than $C'$ passes through at most one 
point of $Z$. The union of these lines must contain $Z'' = Z - Z'$, and so $|Z''| \le k$. 

We have seen that $f'$ is in $[I_{Z' + (m_Z - k) P}]_{m_Z - k + 1}$, which implies $m_{Z'} \le m_Z - k$. Hence, the estimate $m_{Z'} = m_{Z - Z''} \ge m_Z - |Z''|$ gives $k \le |Z''|$. Therefore, we obtain $|Z''| = k$, so $m_{Z'} = m_Z - k=m_Z - |Z''|$ and $\deg(C')=\deg(C)-k=\deg(C)-|Z''|=m_Z+1-|Z''|=m_{Z'}+1$.
It also follows that each of the $k$ lines defined by some $\ell_i$ passes through one of the $k$ points of $Z''$, and no two lines pass through the same point of $Z''$. Thus, the curve $C$ is reduced. 
\end{proof}

Now we slightly improve Lemma \ref{FVlemma}. 
Recall by Lemma \ref{m, t and u} that $u_Z+1\geq m_Z$.

\begin{proposition} 
     \label{prop:vector space decomposition}  
Let $Z$ be a reduced 0-dimensional subscheme of $\PP^2$.      
If $P$ is a general point, then there is a plane curve $C$ defined by a form $f$ of degree $m_Z + 1$ 
that vanishes on $Z$ and to order $m_Z$ on $P$, and a plane curve $D$ defined by a form $g$ of degree 
$u_Z + 2$ that  vanishes on $Z$ and to order $u_Z+1$ on $P$, such that $C \cap D$ is a zero-dimensional 
subscheme, and, for all integers $j \ge 0$, there is an isomorphism of $K$-vector spaces
\[
[I_{Z+jP}]_{j+1} = \{ f \cdot [I_{(j-m_Z)P}]_{j-m_Z} \} \oplus \{ g \cdot [I_{ (j-u_Z-1)P}]_{j-u_Z-1} \}. 
\]
\end{proposition} 

\begin{proof} 
The existence of a curve $C$ with the desired properties is guaranteed by the definition of the multiplicity index $m_Z$.  
Suppose $C$ is defined by a form $f$. Then $f \cdot [I_{(j-m_Z)P}]_{j-m_Z} \subset [I_{Z+jP}]_{j+1}$, and, comparing dimensions by using Lemma \ref{FVlemma}, we see that 
\[
[I_{Z+jP}]_{j+1} = f \cdot [I_{(j-m_Z)P}]_{j-m_Z} \quad \text{ if } \;  m_Z \le j \le u_Z. 
\]
Since, by definition,  $[I_{Z+jP}]_{j+1}  = 0$ if $j < m_Z$, this proves our claim if $j \le u_Z$. 
If $j = u_Z + 1$, then Lemma \ref{FVlemma} gives  that there is a form $g$ of degree $u_Z + 2$ such that 
\[
[I_{Z+(u_Z + 1)P}]_{u_Z+2} = \{ f \cdot [I_{(u_Z + 1-m_Z)P}]_{u_Z + 1-m_Z} \} \oplus \langle g \rangle. 
\]

We are now going to show that $f$ and $g$ are relatively prime. Using the notation of Lemma~\ref{lem:initial deg curves}, 
write $f = f' \cdot \ell_1 \cdots \ell_k$, where $f'$ defines the irreducible curve $C'$ and each $\ell_i$ 
defines a line though $P$ and one of the $k$ points of $Z'' = Z - Z'$. Assume first that $f'$ divides 
$g$. Then the curve defined by $\frac{g}{f'}$ has multiplicity $u_Z + 1 - (m_Z - k)  = \deg \frac{g}{f'}$. 
 Thus, $\frac{g}{f'}$ is a product of linear forms of $I_P$ that must vanish on $Z''$. Hence, 
 $\ell_1 \cdots \ell_k$ divides $\frac{g}{f'}$, which implies 
 $g \in f \cdot [I_{(u_Z + 1-m_Z)P}]_{u_Z + 1-m_Z}$, a contradiction to the choice of $g$. 

Second assume that $k \ge 1$ and that one of the linear forms $\ell_i$ divides $g$. Let $P_i \in Z''$ 
be the point of $Z$ on which $\ell_i$ vanishes. By Lemma~\ref{lem:initial deg curves}, we know 
that $m_{Z'} = m_{Z - Z''} = m_Z - |Z''|$. This gives $m_{Z-P_i} = m_Z - 1$, and thus 
$u_{Z - P_i} = u_Z$. It follows that $\frac{f}{\ell_i} \in [I_{Z-P_i + m_{Z-P_i} P}]_{m_{Z-P_i} + 1}$ and 
$\frac{g}{\ell_i} \in [I_{Z-P_i + u_{Z-P_i} P}]_{u_{Z-P_i}+1}$. Applying the part of the statement we have already shown to $Z - P_i$, we conclude that $\frac{g}{\ell_i} \in \frac{f}{\ell_i} \cdot [I_{(u_{Z - P_i}  - m_{Z - P_i})P}]_{u_{Z - P_i}  - m_{Z - P_i}} = \frac{f}{\ell_i} \cdot [I_{(u_{Z}  - m_{Z} + 1)P}]_{u_{Z}  - m_{Z} + 1}$, which is again a contradiction to the choice of $g$. 

Thus, we have shown that $f$ and $g$ form a regular sequence. 
We claim that 
\[
\{ f \cdot [I_{(j-m_Z)P}]_{j-m_Z} \} \cap \{ g \cdot [I_{(j-u_Z-1)P}]_{j-u_Z-1} \} = 0 \quad \text{ if } \;  j \ge u_Z +1. 
\]
This is clear if $j \le |Z|$ because the degrees of the syzygies of the ideal $(f,g)$ are at least $m_Z + 1 + u_Z + 2 = |Z| + 1$. Assume now that the claim is false for some $j \ge |Z| + 1$. That is, there are forms $h_1, h_2$ of suitable degrees such that $f h_1 = g h_2$, where $h_2$ is a product of $j - u_Z - 2 \ge m_Z + 1$ linear forms in $I_P$. Since $f$ and $g$ are coprime, it follows that $f$ is a product of $m_Z + 1$ linear forms in $I_P$. By the generality of $P$, each of these linear forms vanishes on at most one point of $Z$. However, by definition $f$ vanishes at each point of $Z$, which implies $m_Z + 1 \ge |Z| = m_Z + u_Z + 2$. This is a contradiction because $u_Z \ge 0$ (as $|Z| \ge 2$). Thus, the above claim is shown. It gives that the sum 
\[
\{ f \cdot [I_{(j-m_Z)P}]_{j-m_Z} \} +  \{ g \cdot [I_{ (j-u_Z-1)P}]_{j-u_Z-1} \} \subset [I_{Z+(u_Z + 1)P}]_{j+1}
\]
is a direct sum. Since both sides have the same dimension we get equality, as desired.  
\end{proof} 

As a first consequence, we see that if $u_Z = m_Z -1$ then  there is an irreducible curve of degree $m_Z+1$ that vanishes on $Z$ and at a general point $P$ to order $m_Z$. 

\begin{corollary} 
     \label{for:balance gives irr}
If $Z$ satisfies $u_Z = m_Z -1$ (i.e. $a_Z = b_Z$), then there is an irreducible curve of degree $m_Z+1$ that vanishes on $Z$ and at a general point $P$ to order $m_Z$. 
\end{corollary} 

\begin{proof}
By Proposition \ref{prop:vector space decomposition}, the vector space $[I_{Z+m_ZP}]_{m_Z+1}$ contains two polynomials, $F$ and $G$, that form a regular sequence. 
By Lemma \ref{lem:initial deg curves}, any $K$-linear combination of $F$ and $G$ which is not irreducible has a linear factor.
Suppose there are two such linear combinations, for example $aF+G$ and $bF+G$ for distinct scalars $a$ and $b$, which have a common linear factor $L$. 
We have $aF+G = LH_1$ and $bF+G = LH_2$ for some forms $H_1$ and $H_2$, so $(a-b)F = L(H_1 - H_2)$. Then $L$ is a factor of $F$, so from the equation $aF+G = LH_1$ we also have that $L$ is a factor of $G$, and hence of every curve in the linear system.
This contradicts the fact that $F$ and $G$ are a regular sequence. 

By Lemma \ref{lem:initial deg curves} there are only finitely many possible linear factors, each of which corresponds to a point of $Z$, so we can conclude that at most $|Z|$ curves 
in the pencil defined by $F$ and $G$ are reducible.
Since $K$ is infinite, the general element must be irreducible.
\end{proof}

\begin{example}
We give an example of Corollary \ref{for:balance gives irr}, which shows that
not all of the curves $C$ of Lemma \ref{lem:initial deg curves} need be irreducible.
Let $P$ be a general point of $\PP^2$.
Let $D$ be an irreducible plane quartic with a triple point $P'$. Let $B$ be a smooth cubic through $P'$ meeting
$D$ transversely at 9 points away from $P'$. Let $Z$ be any subset of 7 of those 9 points.
By Bezout's Theorem, there is no cubic through $Z$ singular at $P'$, hence 
$\dim [I_{Z+2P'}]_3 = 0$. Now by semicontinuity we have $\dim [I_{Z+2P}]_3 = 0$.
A dimension count shows that $\dim [I_{Z+3P}]_4 > 0$, hence $m_Z=3$.
Thus $u_Z=|Z|-m_Z-2=2$, so by Corollary \ref{for:balance gives irr},
$[I_{Z+3P}]_4$ contains an irreducible form. For each $P$, pick such an irreducible form
and let $D_P$ be the curve it defines.
Let $Y$ be any subset of 6 points of $Z$. Since the splitting type of $Z$ is $(3,3)$, the splitting type of 
$Y$ is $(2,3)$. Thus there is a cubic through $Y$ singular at $P$. Let $B_P$ be any such cubic. 
Then $B_P$ cannot contain $Z$ since then $B_P$ and $D_P$ would contain a common component.
However, if $L$ is the line through $P$ and the point of $Z$ not on $B_P$, then
$B_P+L$ is a quartic through $Z$ with a triple point at $P$, so we see that
not every form in $[I_{Z+3P}]_4$ is irreducible. Thus if $C$ is a curve defined by
a form in $[I_{Z+3P}]_4$, then either $C$ is irreducible or $C$ is reducible, and both cases occur.
By Lemma \ref{lem:initial deg curves}, if $C$ is irreducible, then $C=C'$, and if $C$ is reducible, then 
$C$ has one linear component containing $P$ and a point of $Z$ and $C'$ is an irreducible cubic singular at $P$ and containing the other 6 points of $Z$.
A priori, $C$ could have two linear components, each containing $P$ and a point of $Z$, with $C'$ being 
an irreducible conic through $P$ and containing the other 5 points of $Z$, or
$C$ could have three linear components, each containing $P$ and a point of $Z$, with $C'$ being a line that does not contain 
$P$ but does contain the other 4 points of $Z$.
Neither can occur here, though: if there were an irreducible conic through 5 points, that conic is the only conic through
those 5 points, so it cannot contain a general point $P$, and if there a line through 4 points of $Z$, then that line
would have to be a component of $B$.
\end{example}

Recall from Theorem \ref{u_ZTheorem} that $Z$ admits unexpected curves if and only if 
$Z$ has an unexpected curve of degree $m_Z+1$, the least degree possible, and that 
if $Z$ has any unexpected curves then the degrees $t$ in which they occur are exactly $m_Z+1\leq t\leq u_Z$. 
By the following result, understanding unexpected curves for $Z$ reduces to understanding them in degree $m_Z+1$.

\begin{corollary}  
    \label{cor:unexpected curves only in initial deg}
Let $Z$ be a finite set of points with an unexpected curve. 
Then $Z$ has a unique unexpected curve $C$ in degree $m_Z+1$, and for each $m_Z+1<t\leq u_Z$,
the unexpected curves of degree $t$ are precisely the curves $C+L_1+\cdots+L_r$, where
$r=t-m_Z-1$ and each $L_i$ is an arbitrary line through the point $P$ (i.e., the general point at which $C$ is singular). 
\end{corollary}

\begin{proof}
This is an immediate consequence of Proposition \ref{prop:vector space decomposition} and the fact for $j\geq m_Z$ that
the nonzero forms in $[I_{(j-m_Z)P}]_{j-m_Z}$ are precisely the products of $j-m_Z$ linear forms vanishing at $P$
(the $j$ in Proposition \ref{prop:vector space decomposition} corresponds to $t-1$ here).
\end{proof}

We begin our quest to understand unexpected curves by, more generally, considering the case that 
there is a unique curve $C_P(Z)$ containing $Z$ with multiplicity $m_Z$ at a general point $P$.
By Proposition \ref{prop:vector space decomposition} this is exactly the case that $m_Z\leq u_Z$
(i.e., that $2m_Z+2\leq |Z|$, since $u_Z=|Z|-m_Z-2$).
Even when unexpected, the curve $C_P(Z)$ sometimes is and sometimes is not 
irreducible (see Example \ref{H19Example} and Propositions \ref{FermatProp}, 
 \ref{prop:unexpected irr}). The following result clarifies the connection between irreducibility and being unexpected.

\begin{corollary}\label{irredStructure}
Let $Z$ be a reduced 0-dimensional subscheme of $\PP^2$ with $m_Z\leq u_Z$, let $P\in\PP^2$ be a general point, let
$C=C_P(Z)$ be the unique curve containing $Z$ with multiplicity $m_Z$ at $P$ and let $t+1$ be the number of components of $C$. 
\begin{enumerate}
\item[(a)] The component $C'$ of $C$ given in Lemma \ref{lem:initial deg curves}
is the unique curve containing $Z'$ with multiplicity $m_{Z'}$ at $P$, where $Z'\subset Z$ is the subset given in the lemma; i.e., $C'=C_P(Z')$,
and we have $m_{Z'} +t\leq u_{Z'}$.

\item[(b)] $C$ is unexpected for $Z$ if and only if $1\leq m_{Z'}$ and $m_Z<u_Z$. (In particular, if $C$ is irreducible, then 
$C$ is unexpected for $Z$ if and only if $1\leq m_Z<u_Z$.)

\item[(c)] If $C$ is not irreducible, then $C'$ is unexpected for $Z'$ if and only if $1\leq m_{Z'}$.

\end{enumerate}
\end{corollary}

\begin{comment}
\begin{proof} 
The condition $m_Z \leq u_Z$ implies uniqueness by Proposition \ref{prop:vector space decomposition}. Using Lemma \ref{lem:initial deg curves}, it remains to show that the curve $C'$ given there is indeed the unique curve $C_P (Z')$. 

By Lemma \ref{lem:initial deg curves}, we know $m_{Z'} = m_Z -  |Z''| = \deg C' - 1$. Thus, Lemma \ref{m, t and u} gives $u_{Z'} = u_{Z} \ge m_Z \ge m_{Z'}$ and so $C'$ is the unique curve through $Z'$ of degree $m_{Z'}+1$ that has multiplicity $m_{Z'}$ at $P$. 
\end{proof} 
\end{comment}

\begin{proof} (a) Since $C'$ is a component of $C$ and is defined by an element of $[I_{Z'+m_{Z'}P}]_{m_{Z'}+1}$,
we see $\dim [I_{Z'+m_{Z'}P}]_{m_{Z'}+1}=1$, so $C'=C_P(Z')$. Moreover, $m_{Z'}+t=m_Z$ and 
$m_Z+u_Z+2=|Z|=|Z'|+t=m_{Z'}+u_{Z'}+2+t$, so $u_{Z'}=u_Z$. 
Thus we have $m_{Z'}+t=m_Z\leq u_Z=u_{Z'}$.

(b) We first show unexpectedness. If $h_Z(t_Z) = |Z|$ then by Theorem \ref{thm:intro} we obtain that $Z$ has an unexpected curve
if $m_Z<u_Z$. (We  did not explicitly need to use $1\leq m_{Z'}$ yet.) 
Thus, it remains to rule out that $h_Z (t_Z) < |Z|$. Indeed, if $h_Z (t_Z) < |Z|$, then Theorem  \ref{thm:tZ = dZ} gives $m_Z = t_Z$.
By Proposition \ref{prop:small tZ}, there are two cases.

\begin{itemize}

\item[-] In case (b)(i) we have that $Z$ is the complete intersection of a conic and a curve of degree 
$t_Z+1$, so $|Z| = 2t_Z+2 = 2m_Z +2$. But then $u_Z +1 = |Z| - m_Z -1 = 2 m_Z +2 - m_Z -1 = m_Z+1$, contradicting our hypothesis.
(We still did not explicitly need to use $1\leq m_{Z'}$.)

\medskip

\item[-] In case (b)(ii), there is a line through $|Z|-t_Z \geq t_Z+2$ of the points, i.e. through $|Z| - m_Z \geq m_Z+2$ of the points. 
Thus this line is a component of any curve of degree $m_Z+1$ containing $Z + m_ZP$, but does not itself contain $P$;
i.e., this line is $C'$, hence $m_{Z'}=0$, contrary to assumption.
\end{itemize} 

Conversely, $0<m_Z$ by Remark \ref{interesting}, and $m_Z<u_Z$ by Theorem \ref{u_ZTheorem}.

(c) We have $m_{Z'}<u_{Z'}$ by (a) and $C'$ is irreducible, so $C'$ is unexpected for $Z'$ by (b) if $1\leq m_{Z'}$. 
Conversely, if $C'$ is unexpected for $Z'$, then $1\leq m_{Z'}$ by Remark \ref{interesting}.
\end{proof} 

We can now prove Theorem \ref{thm:geomVersion}.

\begin{corollary}
   \label{cor:geomVersion}
Let $Z \subset \PP^2$ be a finite set of points. Then 
$Z$ admits an unexpected curve if and only if
$2m_Z+2<|Z|$ but no subset of $m_Z+2$ (or more) of the points is collinear.
In this case, $Z$ has an 
unexpected curve of degree $j$ if and only if $m_Z < j \le |Z|-m_Z-2$. 
\end{corollary} 

\begin{proof} 
 For convenience we set $d=|Z|$, the number of points.
Since $d-m_Z-2=u_Z$ by Lemma \ref{m, t and u}(c), the range of degrees in which unexpected curves can occur is due to Theorem \ref{u_ZTheorem}.
Now assume $Z$ admits an unexpected curve. Then it has one (call it $C$) of degree $m_Z+1$, so by Corollary \ref{irredStructure}(b),
$m_Z<u_Z$ and hence $2m_Z+2<d$. However, if there were a subset of $m_Z+2$ (or more) of points of $Z$ on a line $L$, 
let $Z'$ be the points of $Z$ on $L$ and let $Z''$ be the rest. 
By Bezout's Theorem, $L$ is a component of $C_P(Z)$ 
%UN no 
not 
through $P$, so
%$L=C_P(Z')$ and $d-|Z'|=|Z''|=m_Z$ by Lemma \ref{lem:initial deg curves}. But 
%%UN $m_{Z'}=t_Z=0$, 
%$m_{Z'}=t_{Z'}=0$, 
%and $m_Z\leq t_Z\leq t_{Z'}+|Z''|=m_Z$
%by Corollary \ref{cor:change of t_Z}. Thus $m_Z=t_Z$ so $Z$ cannot admit an unexpected curve, by Theorem \ref{mainThm1},
%contrary to hypothesis.
$L = C_P(Z')$ by Lemma \ref{lem:initial deg curves}. Then clearly $m_{Z'} = 0$, and $ t_{Z'} = 0$ by 
Example \ref{rem:tZ}, so again by Lemma \ref{lem:initial deg curves} we obtain $|Z''| = m_Z$. 
Furthermore,  $m_Z\leq t_Z\leq t_{Z'}+|Z''|=m_Z$
by Corollary \ref{cor:change of t_Z}. Thus $m_Z=t_Z$ so $Z$ cannot admit an unexpected curve, by Theorem \ref{mainThm1},
contrary to hypothesis.

Conversely, assume $2m_Z+2<d$ but no subset of $m_Z+2$ (or more) of the points is collinear.
Then $m_Z<u_Z$, hence $\dim [I_{Z+m_ZP}]_{m_Z+1}=1$ by Proposition \ref{prop:vector space decomposition}
so we can speak of $C_P(Z)$. If we now show that $m_{Z'}\geq1$, then $C_P(Z)$ 
is unexpected by Corollary \ref{irredStructure} and we will be done.
If $m_{Z'}=0$, then $C_P(Z)$ consists of the line $C_P(Z')$ through $s\leq m_Z+1$ points of $Z$,
plus $m_Z$ additional lines, one each for the remaining $d-s\geq d-m_Z-1 > 2m_Z+2-m_Z-1=m_Z+1$.
Thus $m_Z\geq d-s>m_Z+1$, which contradicts $m_{Z'}=0$. 
\end{proof} 

\begin{remark}\label{rem:NegCurves}
The hypothesis $2m_Z+2<|Z|$ of Corollary \ref{cor:geomVersion}
is equivalent to $(m_Z+1)^2-m_Z^2-|Z|<-1$. If we let $X\to\PP^2$
be the blow up of the points of $Z$ and a general point $P$, then
$C^2=(m_Z+1)^2-m_Z^2-|Z|$, where $C$ is the proper transform of 
the curve defined by an element of $[I_{Z+m_ZP}]_{m_Z+1}$. 
Thus, for example, if $C$ is reduced and irreducible with $m_Z>0$, then 
$Z$ admits an unexpected curve if and only if $C^2<-1$.
More generally, if $C$ has fewer than $m_Z+1$ components, then
$Z$ admits an unexpected curve if and only if $C^2<-1$.
\end{remark} 

We summarize part of our results from Lemma \ref{lem:initial deg curves} and Corollary \ref{irredStructure} as follows: 

\begin{theorem} 
        \label{thm:unexp curve structure}
Let $Z$ be a reduced 0-dimensional subscheme of $\PP^2$ that admits an unexpected curve and let $P\in\PP^2$ be a general point.
Then there is a unique unexpected curve of degree $m_Z+1$, namely $C_P(Z)$,
and there is a unique subset $Z' \subset Z$ such that $C_P(Z)$ is the union of 
$C_P(Z')$ and $|Z \setminus Z'|$ lines, where $C_P(Z' )$ is irreducible and is the unique unexpected 
curve of $Z'$ of degree $m_{Z'}+1$. Furthermore, $C_P(Z')$ is 
rational and smooth away from $P$. 
\end{theorem} 

Since by Remark \ref{interesting} the degree of any unexpected curve is at least three, it follows in combination with 
Theorem \ref{thm:unexp curve structure} that every unexpected curve of a finite set $Z \subset \PP^2$ has exactly one irreducible component of degree greater than one. This component is a rational curve that is an unexpected curve of 
the unique subset $Z'\subset Z$ such that $m_{Z'}=m_Z-(|Z|-|Z'|)$.
There is a very natural parametrization of this curve, which works more generally for the curve
$C_P(Z)$ when $m_Z \le u_Z$.

So let $Z$ be a reduced scheme of $d$ points $P_i\in\PP^2$ with $m_Z \le u_Z$.
For each point $P_i$ we have the dual line $L_i \subset (\PP^2)^{\vee}$ defined by linear form $\ell_i \in R = K[x,y,z]$. 
Set $f = \ell_1 \cdots \ell_d$, and let $\ell = \ell_P \in R$ be a general linear form, defining a line $L \subset (\mathbb P^2)^\vee$ that is dual to a general point $P \in \PP^2$. 

\begin{proposition}
       \label{paramprop1}
Assume that the characteristic of $K$ does not divide $|Z|$, that $K$ is algebraically closed, and  that $Z$ satisfies $m_Z \le u_Z$. 
Consider a syzygy 
\[
s_0 f_x + s_1 f_y + s_2 f_z + s_3 \ell = 0
\] 
of least degree of $\Jac (f)+(\ell)=(f_x,f_y,f_z, \ell)$,  and a rational map 
\[
\phi = (t_0 : t_1 : t_2):  (\PP^2)^{\vee} \dashrightarrow \PP^2, 
\]
where $t_0=ys_2-zs_1$, $t_1=-(xs_2-zs_0)$,  and $t_2=xs_1-ys_0$. Then the image of the restriction of $\phi$ 
to the line $L$ defined by $\ell$  is the irreducible curve $C_P(Z')$ determined by the subset $Z' \subset Z$  specified in 
Theorem \ref{thm:unexp curve structure}.
\end{proposition}

\begin{proof} 
The assumption on the characteristic guarantees that the derivation bundle $\mathcal D_Z$ is isomorphic to the syzygy bundle of $\Jac (f)$ (cf.\ Lemma \ref{syzygy bundle}), and thus each form $s_i$ has degree $m_Z = a_Z$. 

For a polynomial $g \in R$, denote by $\bar{g}$ its restriction to $L$. It is a polynomial in two 
variables, and thus a product of linear forms since $K$ is algebraically closed. The fact that $\deg s_i = a_Z$ implies (by definition of $a_Z$) that 
$\bar{\sigma} = (\bar{s}_0, \bar{s}_1, \bar{s}_2)$ is a syzygy of minimal degree of the 
restriction of $\Jac (f)$. It follows that 
the ideal generated by $\bar{s}_0, \bar{s}_1$ and $\bar{s}_2$ has codimension two, 
that is, that these polynomials do not have a common factor. Hence, the rational map 
\[
\sigma=(s_0 : s_1 :s_2): (\PP^2)^{\vee} \dashrightarrow (\PP^2)^{\vee}
\]
induces a morphism $\bar{\sigma}: L \to (\PP^2)^{\vee}$. 

For each $i$, $1 \leq i \leq d$, let $Q_i$ be the point of intersection of $L$ with $L_i$. Since $L$ is general, the $Q_i$ are distinct. 
Note that the line $L_{Q_i}$ dual to $Q_i$ contains $P$ and $P_i$.

Put $t=(t_0,t_1,t_2)$ and let $p$ be a point of $L$. 
Abusing notation, regard $t$ and $p$ as vectors in $K^3$;
then $t(p) = p\times\sigma(p)$ with $\sigma (p) \neq 0$. Hence, $t(p) = 0$ (i.e. $\phi(p)$ is undefined) if and only if  
$\sigma(p)=p$ as points of $\PP^2$.  Assume that this is the case. 
Then since $\sigma$ is a syzygy modulo $\ell$, we have $0=\sigma(p)\cdot\nabla f(p)=p\cdot\nabla f(p)= d \cdot  f(p)$, where $d = |Z|$. This proves the following: 

\begin{quote}
{\em If $p \in L$ and $\phi(p)$ is undefined then $p = L_i \cap L = Q_i$ for some $i$. }
\end{quote}
Notice that it does {\em not} follow that $\phi(Q_i)$ is undefined for all $i$.
For future reference, let $Y'' = \{Q_1,\ldots,Q_n\}$ be the set of  points on $L$ at which the map 
$\phi$ is not defined, and let  $Z'' = \{P_1,\ldots,P_n\}$ be the corresponding subset of $Z$. Furthermore, set $Z' = Z - Z''$, 
and let $Y'$ be comprised of the corresponding points $Q_i = L \cap L_i$ with $P_i \in Z'$. 

It follows that  $h = \bar{\ell_1} \cdots \bar{\ell_n}$ is a greatest common 
divisor of $\bar{t_0}$, $\bar{t_1}$, and $\bar{t_2}$ and that the map $\phi' : L \to \PP^2$ defined by 
$(\bar{t_0}/h :  \bar{t_1}/h : \bar{t_2}/h)$ is a morphism.  Let $\delta=\deg(\phi')$ be the degree of the mapping
(i.e., the degree of the inverse image of $\phi'(p)$ for a general $p\in L$). Then $\phi'(L)$ is an irreducible curve $C'$ 
of degree $(m_Z + 1 -n)/\delta$ that is equal to the Zariski closure of $\phi (L)$. 

Next we show that $\sigma(Q_i)$ is on the line $L_i$ for each $Q_i \in Y = Y' + Y''$. 
Indeed, since $Q_i$ is on $L$, the above syzygy gives $\sigma(Q_i) \cdot \nabla f(Q_i) = 0$.
Now write $f = \ell_ig$. Since $\nabla\ell_i=P_i$, the Leibniz rule gives 
$\nabla f = gP_i + \ell_i\nabla g$. As $Q_i$ is on $L_i$, we get 
$\nabla f(Q_i) = g(Q_i)P_i$. Since $g(Q_i) \neq 0$, from
$0=\sigma(Q_i) \cdot \nabla f(Q_i) =g(Q_i)\sigma(Q_i) \cdot P_i$ we conclude $\sigma(Q_i) \cdot P_i=0$,
hence $\sigma(Q_i) \in L_i$, as desired. 

Notice that if $Q \in L \backslash Y''$ (so $\sigma(Q) \neq Q)$) then $(t_0,t_1,t_2)$ are the coordinates 
of the point dual to the line through the points $Q$ and $\sigma(Q)$. Hence $\phi'$  is a morphism that maps a point $Q \in L \backslash Y''$ to the point 
that is dual to the line through the points $Q$ and $\sigma (Q)$. In particular, if $Q_i\in Y'$, then 
$L_i$ is the line through $Q_i$ and $\sigma(Q_i)$, and hence $\phi' (Q_i) = P_i$. Thus, we see that the curve $C'$ contains $Z'$.

Now we compute the multiplicity of $C'$ at $P$.
We have seen that for $p \in L$, $\phi(p)$ is undefined if and only if $\bar \sigma (p) = p$, and there 
are $n$ such points, namely the set $Y''$. We have also seen that for $p \in L \backslash Y''$, $\phi'$ 
maps $p$ to the point dual to the line through $p$ and $\bar \sigma(p)$. 
Thus the points of $L$ mapping to $P$ include all points of $(L\cap\sigma(L))\backslash Y''$.
Since each $s_i$ has degree $m_Z$, $D = \sigma (L) \cap L$ is a divisor of degree $m_Z$.
Therefore, the multiplicity of $C'$ at $P$ is at least $(m_Z-n)/\delta$. Let $\epsilon+(m_Z-n)/\delta$ be the multiplicity of $C'$ at $P$.
Thus $\epsilon+(m_Z-n)/\delta < \deg(C')=(m_Z-n+1)/\delta$, so $\delta\epsilon+m_Z-n\leq m_Z-n+1$.
Since $\delta\epsilon$ is a nonnegative integer, we must have $\epsilon=0$, hence
$(m_Z-n)/\delta$ and $(m_Z-n+1)/\delta$ are integers so $\delta=1$.
Thus $\deg(C')=m_Z-n+1$ and $C'$ has multiplicity $m_Z-n$ at $P$.

We now have that $C'\cup(\cup_{j = 1}^n L_{Q_{i_j}})$ has degree $m_Z+1$, contains $Z$ and has $P$
as a point of multiplicity $m_Z$.  Thus, $C'\cup(\cup_{j = 1}^n L_{Q_{i_j}})$ is the unique curve $C_P (Z)$ with these properties, 
and $C' = C_P (Z')$ by Corollary \ref{irredStructure}.
\end{proof}

In the proof of the result above, if none of the $\bar{\ell_i}$ is a common factor for $\bar{t_0}, \bar{t_1}$ and $\bar{t_2}$,
then $n=0$ so $\deg(C_P(Z'))=m_Z+1=\deg(C_P(Z))$, hence $C_P(Z')=C_P(Z)$ is irreducible.
If, in the above result, $\sigma$ is a global syzygy (i.e., $s_3=0$),
then $\sigma$ and $\ell$ become independent of each other, and
a minor modification of the argument above then gives us the following
criterion for irreducibility. 
The advantage here is not having to work modulo a general linear form $\ell$,
which can be a computational convenience when testing irreducibility explicitly.

\begin{proposition}\label{paramprop2} 
Assume that the characteristic of $K$ does not divide $|Z|$, that $K$ is algebraically closed, and that $Z$ satisfies $m_Z \le u_Z$. 
Suppose further that $\Jac (f)$ has a syzygy $s_0 f_x + s_1 f_y + s_2 f_z = 0$, where each $s_i$ has degree $m_Z$. If none of the forms 
$\ell_i$ is a common divisor of $t_0=ys_2-zs_1$, $t_1=-(xs_2-zs_0)$,  and $t_2=xs_1-ys_0$, 
then the curve $C_P(Z)$ is irreducible. 
\end{proposition}

\begin{remark}\label{recoveryRem}
Given the parametrization $\phi$ in Proposition \ref{paramprop1}, keeping in mind that 
$\bar{\sigma}\cdot\overline{\nabla f}=0$, we can recover 
$\bar{\sigma}$ using facts about triple vector products. Working formally, 
the first component of $(p \times \sigma(p))\times \nabla f$ is
$$f_y(xs_1-ys_0)+f_z(xs_2-zs_0)=x(f_ys_1+f_zs_2)-s_0(yf_y+zf_z).$$
But $x(f_ys_1+f_zs_2)= -xf_xs_0$ modulo $\ell$, so the first component
modulo $\ell$ is $-s_0(xf_x+yf_y+zf_z)=-\deg(f)fs_0$. In the same way
the second and third components are $-\deg(f)fs_1$ and $-\deg(f)fs_2$.
Thus for all $p\in L$ we have 
\[
(\phi \times \overline{\nabla f})(p) = (p \times \bar{\sigma}(p)) \times (\overline{\nabla f}(p)) = -\deg(f)(\bar{f}\bar{\sigma})(p);
\]
i.e., 
%UN $\bar{\phi'} \times \overline{\nabla f} = -\deg(f)\frac{\bar{f}}{h}\bar{\sigma}$, 
$\phi' \times \overline{\nabla f} = -\deg(f)\frac{\bar{f}}{h}\bar{\sigma}$, 
where $h$ is a greatest common 
divisor of $\bar{t_0}$, $\bar{t_1}$, and $\bar{t_2}$. 

Similarly, in Proposition \ref{paramprop2} we have $\phi \times {\nabla f} = -\deg(f)\frac{f}{h}\sigma$.
\end{remark}

We now consider the change 
in the multiplicity index if one adds a point to a given set of points. 

\begin{lemma} \label{lem:change of index}
Let $P_1,\ldots,P_s,Q$ be distinct points of $\PP^2$ and let $Z = P_1+\ldots+P_s$. Then $m_{Z+Q} = m_Z$ if either $u_Z = m_Z - 1$ or 
$Q \in \bigcap_{P \in P^2} C_P(Z)$. Otherwise $m_{Z+Q} = m_Z + 1$.
\end{lemma}

\begin{proof}
Note that $\dim_K [I_Z\cap I_P^{m_Z}]_{m_Z+1}>1$ if and only if $u_Z = m_Z - 1$.
For each integer $j \ge 0$, one has $I_{Z+Q} \cap I_P^j \subset I_Z \cap I_P^j$,
hence $m_{Z+Q,P}\geq m_Z$. If $Q$ 
lies on $C_P$ for all $P$ 
then $\dim[I_{Z+Q}\cap I_P^{m_Z}]_{m_Z+1}= \dim_K [I_Z\cap I_P^{m_Z}]_{m_Z+1}\geq 1$,
so $m_{Z+Q,P}=m_Z$. If $\dim_K [I_Z\cap I_P^{m_Z}]_{m_Z+1}>1$, then, since
$\dim_K [I_{Z+Q}\cap I_P^{m_Z}]_{m_Z+1}$ drops by at most 1, we have 
$\dim[I_{Z+Q}\cap I_P^{m_Z}]_{m_Z+1}\geq 1$, and again $m_{Z+Q,P}=m_Z$.
If $Q$ does not lie on $C_P$ for some (hence for general) $P$ 
then the dimension
drops exactly one, so if also $\dim_K [I_Z\cap I_P^{m_Z}]_{m_Z+1}=1$
we get $\dim[I_{Z+Q}\cap I_P^{m_Z}]_{m_Z+1}= 0$, hence $m_{Z+Q,P}\geq m_Z+1$.
But let $f \neq 0$ be a form of degree $m_Z + 1$ 
in $I_Z \cap I_P^{m_Z}$. Let $\ell$ be a linear form that defines the line through $P$ and $Q$. Then 
$\ell f \neq 0$ is in $[I_{Z+Q} \cap I_P^{m_Z+1}]_{m_Z+2}$, which shows $m_{Z+Q, P} \le m_Z+1$. 
\end{proof} 

See Example \ref{H19Example} for an illustration of how $Q$ can lie on all the curves 
$C_P(Z)$.

Thus, given $m_Z$, there are only two possible values of $m_{Z+Q}$. 
When the number 
of points of $Z$ is odd and  $m_Z$ is as large as 
possible, we can say which of these values occurs for an arbitrary point $Q$.

\begin{corollary} \label{prop:change of mZ, max mZ}
Let $Z$ be a finite reduced subscheme of $\PP^2$.
If $m_Z =  \frac{|Z| - 1}{2}$, then $m_{Z + Q} = m_Z$ for any point $Q$ not in $Z$.
\end{corollary} 

\begin{proof}
If $m_Z = \frac{|Z|-1}{2}$, then $|Z|=2m_Z+1$, so $u_Z=|Z|-2-m_Z=m_Z-1$. 
Now the result follows by Lemma~\ref{lem:change of index}.
\end{proof}

If  $m_Z < \frac{|Z|-1}{2}$ and $Q$ is a \emph{general} point,  we now find the value of $m_{Z+Q}$.

\begin{corollary} \label{conj:change of dZ}
Let $Z$ be a finite reduced subscheme of $\PP^2$ and let $Q$ be a general point.
If $m_{Z} < \frac{|Z| - 1}{2}$, then $m_{Z + Q} = m_Z + 1$.
\end{corollary}

\begin{proof}
If $m_Z < \frac{|Z|-1}{2}$, then $|Z|>2m_Z+1$, so $u_Z=|Z|-2-m_Z>m_Z-1$. Moreover, $\bigcap_{P \in P^2} C_P(Z)$ is a finite set. 
Hence the result follows from Lemma \ref{lem:change of index}. 
\end{proof}

\begin{remark}\label{constructing unexpected curves}
We can describe more precisely how unexpected curves arise.
Assume a reduced point scheme $Z$ has an unexpected curve $C$ of some degree $t$.
By Theorem \ref{u_ZTheorem} and Proposition \ref{prop:vector space decomposition},
$m_Z<t\leq u_Z$ and $C$ is the union of $C_P(Z)$ with $t-m_Z-1$ lines though $P$
(indeed, the linear system of curves corresponding to $[I_{Z+m_ZP}]_{m_Z+1}$ is the union of $C_P(Z)$ with 
all choices of $t-m_Z-1$ lines though $P$, and so they are all unexpected).
Moreover, by Lemma \ref{lem:initial deg curves} and Corollary \ref{irredStructure},
there is a unique subset $Z'\subseteq Z$ such that $C_P(Z')$ is irreducible and unexpected for $Z'$;
it has degree $m_{Z'}+1= m_Z+1-(|Z|-|Z'|)$ and we have
that $C_P(Z)$ is the union of $C_P(Z')$ with the lines through $P$ and the $|Z|-|Z'|$ points of $Z$ not in $Z'$.

Thus every $Z$ with an unexpected curve $C$ comes from a $Z'$ with an irreducible unexpected curve,
and $Z=Z'+Q_1+\cdots+Q_r$ for some set of $r$ distinct points $Q_i$ not in $Z'$. 
Since $m_Z<u_Z$, $m_{Z'}+r=m_Z$ and $m_Z+u_Z+2=|Z|=|Z'|+r=m_{Z'}+u_{Z'}+2+r$,
we see that $m_{Z'}+r=m_Z<u_Z=u_{Z'}$, so $r\leq u_{Z'}-(m_{Z'}+1)$.

In fact, if $Z$ has an unexpected curve, then $Z+Q_1+\cdots+Q_i$ also has an unexpected curve
for any distinct points $Q_i$ not in $Z$, for any $0\leq i\leq u_Z-(m_Z+1)$.
To see this, assume $u_Z>m_Z+1$ and let $Y=Z+Q$ for any point $Q\not\in Z$.
By induction it is enough to show $Y$ has an unexpected curve and that $u_Y-m_Y\geq u_Z-m_Z-1$. 
But 
%UN $m_Z\leq m_Y\leq m_Y+1$ 
$m_Z\leq m_Y\leq m_Z+1$
by Lemma \ref{lem:change of index}, so $u_Y\geq u_Z$ 
(since $m_Z+u_Z+2=|Z|$ and $m_Y+u_Y+2=|Y|=|Z|+1$), hence
$u_Y-m_Y\geq u_Z-m_Z-1$.

Assume that $Y$ does not have an unexpected curve. Then Theorem \ref{u_ZTheorem} gives $m_Y \ge t_Y$, and Corollary \ref{cor:geomVersion} shows that at least $m_Y + 2$ points of $Y$ are on a line $L$. Hence $L$ contains at least $m_Y + 1$ points of $Z$. If $m_Y > m_Z$, then at least $m_Z + 2$ points of $Z$ are collinear, which contradicts the assumption that $Z$ has an unexpected curve, using again Corollary \ref{cor:geomVersion}. We conclude that $m_Y = m_Z$. Now Theorem \ref{u_ZTheorem} and Corollary \ref{cor:change of t_Z} yield $m_Y = m_Z < t_Z \le t_Y$, a contradiction to $m_Y \ge t_Y$. Hence, $Y$ has an unexpected curve, as claimed. 
\end{remark}

We will  observe on more than one occasion below that it is of interest to know when 
$Z$ admits an {\em irreducible} unexpected curve of minimal degree $m_Z+1$. 
This motivates the next result.

\begin{corollary} 
       \label{irred curve least deg}
Assume that $Z$ is a finite set of points in $\mathbb P^2$ and let $P \in \PP^2$ be a general point.
Then every nonzero form in $[I_{Z+m_ZP}]_{m_Z+1}$ is irreducible if and only if $m_{Z - Q} = m_Z$ for each point $Q \in Z$.
\end{corollary}

\begin{proof}
Assume $m_{Z-Q} = m_Z$ for all $Q \in Z$. 
Let $C$ be a curve of degree $m_Z+1$ containing $Z$ and having multiplicity $m_Z$  at the general point $P$.   
By Lemma \ref{lem:initial deg curves},
if $C$ is not irreducible then there is at least one component of $C$ consisting of a line 
joining $P$ and a point $Q \in Z$. Removing this point and this line shows that 
$m_{Z-Q} < m_Z$, giving a contradiction.

Now assume $m_{Z-Q} \neq m_Z$ for some $Q$ (hence $m_{Z-Q} = m_Z -1$ by Lemma \ref{lem:change of index}), 
let $0\neq F\in[I_{Z-Q+(m_Z-1)P}]_{m_Z}$
and let $\ell$ be the linear form defining the line joining $Q$ to $P$. 
Then $\ell  F\in [I_{Z+m_ZP}]_{m_Z+1}$ is not irreducible.
\end{proof}

We now give a different criterion from the dual point of view (compared to Lemma \ref{lem:change of index}) 
concerning when the addition of a point increases the multiplicity index $m_Z$ (recall this is equal to $a_Z$)  and when it does not.
Let $Z$ be a reduced scheme of points, and let $Y=Z+Q$ for some $Q=(a : b :c)$ 
not in $Z$. 
Note the line dual to $Q$ is 
%UN $L_Q=ax+by+cz$. 
defined by $\ell_Q = ax+by+cz$. 
%UN
%Let $G$ be the product of the linear forms dual to the points of $Z$ 
%so $F = L_Q G$ is
%the product of the linear forms dual to the points of~$Y$. Let $\ell$ be a 
%%BH general linear form. 
%generic linear form. 
%Let $G_x, G_y, G_z$ be the first order partial derivatives of $G$, and $g_x, g_y, g_z$ their 
%restrictions modulo $\ell$ (these are not necessarily derivatives), and analogously for $F$. 
%Let $\ell_Q$ be the restriction of $L_Q$ modulo $\ell$. Let $g,f$ be the restrictions of $G, F$ respectively. Assume $\deg G = d$.
 Let $g$ be the product of the linear forms dual to the points of $Z$, and  
so $f = \ell_Q g$ is
the product of the linear forms dual to the points of~$Y$. Let $\ell$ be the general linear form dual to a general point $P \in \PP^2$. Denote the image of a polynomial $h \in R$ in $\overline{R} = R/\ell R$ by $\overline{h}$.

\begin{proposition}\label{my=mzDualCriterion}
Assume that $m_Z \le u_Z$,   that the characteristic of $K$ does not divide $|Z|$ nor $|Y|$, and that $K$ is algebraically closed. Then one has:   

\begin{itemize}

\item[(a)] For a general linear form $\ell$,  
consider a syzygy of least degree $r g_x + s g_y + t g_z + u \ell = 0$ of $\Jac (g) + (\ell)$, and so $r,s,t \in [R]_{m_Z}$. 
Then $\overline{\ell}_Q$ divides $\overline{ar+bs+ct}$ in $\overline{R}$ if and only if  $m_Y = m_Z$.

\item[(b)] Assume $\Jac (g)$ has a syzygy $r g_x + s g_y + t g_z = 0$ with $r,s,t \in [R]_{m_Z}$ (this will certainly be the case if the line arrangement dual to $Z$ is free). 
Then $m_Y = m_Z$ if and only if $\ell_Q$ divides $ar+bs+ct$.
\end{itemize}
\end{proposition} 

\begin{proof} 
By Euler's theorem we have  $x f_x + y f_y + z f_z = (d+1)f$, where $d = |Z|$. 
Abusing notation, regard $Q = (q, b, c)$ as a vector in $K^3$. As observed above, the Leibniz rule gives 
$\nabla f = g Q + \ell_Q \nabla g$.  

We first prove (a). Consider the dot product:
\begin{align*}
\hspace{4cm}&\hspace{-4cm} 
\displaystyle  \left [ \ell_Q (r, s, t)-\frac{1}{d+1} \big (Q \cdot (r, s, t) \big ) (x,y,z) \right ] \cdot \nabla f \\
& =  \ell_Q (r, s, t) \cdot \left [ g Q + \ell_Q \nabla g \right ] - \big (Q \cdot (r, s, t) \big ) f \\
& =  \ell_Q g \big (Q \cdot (r, s, t) \big ) +  \ell_Q^2  \big (\nabla g \cdot (r, s, t) \big ) - \big (Q \cdot (r, s, t) \big ) f\\
& =  - \ell_Q^2 u \ell. 
\end{align*}
This equation represents a syzygy of $\Jac (f) + (\ell)$. If $\ell_Q$ divides 
$ar+bs+ct = Q \cdot (r, s, t)$ modulo~$\ell$, then canceling $\ell_Q$ gives a syzygy, where the 
coefficients of the partial derivatives of $f$ have degree $m_Z$. 
Hence  we conclude  $m_Y=m_Z$. 

Conversely, assume now $m_Y=m_Z$. Thus, there is a syzygy $m f_z + n f_y + o f_z + p \ell = 0$  
with $m, n, o \in [R]_{m_Z}$. The assumption $m_Z \le u_Z$ implies that $C_P (Z) = C_P (Y)$. 
Proposition \ref{paramprop1} gives a parametrization of its irreducible component 
$C_P (Z') = C_P (Y')$. It is obtained from the cross product of $(r, s, t)$ and $(x, y, z)$ and of 
$(m, n, o)$ and $(x, y, z)$, respectively. It follows that there is a form $h \in R$ such that 
\[
\overline{(m, n, o)} = \overline{(r, s, t)}  + \overline{h} \, \overline{(x, y, z)}. 
\] 
Taking the dot product with $\overline{\nabla f} = \overline{\ell}_Q \overline{\nabla g} + \overline{g} \,\overline{Q}$, we obtain in $\overline{R}$
\[
0 = \big (\overline{(r, s, t)} \cdot \overline{Q}  \big ) \overline{g} + (d+1) \overline{h} \, \overline{f}. 
\]
Since $\overline{f} = \overline{\ell}_Q \overline{g}$, we conclude that $\overline{\ell}_Q$ divides $\overline{(r, s, t) \cdot Q}$, as claimed. 

We now prove (b). By assumption, $ar+bs+ct$ and $\ell_Q$ are independent of $\ell$. Hence, part (a) gives the desired conclusion. 
\end{proof}

\begin{remark} 
Let $Z$ be a finite set of points such that $m_Z\leq u_Z$.
%UN capital to lower case
Let $g$ be the product of the linear forms dual to the points of $Z$, and 
assume there exists a syzygy $r g_x + s g_y + t g_z = 0$ of degree $m_Z$.
Then Proposition \ref{my=mzDualCriterion} gives a way to compute
$\cap_{P\in\PP^2}C_P(Z)$. One just finds the locus of all $(a,b,c)$
such that $ax+by+cz$ divides $ar+bs+ct$. For example, to find all such
$(a,b,c)$ with $a\neq0$, just plug $-(by+cz)/a$ in for $x$ in $ar+bs+ct$
and regard the result as a polynomial with coefficients in $K(b/a,c/a)$.
The locus is given by the vanishing of these coefficients.
\end{remark}

%%%%%%%%%%%%%%%%%%%%%%%%%%%%%%%%%%%%%%%%%%%%%%%%%%%%%%

\section{Examples}
\label{sec:examples}

In this section we use the theory of line arrangements to present   examples that illustrate some of the ideas in the preceding sections, including the role of the characteristic. 
We also establish new stability results and show that points in linearly general position do not have unexpected curves. These examples make it clear that sets of points that admit 
unexpected curves are special, but nevertheless they occur surprisingly often.

We first exhibit a line arrangement that is not free and is dual to a set of points that has a unique unexpected curve, which is reducible. 

\begin{example}\label{H19Example}
For this example we assume our ground field has characteristic 0.
Consider the line configuration given by the lines defined by the following 19 linear forms:
$x$, $y$, $z$, $x+y$, 
$x-y$, $2x+y$, $2x-y$, $x+z$, 
$x-z$, $y+z$, $y-z$, $x+2z$, 
$x-2z$, $y+2z$, $y-2z$, $x-y+z$, 
$x-y-z$, $x-y+2z$, $x-y-2z$, shown in Figure \ref{H19Fig}.
Let $Z$ be the corresponding reduced scheme consisting of the 19 points dual to the lines,
sketched in Figure \ref{dualH19Fig}.

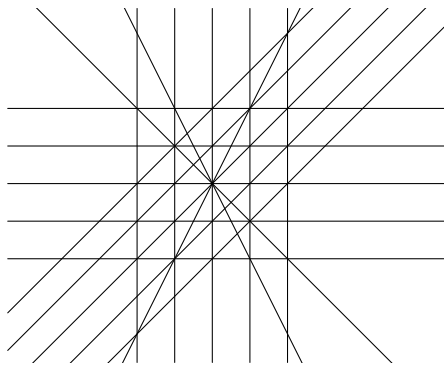
\begin{figure}[htbp]
\begin{center}
\begin{tikzpicture}[line cap=round,line join=round,>=triangle 45,x=.5cm,y=.5cm]
\clip(-5.44,-4.76) rectangle (6.16,4.66);
\draw [domain=-5.44:6.16] plot(\x,{(-0.0-1.0*\x)/-1.0});
\draw (0.0,-4.76) -- (0.0,4.66);
\draw [domain=-5.44:6.16] plot(\x,{(-0.0-0.0*\x)/1.0});
\draw [domain=-5.44:6.16] plot(\x,{-2.0*\x});
\draw [domain=-5.44:6.16] plot(\x,{(-0.0--2.0*\x)/1.0});
\draw [domain=-5.44:6.16] plot(\x,{(-0.0--1.0*\x)/-1.0});
\draw [domain=-5.44:6.16] plot(\x,{(-2.0-0.0*\x)/-1.0});
\draw [domain=-5.44:6.16] plot(\x,{(-1.0-0.0*\x)/-1.0});
\draw [domain=-5.44:6.16] plot(\x,{(--1.0--0.0*\x)/-1.0});
\draw [domain=-5.44:6.16] plot(\x,{(--2.0--0.0*\x)/-1.0});
\draw (-2.0,-4.76) -- (-2.0,4.66);
\draw (-1.0,-4.76) -- (-1.0,4.66);
\draw (1.0,-4.76) -- (1.0,4.66);
\draw (2.0,-4.76) -- (2.0,4.66);
\draw [domain=-5.44:6.16] plot(\x,{(--1.0--1.0*\x)/1.0});
\draw [domain=-5.44:6.16] plot(\x,{(--2.0--1.0*\x)/1.0});
\draw [domain=-5.44:6.16] plot(\x,{(-1.0--1.0*\x)/1.0});
\draw [domain=-5.44:6.16] plot(\x,{(-4.0--2.0*\x)/2.0});
\end{tikzpicture}
\caption{A configuration of 19 lines (the line at infinity, $z=0$, is not shown).}
\label{H19Fig}
\end{center}
\end{figure}

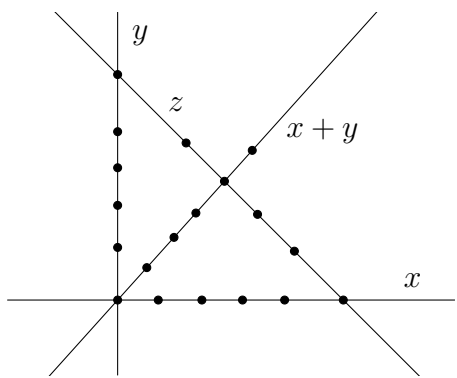
\begin{figure}[htbp]
\begin{center}
\begin{tikzpicture}[line cap=round,line join=round,>=triangle 45,x=1.0cm,y=1.0cm]
\clip(-1.463688897430969,-0.9982420398488032) rectangle (4.481599077455316,3.829741815619195);
\draw (0.0,-0.9982420398488032) -- (0.0,3.829741815619195);
\draw [domain=-1.463688897430969:4.481599077455316] plot(\x,{(--9.0-3.0*\x)/3.0});
\draw [domain=-1.463688897430969:4.481599077455316] plot(\x,{(-0.0-0.0*\x)/3.0});
\draw [domain=-1.463688897430969:4.481599077455316] plot(\x,{(-0.0--1.58*\x)/1.42});
\draw (3.6615593567813454,0.5) node[anchor=north west] {$x$};
\draw (0.5249074252034089,2.845694150810431) node[anchor=north west] {$z$};
\draw (0.05,3.8) node[anchor=north west] {$y$};
\draw (2.1,2.5586802485745417) node[anchor=north west] {$x+y$};
\begin{scriptsize}
\draw [fill=black] (0.0,0.0) circle (1.5pt);
\draw [fill=black] (0.0,3.0) circle (1.5pt);
\draw [fill=black] (3.0,0.0) circle (1.5pt);
\draw [fill=black] (0.0,0.7) circle (1.5pt);
\draw [fill=black] (0.0,1.26) circle (1.5pt);
\draw [fill=black] (0.0,1.76) circle (1.5pt);
\draw [fill=black] (0.0,2.24) circle (1.5pt);
\draw [fill=black] (1.42,1.58) circle (1.5pt);
\draw [fill=black] (0.91,2.09) circle (1.5pt);
\draw [fill=black] (1.86,1.140) circle (1.5pt);
\draw [fill=black] (2.35,0.65) circle (1.5pt);
\draw [fill=black] (0.38753589789044496,0.43120191455415724) circle (1.5pt);
\draw [fill=black] (0.7492696330437865,0.8336943804290022) circle (1.5pt);
\draw [fill=black] (1.0403935472433967,1.157620989186315) circle (1.5pt);
\draw [fill=black] (1.7896631802871834,1.9913153696153176) circle (1.5pt);
\draw [fill=black] (0.54,0.0) circle (1.5pt);
\draw [fill=black] (1.12,0.0) circle (1.5pt);
\draw [fill=black] (1.66,0.0) circle (1.5pt);
\draw [fill=black] (2.22,0.0) circle (1.5pt);
\end{scriptsize}
\end{tikzpicture}
\caption{A sketch of the points dual to the lines of the line configuration given in Figure \ref{H19Fig}.}
\label{dualH19Fig}
\end{center}
\end{figure}

\noindent It is not hard to verify that the first difference of the Hilbert function of $Z$ is $\Delta h_Z = (1,2,3,4,4,4,1)$, from which  we find that $t_Z=9$. Picking a random point $P$, Macaulay2 \cite{M2} finds that $[I_{Z+7P}]_8=0$. By upper semicontinuity, this means $m_Z>7$. Thus we have $8\leq m_Z\leq t_Z=9$. We claim that in fact $m_Z = 8$, i.e. that the splitting type is $(8,10)$. 

For a  general linear form $\ell$, set $\bar{R} = R/\ell R$ and $\bar{J} = \frac{J+(\ell)}{(\ell)}$, where $J \subset R$ is the Jacobian ideal. 
Consider the graded exact sequence induced by multiplication by $\ell$
\[
(R/J) (-1)  \stackrel{\ell}\longrightarrow R/J \to \bar{R}/\bar{J} \to 0. 
\]
Using a computer algebra system,  one gets $\dim_K [R/J]_{25} = 243$ and $\dim_K [R/J]_{26} = 244$. Hence, the above exact sequence, considered in degree 26, gives  $[\bar{R}/\bar{J}]_{26} \neq 0$. The minimal free resolution of $\bar{R}/\bar{J}$ over $\bar{R}$ has the form
\[
0 \to \mathbb F_2 \to \bar{R}^3 (-18) \to \bar{R} \to \bar{R}/\bar{J} \to 0. 
\]
Since $[\bar{R}/\bar{J}]_{26} \neq 0$, we obtain $[\mathbb F_2]_{26} \neq 0$. It follows that the splitting type is $(26-18,28-18) = (8,10)$ as claimed. Thus there is an unexpected curve only 
in degree 9. One can verify using Corollary \ref{irred curve least deg} and a computer algebra program that the unexpected curve is not irreducible, and indeed has two components, one of 
which is a line. 
Indeed, using \cite{cocoa} we have seen that the linear component is the line joining the general point $P$ to the point $[2,1,0]$. 

\end{example}

\begin{example} \label{example20}
It is interesting to note (based on computer experiments) that the arrangement of Example \ref{H19Example} is not free, but that if we either (i) remove $2x+y$ alone or (ii) replace $2x+y$ by $2y-x$ or (iii) add $(2y-x)$ to the configuration of 19 lines, these new configurations are free with splitting type (respectively) $(7,10)$,  $(7, 11)$ or $(8,11)$.

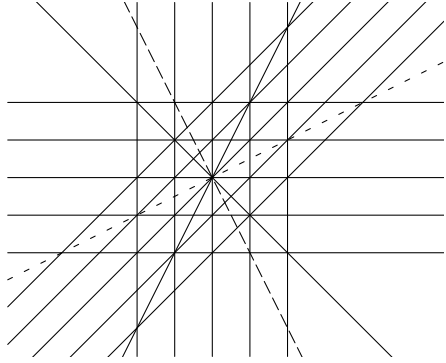
\begin{figure}[htbp] %\label{config20}
\begin{center}
\begin{tikzpicture}[line cap=round,line join=round,>=triangle 45,x=.5cm,y=.5cm]
\clip(-5.44,-4.76) rectangle (6.16,4.66);
\draw [domain=-5.44:6.16] plot(\x,{(-0.0-1.0*\x)/-1.0});
\draw (0.0,-4.76) -- (0.0,4.66);
\draw [domain=-5.44:6.16] plot(\x,{(-0.0-0.0*\x)/1.0});
%\draw [dotted,domain=-5.44:6.16] plot(\x,{-2.0*\x});
\draw [dash pattern=on 5pt off 2pt,domain=-5.44:6.16] plot(\x,{-2.0*\x});
\draw [domain=-5.44:6.16] plot(\x,{(-0.0--2.0*\x)/1.0});
\draw [domain=-5.44:6.16] plot(\x,{(-0.0--1.0*\x)/-1.0});
\draw [domain=-5.44:6.16] plot(\x,{(-2.0-0.0*\x)/-1.0});
\draw [domain=-5.44:6.16] plot(\x,{(-1.0-0.0*\x)/-1.0});
\draw [domain=-5.44:6.16] plot(\x,{(--1.0--0.0*\x)/-1.0});
\draw [domain=-5.44:6.16] plot(\x,{(--2.0--0.0*\x)/-1.0});
\draw [dash pattern=on 2pt off 5pt] (-5.44,-2.72) -- (6.16,3.08);
\draw (-2.0,-4.76) -- (-2.0,4.66);
\draw (-1.0,-4.76) -- (-1.0,4.66);
\draw (1.0,-4.76) -- (1.0,4.66);
\draw (2.0,-4.76) -- (2.0,4.66);
\draw [domain=-5.44:6.16] plot(\x,{(--1.0--1.0*\x)/1.0});
\draw [domain=-5.44:6.16] plot(\x,{(--2.0--1.0*\x)/1.0});
\draw [domain=-5.44:6.16] plot(\x,{(-1.0--1.0*\x)/1.0});
\draw [domain=-5.44:6.16] plot(\x,{(-4.0--2.0*\x)/2.0});
\end{tikzpicture}
\caption{A configuration of 20 lines (the line at infinity, $z=0$, is not shown).}
\label{H20Fig}
\end{center}
\end{figure}

In fact, in Figure \ref{H20Fig}

\begin{itemize}

\item the arrangement of 18 solid lines is free and irreducible but not complete (``irreducible'' meaning $C_P(Z)$ is  irreducible for a general point $P$, where $Z$ is the point scheme dual to the 18 lines, and ``not complete'' meaning  there is a point $Q$ not in $Z$ such that $m_{Z+Q}=m_Z$).

\item The arrangement of 18 solid lines plus the short-dashed line is free, irreducible and complete (i.e., if $Z$ is the point  scheme dual to the 19 lines, then $m_{Z+Q}=m_Z+1$ for all points $Q$ not in $Z$).

\item The arrangement of 18 solid lines plus the long-dashed line is not free, not irreducible and not complete.

\item The arrangement of all 20 lines is free and complete, but not irreducible.

\end{itemize}

These observations suggest the following question: Is the line arrangement ${\mathcal L}_Z$ for $Z$ always free if ${\mathcal L}_Z$ is irreducible or complete?
Or the converse?

On the dual side, taking $Z$ from Example \ref{H19Example}, if we set $Z_1 = Z \backslash \{ [2,1,0] \}$ and $Z_2 = Z_1 \cup \{ [-1,2,0] \}$ and $Z_3 = Z \cup \{ [-1,2,0] \}$, we obtain $m_{Z_1} = m_{Z_2} = 7$ and $m_Z = m_{Z_3} = 8$. Checking the Hilbert functions, one can show that these sets all have unexpected curves ($t_{Z_1} = 8, t_Z = t_{Z_2} = t_{Z_3} = 9$) and using the results of Section~\ref{structure} one can verify that the unexpected curve for $Z_1$ is irreducible and coincides with the unexpected curve for $Z_2$, while the unexpected curve for $Z_3$ coincides with that for $Z$ and is not irreducible. As the general point $P$ varies, all unexpected curves for $Z_1$ also contain $[-1,2,0]$.

\end{example}

In order to derive our next results we need the concept of a stable vector bundle. Since we need the Grauert-M\"ulich theorem, we will assume now that $K$ has characteristic zero and is algebraically closed. 
For unexplained terminology on vector bundles we refer to \cite{OSS}. Stable vector bundles of rank two can be characterized cohomologically. 

\begin{lemma} [{\cite[Lemma 3.1]{H}}] 
          \label{stable lemma}
A reflexive sheaf $\mathcal F$ of rank two over $\mathbb P^n$ is stable if and only if $H^0(\mathcal F_{norm}) = 0$.  If $c_1(\mathcal F)$ is even, then $\mathcal F$ is semistable iff $H^0(\mathcal F_{norm}(-1)) = 0$.  If $c_1(\mathcal F)$ is odd then semistability and stability coincide.
\end{lemma} 

Stability is related to the existence of unexpected curves as we see now. 

\begin{proposition}
   \label{prop:unexpected, so unstable}
Let $\mA$ be a line arrangement with splitting type $(a_Z, b_Z)$, dual to a set of points $Z$.  If $Z$ admits an unexpected curve, then $b_Z \ge a_Z + 2$. In particular, the derivation bundle of $\mA$ is not semistable.     
\end{proposition}

\begin{proof}
We have seen in Theorem \ref{mainThm1} that if $Z$ has an unexpected curve then $b_Z -a_Z \geq 2$. If the derivation bundle of $\mA$ were semistable, then the Grauert-M\"ulich theorem \cite{GM} gives $b_Z - a_Z \le 1$, hence the result.
\end{proof} 

The following result is useful for establishing stability.

\begin{lemma}
           \label{hal thm}
Let $\mathcal A$ be $(\mathcal A', \mathcal A, \mathcal A'')$ a triple of line arrangements, where $\mA$ consists of $d$ lines.  Then one has: 

\begin{itemize}

\item[(a)] $(${\cite[Theorem 4.5(a)] {hal}}$)$ If $d$ is odd, then $\mathcal D$ is stable if $\mathcal D'$ is stable and $|\mathcal A''| > \frac{d+1}{2}$.

\item[(b)] If $d$ is odd, then $\mathcal D$ is semistable if $\mathcal D'$ is stable.  

\item[(c)] $($\cite[Theorem 4.5(c)] {hal}$)$ If $d$ is even, then $\mathcal D$ is stable if $\mathcal D'$ is semistable and $|\mathcal A''| > \frac{d}{2}$. 

\item[(d)] If $d$ is even, then $\mathcal D$ is stable if $\mathcal D'$ is stable. 

\end{itemize}

\end{lemma} 

\begin{proof}
According to \cite[Theorem 3.2]{hal}, there is an exact sequence 
\[
0 \to \mD' (-1) \to \mD \to \mO_{\PP^1} (1 - |\mA''|) \to 0. 
\]
It implies parts (a) and (c). Using that for any vector bundle $\cE$ of rank two on $\PP^2$ one has $\cE^{\vee} \cong \cE (c_1 (\cE))$,  dualizing gives the exact sequence (see also \cite[Proposition 5.1]{FV2})
\begin{equation}
   \label{eq:dual hal}
0 \to  \mD \to \mD' \to \mO_{\PP^1} (-d + |\mA''| + 1) \to 0. 
\end{equation} 
Applying Lemma \ref{stable lemma}, parts (b) and (d) follow. 
\end{proof}

\begin{remark}
Lemma \ref{hal thm}(b) improves \cite[Theorem 4.5(b)] {hal} by eliminating any assumption on  $\mA''$. Note that in this case stability and semistability of $\mD'$ are equivalent by Lemma \ref{stable lemma}. 
\end{remark}

As a first consequence, we get information on sufficiently general line arrangements. 

\begin{proposition} 
        \label{star config type}
Let $\mathcal A_{d}$ be a configuration of  $d$ lines in $\mathbb P^2$ such that no three lines of $\mathcal A_{d}$ meet in a point.  Then the splitting type for $\mathcal A_{d}$ is
\[
\left ( \left \lfloor \frac{d-1}{2} \right \rfloor,  \left \lceil\frac{d-1}{2} \right \rceil \right ).
\]
Moreover, $\mathcal A_{d}$ is free if and only if $d \le 3$. 
\end{proposition}

\begin{proof}

Let $J_{d}$ be the Jacobian ideal of $\mathcal A_{d}$ and let $\bar J_{d}$ be its saturation.  By assumption, the lines in $\mathcal A_{d}$ form a star configuration. Thus, by  \cite{GHM} we know that the minimal free resolution of $\bar J_{d}$ is 
\[
0 \rightarrow R(-d)^{n-1} \rightarrow R(-d+1)^{n} \rightarrow \bar J_{d} \rightarrow 0.
\]
In particular, $J_{d}$ is saturated if and only if $d \leq 3$, so $\mathcal A_d$ is free if and only if $d \leq 3$.

Let us establish some notation.  This minimal free resolution for $J_d$ truncates to a short exact sequence
\[
0 \rightarrow E_d \rightarrow R(-d+1)^3 \rightarrow J_d \rightarrow 0.
\]
Let $\mathcal E_d$ be the sheafification of the reflexive module $E_d $.  Then $\mathcal D_d = \mathcal E_d(d-1)$ is the derivation bundle of  $\mathcal A_d$.  Note also that $(\mathcal D_d)_{norm} = \mathcal E_d(\frac{3d-3}{2})$ when $d$ is odd, and $(\mathcal D_d)_{norm} = \mathcal E_d(\frac{3d-4}{2})$ if $d$ is even.  
%Also, $(\mathcal D_{d-1})_{norm} = \mathcal E_{d-1}(\frac{3d-7}{2})$.  

First consider $d = 3$.  Then $\mathcal A_{d}$ is free and we have the minimal free resolution
\[
0 \rightarrow R(-3)^2 \rightarrow R(-2)^3 \rightarrow J_3 \rightarrow 0.
\]
Thus $\mathcal E_3 = \mathcal O_{\mathbb P^2}(-3)^2$, $\mathcal D_3 = \mathcal O_{\mathbb P^2}(-1)^2$ and $(\mathcal  D_3)_{norm} = \mathcal O_{\mathbb P^2}^2$.  By Lemma \ref{stable lemma}, $\mathcal D_3$ is semistable.
Clearly the splitting type for $\mathcal A_3$ is $(1,1)$ as claimed.  

Now assume that $d = 4$.  It  follows from Lemma \ref{hal thm} that $\mathcal D_4$ is stable, so the splitting type is as claimed thanks to the Grauert-M\"ulich theorem \cite{GM}.

Using Lemma \ref{hal thm}, we obtain by induction that $\mathcal D_d$ is stable for all $d \geq 4$.  Hence by the Grauert-M\"ulich theorem, the splitting type of $\mathcal D_d$ is as claimed.
\end{proof} 

This has the following consequence for the dual set of points. Recall that a set of points in $\PP^2$ is said to be in \emph{linearly general position} if no three of its points are on a line. Note that this is very different from assuming that $Z$ is a general set of points.

\begin{corollary} 
      \label{cor:dZ lin gen position} 
Let $Z$ be a set of  points in $\mathbb P^2$ in linear general position.  Then 
$m_Z = \left \lfloor \frac{|Z|-1}{2} \right \rfloor$, $u_Z = \left \lceil\frac{|Z|-1}{2} \right \rceil - 1$,  and   $Z$ does not admit an unexpected curve. Furthermore, for a general point $P$,   
\[
\dim [I_{Z+m_ZP}]_{m_Z+1} 
 = \begin{cases}
2 & \text{if $|Z|$ is odd}; \\
1 & \text{if $|Z|$ is even},  
\end{cases}
\]
and $[I_{Z + m_Z P}]_{m_Z +1}$ contains an irreducible form. 
\end{corollary}

\begin{proof} 
Notice that a set of points is in linearly general position if and only if the set of dual lines has the property that no three of them meet in a point. Hence, Proposition \ref{star config type} gives the asserted values of $m_Z$ and $u_Z$. Combined with Theorem \ref{mainThm1}, we get that $Z$ does not admit an unexpected curve. It remains to show the irreducibility statement. 

First, assume $Z$ is even. Then we have seen that, for each point $Q \in Z$, one has 
$m_Z = \frac{|Z|-2}{2} = m_{Z-Q}$. Hence, the unique curve determined by $[I_{Z + m_Z P}]_{m_Z +1}$ is irreducible by Corollary \ref{irred curve least deg}. 

Second, assume $Z$ is odd. Then $u_Z = m_Z -1$ and Corollary \ref{for:balance gives irr} gives the claim.  
%Choose a point $Q \notin Z$ such that $Z + Q$ is in linearly general position. Then $m_Z = \frac{|Z|-1}{2} = m_{Z+Q}$, and so $[I_{Z + Q + m_{Z + Q} P}]_{m_{Z+Q} +1} \subset [I_{Z + m_Z P}]_{m_Z +1}$. We just showed that the space on the left-hand side contains an irreducible form, which completes the argument. 
\end{proof} 

\begin{remark}
\begin{itemize}

\item[(i)] Corollary \ref{cor:dZ lin gen position} is a statement about a set of points. 
It would be interesting to have a more direct proof and to decide if the conclusion is also true if the base field has positive characteristic. 

\item[(ii)] The assumption that no three lines of $\mathcal A_d$ meet in a point (or the dual version, that the points are in linearly general position) allows for a clean result. Nevertheless, the proof requires much less. If the line arrangement can be built up from a set of lines with semistable syzygy bundle such that each line added meets the existing $e$ lines (say) in more than $\lfloor \frac{e+1}{2} \rfloor$ points then the same conclusion holds, thanks to Lemma~\ref{hal thm}.
\end{itemize}

\end{remark}

There are some further theoretical tools for determining splitting types, which we consider now. 
Let $\mA = \mA (f)$ be a line arrangement in $\PP^2$.   Let $L$ be one of the components of $\mA$ defined by a linear form $\ell$.  Let
$g = f/\ell$.  Then $\bar{g}$, the restriction of $g$ to $L$, is a polynomial of  the same degree as
$g$ though it is not necessarily reduced.  If $\bar{g}'$ is the radical of $\bar{g}$, then 
$\bar{g}'$ defines a hyperplane arrangement  of $L = \PP^1$, called the \emph{restriction}, which we
denote $\mA''$. Moreover, the arrangement defined by $g$ is often denoted by $\mA'$, and one  
refers to $(\mathcal A', \mathcal A, \mathcal A'')$ as a triple of hyperplane arrangements. 
Thus if $\mathcal A$ is a line arrangement then $\mathcal A'$ is obtained from 
$\mathcal A$ by removing a line $L$, and $\mathcal A''$ is the restriction of $\mathcal A'$ to $L$.
Notice that the arrangement $\mathcal A'' \subset \mathbb P^1$ is free with exponent $|\mathcal A''|-1$.

\begin{remark} \label{stefan}
A line arrangement $\mathcal A$ in $\mathbb P^2$ is {\em supersolvable} if it has a so-called {\em modular point}, i.e. a point $P$ with the property that if $\ell_1, \ell_2 \in \mathcal A$ and if $Q$ is the intersection of $\ell_1$ and $\ell_2$ then the line joining $P$ and $Q$ is a line of $\mathcal A$. A standard fact is that if $\mathcal A$ is a supersolvable line arrangement consisting of $d$ lines, $m$ of which pass through the  modular point $P$, then $\mathcal A$ is free, and the splitting type is $(m-1,d-m)$. We are grateful to S. Toh\^{a}neanu for pointing out that the computation of the splitting type is a simple application of the addition-deletion theorem (Theorem \ref{add-del} below) using induction on $d$, with  the base case being the case that all lines pass through a single point. 
\end{remark}

\begin{theorem} [{Addition-Deletion Theorem; see, e.g., \cite[Theorem 4.51]{OT}}] 
       \label{add-del}
Let $(\mathcal A', \mathcal A, \mathcal A'')$ be a triple of line arrangements. Then any two of the following imply the third:

\begin{itemize}

\item[] $\mathcal A$ is free with exponents $(a+1,b)$ (respectively, $(a, b+1)$);

\item[] $\mathcal A'$ is free with exponents $(a,b)$;

\item[] $\mathcal A''$ is free with exponent $(b)$ or $(a)$ (i.e. $\mathcal A'$ meets $\ell$ in $b+1$ (resp., $a+1$) points, ignoring multiplicity).

\end{itemize}

\end{theorem}

We use this result to study so-called \emph{Fermat arrangements} of lines \cite{U}. 
We note that these are also sometimes known as {\em monomial arrangements}
(see \cite[Example 10.6]{refSuciu} and \cite[page 247]{OT}).
These arrangements consist of $3 t$ lines ($t \ge 1$) that are defined by the linear factors of 
$f=(x^t-y^t)(x^t-z^t)(y^t-z^t)$. If  $t>3$ or $t=2$, there are $t^2$ points where exactly 3 lines cross and 3 points where exactly
$t$ lines cross, and no other crossing points. When $t=3$, there are 12 points where exactly 3 lines cross
and no other crossing points. When $t=1$ there is only one crossing point, and 3 lines cross there.
The set of  points 
$Z_t$ dual to the lines is defined by the ideal $(x^t+y^t+z^t, xyz)$
(i.e., the intersection of the Fermat $t$-ic with the coordinate axes) when $t$ is odd, and by
$(x^t-y^t,z)\cap (x^t-z^t,y)\cap(y^t-z^t,x)$ when $t$ is even.
Although the freeness is known (and the splitting types too, in terms of degrees of generators
of certain rings of invariants) \cite[Theorem 6.60, \& p.\ 247]{OT},
for the reader's convenience, we include a short proof here as part of the next result.

\begin{proposition} \label{FermatProp}
Suppose that the base field $K$ contains a primitive $t$-th root of unity. 
If $t>2$, then the  Fermat line configuration is free, with splitting type $(t+1, 2t-2)$. If $t \geq 5$, 
the dual set of points $Z=Z_t$ admits unexpected curves of degrees $t+2,\dots,2t-3$ and we have
$m_Z=t+1$, $u_Z=2t-3$ and, for $t\geq 5$, $t+1<t_Z\leq (3t-1)/2$.
The unexpected curve of degree $t+2$ is unique and irreducible.
\end{proposition}

\begin{proof}
We first prove freeness. We will start with a slightly larger line arrangement, and produce the Fermat 
arrangement by removing two lines.  The configuration of lines defined by the factors of $g=xy(x^t-y^t)(x^t-z^t)(y^t-z^t)$
is supersolvable since every point of intersection of two of the lines
is on one of the lines through the point defined by $x=0$ and $y=0$.
Thus the line arrangement $\mA (g)$ is free (see Remark \ref{stefan}).

Now we  determine its splitting type, $(a,b)$, where $a \leq b$.
Observe that there are $d = 3t+2$ lines in $\mA = \mA (g)$, and the  modular point lies on $m = t+2$ lines. Hence by Remark \ref{stefan}, the splitting type of $\mA$ is $(t+1,2t)$. 

Next we successively remove the lines defined by $x$ and $y$ from $\mA$. First let $\mA'= \mA (\frac{g}{x})$ and let $A''$ be the arrangement obtained by restricting
$\mA'$ to $x=0$. Clearly $\mA''$ is free with type $t+1$, so by  the Addition-Deletion Theorem \ref{add-del} $\mA'$ is an arrangement 
which is free of type $(t+1,2t-1)$. Now delete $y$ from $\mA'$ and apply Addition-Deletion again
to see that $(x^t-y^t)(x^t-z^t)(y^t-z^t)$ gives a free arrangement of type $(t+1,2t-2)$. 

Thus for the dual set of points $Z$ we have that $m_Z = t+1$ and $u_Z = 2t -3$. 
By \cite[Theorem III.1(a)]{Ha},
the $3t$  points of $Z$  impose independent conditions on forms of degree $t+1$ or more,
so $h^0(\mathcal I_Z(j+1))=\binom{j+3}{2}-3t$ for $j+1\geq t+1$.  Thus, taking $j=m_Z=t+1$, we have $h^0(\mathcal I_Z(t+2))-\binom{t+2}{2}=5-t$ 
and since $t_Z\geq m_Z$, we see $t_Z>m_Z$ for $t\geq 5$. 
Taking $j=t+1+s$, we have $h^0(\mathcal I_Z(t+2+s))-\binom{t+2+s}{2}=\binom{t+4+s}{2}-3t-\binom{t+2+s}{2}=2s-t+5$. Since this is positive for
$s>(t-5)/2$, we have $t_Z\leq t+1+(t-3)/2=(3t-1)/2$. Thus $m_Z<t_Z\leq (3t-1)/2$ for $t\geq 5$.

Now Theorem \ref{mainThm1} gives that,  for $t \geq 5$, the set  $Z$ admits an unexpected curve of degree $j$  whenever  $t+2 \leq j \leq 2t-2$. By Corollary \ref{irredStructure}, the unexpected curve $C_P = C_P (Z)$ of degree 
$t+2$ is unique. 

It remains to prove that $C_P$ is irreducible. 
For each point $q \in Z$, consider the set $A_q$ of points $P \in \mathbb P^2$ such 
that $[I_{(m_Z-1)P+Z-q}]_{m_Z} \neq 0$. If none of the sets $A_q$, $q\in Z$, has 
closure containing a nonempty open set, then $C_P$ is irreducible for general 
$P$ by Corollary \ref{irred curve least deg}. To prove that this is indeed the case, we argue by 
contradiction. 

Assume that $A_q$ has closure containing a nonempty open 
set for some $q$. Then by upper-semicontinuity $A_q$ contains a nonempty open set 
$V = V_q \subseteq U$ such that for all $P \in V$ the line through $P$ and $q$ is a 
component of $C_P$ (see Corollary \ref{irredStructure}).

Note that the points of $Z$ all are of the form $(0,1,\alpha^j)$ or cyclic permutations thereof,
where $\alpha$ is a primitive root of $x^t-1=0$. Thus the diagonal matrices of the form $\operatorname{Diag}(1,1,\alpha^i)$,
together with permutations of the variables, give a transitive action on $Z$
by linear automorphisms of $\PP^2$. Let $q\neq q'\in Z$ and let
$\phi=\phi_{q'}$ be one of these linear automorphisms, chosen such that $\phi(q)=q'$.

For each $P\in \phi(V)\cap V$, we have that the line $L_{q,P}$ through $q$ and $P$ is a component of $C_P$
(since $P\in V$). But $P\in \phi(V)$, so $P=\phi(Q)$ for some 
$Q\in V$, and $L_{q,Q}$ is a component of $C_Q$ (since $Q\in V$).
Uniqueness tells us that $\phi(C_Q)= C_P$, and 
so $\phi(L_{q,Q})=L_{\phi(q),\phi(Q)}=L_{\phi(q),P}$ is also a component of $C_P$.

Let 
$$
W=\Big(\bigcap_{q'}\phi_{q'}(V)\Big)\cap V, 
$$ 
where the intersection is over all points $q'\in Z - q$. 
By the argument above, for each $P\in W$, every line through $P$ and a point of $Z$ is a component of
$C_P$. Thus for a general point $P$, $C_P$ has $3t$ linear components,
hence $3t\leq \deg(C_P)=t+2$. Since $t\geq 5$, this is impossible and so
$C_P$ is irreducible.
\end{proof}

\begin{remark} \label{KleinAndWiman}
For $t \geq 3$, the Fermat line arrangement has the remarkable property that wherever two of the lines cross there is at least one more line through the crossing point. Apart from the trivial case of 3 or more concurrent lines, only two other complex line arrangements are known with that property. One, due to F.\ Klein in 1879, has 21 lines with 49 crossing points, 21 of which are where exactly 4 lines cross and 28 of which are where exactly  3 lines cross. The other, due to A.\ Wiman in 1896, has 45 lines with 201 crossing points, 36 of which are where exactly 5 lines cross, 45 of which are where exactly 4 lines cross and 120 of which are where exactly 3 lines cross. 
A Macaulay2 calculation shows both arrangements are free, and (as noted in \cite[Example 4.1.6]{brianlect}) their splitting types are (respectively) (9,11) and (19,25); for a more conceptual verification see \cite{ilardi}.
Using this information, as well as Macaulay2 \cite{M2} to compute the Hilbert function, we conclude:

\begin{itemize}

\item If $Z$ is the set of 21 points dual to the 21 lines of the Klein configuration, then $m_Z=9$, $u_Z = 10$ and $t_Z=10$, so $Z$ has an unexpected curve in degree 10.

\item If $Z$ is the set of 45 points dual to the 45 lines of the Wiman configuration, then $m_Z=19$, $u_Z = 24$, and $t_Z=22$, so $Z$ has unexpected curves in degrees 20, 21, 22, 23 and 24.

\end{itemize}

\noindent Moreover, by Proposition \ref{paramprop2}, the unexpected curve in degree $m_Z+1$ is irreducible for both the Klein and the Wiman line arrangements.
See \cite{BDHHSS} for a detailed discussion of these line arrangements and for additional references.
\end{remark}

We now describe another infinite family of sets of points in which each set has an irreducible 
unexpected curve. This family is defined over the field of rational numbers. 
We begin by describing the family of dual  line arrangements.

\begin{example} \label{surprise}
 Let $\mathcal A$ be the arrangement of five lines defined by the form $xyz(x+y)(x-y)$. We will denote by $a$ the line $x-y = 0$, by $d$ the line $x+y =0$, by $i$ the line at infinity ($z =0$), and by $h_1$ and $v_1$ the $x$ and $y$ axes, respectively. We remark in passing that there is some flexibility in the choice of these five lines, but that an arbitrary  configuration of five lines with the same intersection lattice is {\em not} always going to lead to arrangements with the properties that we will describe. (For example, replacing $x-y$ by any other line through the origin will fail to satisfy the requirement below that $h_3$ passes through $d \cap v_2$.)

We will add lines to $\mathcal A$, and define the line arrangements $\mathcal A_k$ inductively, where $k$ is the total number of lines that we have added to $\mathcal A$. In what follows, for simplicity we will refer to the lines containing the point of intersection of $i$ and $v_1$ as ``vertical lines," and the lines containing the point of intersection of $i$ and $h_1$ as ``horizontal lines."

$\mathcal A_1$ is obtained by adding to $\mathcal A$  an arbitrary vertical line, $v_2$. The next three lines added to $\mathcal A_1$ are then determined: $h_2$ is the horizontal line through $a \cap v_2$, $v_3$ is the vertical line through $d \cap h_2$, and $h_3$ is the horizontal line through $a \cap v_3$. The key fact is that $h_3$ also passes through $d \cap v_2$. This gives the arrangements $\mathcal A_1, \mathcal A_2,\mathcal A_3, \mathcal A_4$.

We continue in this way, taking an arbitrary vertical line $v_4$ and adding a horizontal line $h_4$, a vertical line $v_5$, and another horizontal line $h_5$ in the manner just described to obtain configurations $\mathcal A_5, \mathcal A_6, \mathcal A_7, \mathcal A_8$. Of special interest to us will be the configurations $\mathcal A_n$ where $n$ is a multiple of 4. In particular, $\mathcal A_{4(k+1)}$ is obtained from $\mathcal A_{4k}$ by adding the lines $v_{2k+2}, h_{2k+2}, v_{2k+3}, h_{2k+3}$ in that order. See Figure~\ref{fig:line-config}
for an example of the line configuration. 

Notice that $\mathcal A_4$ is the $B_3$ arrangement, so our example includes the one studied in \cite{DIV} as a special case.

\begin{comment}
\begin{figure}[!ht]    
    \includegraphics[scale=.4]{line-config}
    \caption{The line arrangement $\mathcal{A}_{12}$.  }
         \label{fig:line-config}
\end{figure}
\end{comment}

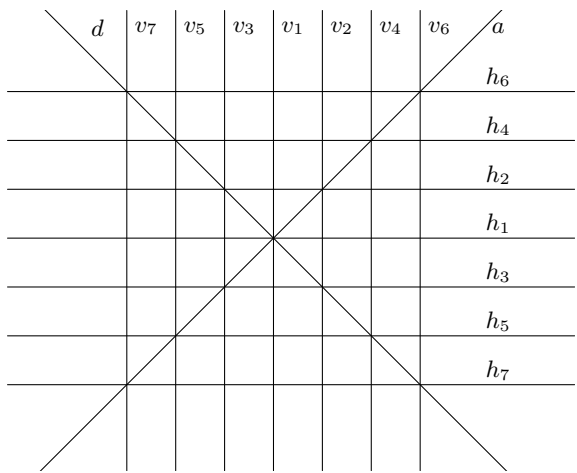
\begin{figure}[htbp]
\begin{center}
\begin{tikzpicture}[line cap=round,line join=round,>=triangle 45,x=.65cm,y=.65cm]
\clip(-5.44,-4.76) rectangle (6.16,4.66);
\draw [domain=-5.44:6.16] plot(\x,{(-0.0-1.0*\x)/-1.0});
\draw (0.0,-4.76) -- (0.0,4.66);
\draw (-3,-4.76) -- (-3,4.66);
\draw (3,-4.76) -- (3,4.66);
\draw [domain=-5.44:6.16] plot(\x,{(-0.0-0.0*\x)/1.0});
%\draw [domain=-5.44:6.16] plot(\x,{-2.0*\x});
%\draw [domain=-5.44:6.16] plot(\x,{(-0.0--2.0*\x)/1.0});
\draw [domain=-5.44:6.16] plot(\x,{(-0.0--1.0*\x)/-1.0});
\draw [domain=-5.44:6.16] plot(\x,{(-3.0-0.0*\x)/-1.0});
\draw [domain=-5.44:6.16] plot(\x,{(-2.0-0.0*\x)/-1.0});
\draw [domain=-5.44:6.16] plot(\x,{(-1.0-0.0*\x)/-1.0});
\draw [domain=-5.44:6.16] plot(\x,{(--1.0--0.0*\x)/-1.0});
\draw [domain=-5.44:6.16] plot(\x,{(--2.0--0.0*\x)/-1.0});
\draw [domain=-5.44:6.16] plot(\x,{(--3.0--0.0*\x)/-1.0});
\draw (-2.0,-4.76) -- (-2.0,4.66);
\draw (-1.0,-4.76) -- (-1.0,4.66);
\draw (1.0,-4.76) -- (1.0,4.66);
\draw (2.0,-4.76) -- (2.0,4.66);
%\draw [domain=-5.44:6.16] plot(\x,{(--1.0--1.0*\x)/1.0});
%\draw [domain=-5.44:6.16] plot(\x,{(--2.0--1.0*\x)/1.0});
%\draw [domain=-5.44:6.16] plot(\x,{(-1.0--1.0*\x)/1.0});
%\draw [domain=-5.44:6.16] plot(\x,{(-4.0--2.0*\x)/2.0});
\begin{scriptsize}
\draw (-3.6,4.3) node {$d$};
\draw (-2.6,4.3) node {$v_7$};
\draw (-1.6,4.3) node {$v_5$};
\draw (-.6,4.3) node {$v_3$};
\draw (.4,4.3) node {$v_1$};
\draw (1.4,4.3) node {$v_2$};
\draw (2.4,4.3) node {$v_4$};
\draw (3.4,4.3) node {$v_6$};
\draw (4.6,4.3) node {$a$};
\draw (4.6,3.3) node {$h_6$};
\draw (4.6,2.3) node {$h_4$};
\draw (4.6,1.3) node {$h_2$};
\draw (4.6,.3) node {$h_1$};
\draw (4.6,-.7) node {$h_3$};
\draw (4.6,-1.7) node {$h_5$};
\draw (4.6,-2.7) node {$h_7$};
\end{scriptsize}

\end{tikzpicture}
\caption{The line arrangement $\mathcal{A}_{12}$ (the line at infinity, $z=0$ and denoted $i$, is not shown).}
 \label{fig:line-config}
 \end{center}
\end{figure}

%\begin{figure}[!ht]    
%    \includegraphics[scale=.5]{point-config}
%    \caption{The point configuration dual to $\mathcal{A}_{12}$.  }
%        \label{fig:point-config}
%\end{figure}

One can easily check using Theorem~\ref{add-del} that these configurations are all free, with splitting types as follows:

\begin{itemize}
\item $(2k+1,2k+3)$ for $\mathcal A_{4k}$;

\item $(2k+2,2k+3)$ for $\mathcal A_{4k+1}$;

\item $(2k+3,2k+3)$ for $\mathcal A_{4k+2}$;

\item $(2k+3,2k+4)$ for $\mathcal A_{4k+3}$;

\end{itemize}
\end{example}

Let us denote by $Z_n$ the set of $n+5$ points dual to the line arrangement $\mathcal A_n$. 

\begin{proposition}
    \label{prop:unexpected irr}
If $k \ge 1$, then $Z_{4k}$ has multiplicity index $m_{Z_{4k}} = 2k +1$, speciality index $u_{Z_{4k}} = 2k +2$, and $Z_{4k}$ admits a unique unexpected curve. It is irreducible and has degree $m_{Z_{4k}} + 1 = 2k +2$. 
\end{proposition} 

\begin{proof}
Since $\mA_{4k}$ has splitting type $(2k+1,2k+3)$, we get the claimed  values of $m_{Z_{4k}}$ and $u_{Z_{4k}}$. Now Theorem \ref{mainThm1} and Proposition \ref{prop:vector space decomposition} give that $Z_{4k}$ admits a unique unexpected curve  of degree $2 k+2$. It remains to show its irreducibility. 

To this end we use Corollary \ref{irred curve least deg}. It shows that we are done once we have proven that removing any line $L$ from the arrangement $\mA_{4k}$ gives an arrangement  $\mA_{4k}\setminus L$, with splitting type $(2k+1, 2k+2)$.

First, let $L$ be any line of $\mA_{4k}$ other than the line at infinity $i$, defined by $z = 0$. Then $L$ meets the other lines of $\mA_{4k}$ in $2k + 2$ points. Hence Addition-Deletion yields that $\mA_{4k} \setminus L$ is a free arrangement with splitting type $(2k+1, 2k+2)$, as claimed. 

Second, consider the line $i$, and set  $\mA' = \mA_{4k} \setminus i$. The line $i$ meets the lines in $\mA'$ in four points. Hence, if $k = 1$ (i.e., $\mA_4$ is the B3 configuration), then we conclude as in the first case that $\mA'$ has splitting type $(3, 4)$, as desired. Let $k \ge 2$. Now we need a different argument. 

Let $h$ be the product of $4 k + 3$ linear forms such that $\mA_{4k} = \mA (z (x^2 - y^2) h)$, and so $\mA' = \mA ( (x^2 - y^2) h)$. As observed above, the arrangement  $\mA (z (x - y) h)$ is free with splitting type $(2k+1, 2k+2)$. Since the line defined by $x-y$ meets $\mA (z h)$ in $2k+2$ points, we see that $\mA (z h)$ is free with splitting type $(2k+1, 2k+1)$. The line $z = 0$ meets the lines of $\mA (h)$ in two points. Hence, the logarithmic bundles are related by the exact sequence (see Sequence \ref{eq:dual hal})
\[
0 \to \mD (h z) \to \mD (h) \to \mO_{\PP^1} (-4 k) \to 0. 
\]
Since $\mD (h z) \cong \mO_{\PP^2}^2 (-2k-1)$ and $\mD (h)_{norm} = \mD (h) (2k)$, we conclude that $H^0 (\mD (h)_{norm}) = 0$, and so $\mD (h)$ is stable by Lemma \ref{stable lemma}. Now Lemma \ref{hal thm} shows that $\mA ((x - y)h)$ is semistable. Hence, its splitting type is $(2k+1, 2k +1)$ by the Grauert-M\"ulich theorem. We have already seen that $\mA_{4k} = \mA (z (x^2 - y^2) h)$ has splitting type $(2k+1, 2k+3)$. Using this information, Lemma \ref{lem:change of index} yields that $\mA' = \mA ( (x^2 - y^2) h)$ has splitting type $(2k+1, 2k+2)$. This completes the argument.  
\end{proof}

\begin{remark}
Let $Z \subset \PP^2$ be a set of points with $2m_Z +2 \le |Z|$. Let $P$ be a general point. Then  $[I_{Z+m_Z P}]_{m_Z + 1}$ determines a unique curve $C_P$.    This curve depends on $P$, and only the degree is necessarily invariant as  $P$ moves. 
Lemma \ref{lem:change of index}  shows that {\color{blue} for any given $P$,} if $Q \in C_P(Z)$ then $m_{Z+Q,P} = m_Z$. Notice that this is not necessarily equal to $m_{Z+Q}$. However, if there is a point $Q \notin Z$ such that $Q \in \bigcap_{P \in \mathbb P^2} C_P(Z)$ then we do obtain $m_Z = m_{Z+Q}$. 

We find it very surprising that such a point $Q$ can exist, i.e. that there can be a new point common to every curve in the family $\{ C_P \}$ (which is not a linear system) as $P$ varies in $\mathbb P^2$. Nevertheless, we saw this already in Example \ref{example20}, and Corollary \ref{irred curve least deg} shows that this has to happen even for \emph{each} point $Q$ of $Z$ when passing from $Z - Q$ to $Z$, provided $2m_Z +3 \le |Z|$ and the curve $C_P$ is irreducible. Indeed, the converse is true as well, and we used it to prove the irreducibility of the unexpected curve in Proposition \ref{prop:unexpected irr}. 
\end{remark}

%%%%%%%%%%%%%%%%%%%%%%%%%%%%%%%%%%%%%%%%%%%%%%%%%%%%%%

\section{Connections and corrections}
\label{sec:connections}

The paper \cite{DIV} introduced connections between the splitting type of the syzygy bundle and two seemingly unrelated topics: the Strong Lefschetz Property for certain ideals of powers of linear forms and Terao's conjecture for planar arrangements. The first version of the current paper was inspired by \cite{DIV}, but it pointed out some inaccuracies in that paper. The paper \cite{DI} continued this investigation by extending somewhat the results of \cite{DIV} and correcting most of the issues that we had pointed out. Thus in this section it is important to keep on  record the  example from our first version that was cited in \cite{DI} as motivating their changes (see Example \ref{ctrex to DIV} below), and to expand on the new observations in \cite{DI} about the connections to unexpected curves. 

\subsection{SLP}

We first recall the main definition.

\begin{definition}
An artinian algebra $A = R/I$ {\em satisfies the Strong Lefschetz Property (SLP) at range $k$ in degree $d$} if, for a general linear form $L$, the homomorphism $\times L^k : [A]_d \rightarrow [A]_{d+k}$ has maximal rank. We say that $A$ {\em fails SLP at  range $k$ in degree $d$ by $\delta > 0$} if, for a general linear form $L$, the multiplication $\times L^k : [A]_d \rightarrow [A]_{d+k}$ has rank $\min \{ h_A(d), h_A (d+k) \} - \delta$.
\end{definition}

We also recall the following important result.

\begin{theorem}[\cite{EI}]   
       \label{thm:inverse-system}
Let $\wp_1, \dots, \wp_m$ be ideals of  $m$ distinct points in $\mathbb P^{n-1}$.   Choose positive integers $a_1,\dots,a_m$, and let $(l_1^{a_1} ,\dots,l_m^{a_m}) \subset R = k[x_1,\dots,x_n]$ be the ideal generated by powers of the  linear forms that are dual to the points  $\wp_i$. 
Then for any integer $j \geq \max\{ a_i\}$,
\[
\dim_K \left [R/ (l_1^{a_1}, \dots, l_n^{a_m} )  \right ]_j =
\dim_K \left [ \wp_1^{j-a_1 +1} \cap \dots \cap \wp_n^{j-a_m+1} \right ]_j .
\]
\end{theorem}

The first version of this paper gave the following example, cited in \cite{DI} without details.

\begin{example} \label{ctrex to DIV}

Let $1 \leq a \leq b-1$.  Define the arrangement $\mathcal A_{a,b}$ by the lines
\[
\begin{array}{l}
	z, \\
	x, x+z, x+2z, ..., x+(a-1)z, \\
	y, y+z, y+2z, ..., y+(b-1)z 
\end{array}
\]
It is easy to see $\mathcal A_{a,b}$ is supersolveable, hence free.  Moreover, using
addition-deletion  (or Remark \ref{stefan}) it is easy to see that the splitting type is $(a,b)$.
Let $Z$ be the set of points dual to these lines. For a concrete example, we will take $a=3$ and $b = 13$
(see Figure \ref{DualA3-13}).

%\vspace{.2in}

\begin{figure}[htbp]
\begin{center}
\begin{tikzpicture}[scale = .5]
\node at (0,0) {$\bullet$};
\node at (1,0) {$\bullet$};
\node at (2,0) {$\bullet$};
\node at (3,0) {$\bullet$};
\node at (4,0) {$\bullet$};
\node at (5,0) {$\bullet$};
\node at (6,0) {$\bullet$};
\node at (7,0) {$\bullet$};
\node at (8,0) {$\bullet$};
\node at (9,0) {$\bullet$};
\node at (10,0) {$\bullet$};
\node at (11,0) {$\bullet$};
\node at (12,0) {$\bullet$};
\node at (13,0) {$\bullet$};
\node at (0,1) {$\bullet$};
\node at (0,2) {$\bullet$};
\node at (0,3) {$\bullet$};
%\node at (5.5,2.5) {\huge{ $\bullet$}};
%\node at (6.5,2.5) {($a$)};
\draw (-1,0) -- (14,0);
\draw (0,-1) -- (0,4);
\draw[snake=brace,segment aspect=0.5] (-.8,.8) -- (-.8,3.2);
\draw[snake=brace,segment aspect=0.5] (13.2,-.8) -- (.8,-.8);
\node at (-1.3,2) {$3$};
\node at (7,-1.5) {$13$};
\end{tikzpicture}
\caption{The points $Z$ dual to $\mathcal A_{3,13}$.}
\label{DualA3-13}
\end{center}
\end{figure}
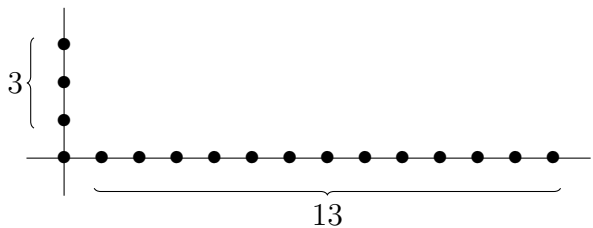

%\vspace{.1in}

\noindent The associated splitting type  is $(3,13)$; thus the derivation bundle is unstable.
It is not hard to compute the Hilbert function of this set of points and to verify that $t_Z = 3$. Since the splitting type immediately gives $m_Z = 3$, we see from Theorem \ref{mainThm1} that $Z$ does not admit an unexpected curve.

Consider the ideal
\[
I = \langle x^8,(x+z)^8,(x+2z)^8,y^8, (y+z)^8,\dots,(y+12z)^8,z^8 \rangle.
\]
Its Hilbert function is 
\[
[1, 3, 6, 10, 15, 21, 28, 36, 33, 27, 19, 12, 7, 3, 1],
\]
as can be verified either on a computer or by hand. For a general linear form $L$, the Hilbert function of $R/(I,L^2)$ is
\[
[1, 3, 5, 7, 9, 11, 13, 15, 5].
\]
Since 
\[
[R/I]_{i-2} \stackrel{\times L^2}{\longrightarrow} [R/I]_i \rightarrow [R/(I,L^2)]_i \rightarrow 0
\]
is exact, a comparison of these two Hilbert functions shows that $\times L^2 : [R/I]_{i-2} \rightarrow [R/I]_{i}$ has maximal rank for all $i$. Thus $R/I$ {\em does} have SLP in range 2. This shows that in order for SLP to fail in range 2, it is not enough to have an unbalanced splitting type. 
%The missing ingredient, provided in \cite[Proposition 3.10]{DI} (see below) but not in \cite{DIV}, is that the ideal of $Z$ must be generated (in this case) in degree $\leq 8 = \frac{3+13}{2}$, which is clearly not the case here. 
\end{example}

The authors of \cite{DI} prove the following correction of a result in \cite{DIV}, based on this example. 

\begin{proposition} [{\cite[Proposition 3.10]{DI}}]
Let $I \subset R = \mathbb C[x,y,z]$ be an artinian ideal generated by $2d+1$ polynomials $\ell_1^d,\dots,\ell_{2d+1}^d$, where $\ell_i$ are distinct linear forms. Let $Z$ be the corresponding points dual to the $\ell_i$. If $Z$ contains no more than $d+1$ points on a line then the following conditions are equivalent:

\begin{itemize}
\item[(i)] The algebra $R/I$ fails the SLP at range 2 in degree $d-2$;

\item[(ii)] The derivation bundle $\mathcal D_0(Z)$ is unstable with splitting type $(d-s,d+s)$ for some $s \geq 1$.
\end{itemize}
\end{proposition}

\noindent The new ingredient in this result compared to \cite[Proposition 7.2]{DIV}  is the condition on points on a line. The authors observe that it is related to the question of whether the forms $\ell_i^d$ are all linearly independent, via Theorem \ref{thm:inverse-system}, but we omit the details here.
In particular, it no longer applies to Example \ref{ctrex to DIV} because of the combination of the numerical constraint and the condition on collinear points. Note that we are maintaining their notation, so their $d$ (the degree of the forms) is not the same as our $d$ (the number of points).

In the following result, we generalize this in two ways. First, there is no numerical assumption. Second, we show that failure of SLP is equivalent to the existence of an unexpected curve.
%\end{remark}

\begin{theorem} \label{SLP condition}
Let $\mathcal A(f)$ be a line arrangement in $\mathbb P^2$, where $f = L_1\cdots L_d$, and let $Z$ be the set of points in $\mathbb P^2$ dual to these lines.  Then $Z$ has an unexpected curve of degree $j+1$ if and only if $R/(L_1^{j+1},\dots,L_d^{j+1})$ fails the SLP in range 2 and degree $j-1$.
\end{theorem}

\begin{proof}
Let $P$ be a general point in $\mathbb P^2$, and let $L$ be the  linear form dual to $P$. Consider the multiplication map  
\[
\times L^2 : [R/(L_1^{j+1},\dots,L_d^{j+1})]_{j-1} \rightarrow [R/(L_1^{j+1},\dots,L_d^{j+1})]_{j+1}. 
\]
Clearly $\dim_K [R/(L_1^{j+1},\dots,L_d^{j+1})]_{j-1} = \binom{j+1}{2}$. By Macaulay duality, 
\[
\dim_K [R/(L_1^{j+1},\dots,L_d^{j+1})]_{j+1} = h^0(\mathcal I_Z(j+1)).
\]
Hence, the expected dimension of the cokernel is $\max \left \{ h^0(\mathcal I_Z(j+1)) - \binom{j+1}{2}, 0 \right \}$. In other words,  $R/(L_1^{j+1},\dots,L_d^{j+1})$ fails the SLP in range 2 and degree $j-1$ if and only if 
\[
\dim_K (\coker (\times L^2))  > \max \left \{ h^0(\mathcal I_Z(j+1)) - \binom{j+1}{2}, 0 \right \}.
\]

Now, the cokernel of the considered multiplication by $L^2$ is $[R/(L_1^{j+1},\dots,L_d^{j+1},L^2)]_{j+1}$. By Theorem \ref{thm:inverse-system}, its dimension is $h^0((\mathcal I_Z \otimes I_P^{j})(j+1))$. Thus, we have shown that 
$R/(L_1^{j+1},\dots,L_d^{j+1})$ fails the SLP in range 2 and degree $j-1$ if and only if 
\[
h^0((\mathcal I_Z \otimes I_P^{j})(j+1)) > \max \left \{ h^0(\mathcal I_Z(j+1)) - \binom{j+1}{2}, 0 \right \}, 
\]
that is, $Z$ admits an unexpected curve of degree $j+1$. 
\end{proof}

\begin{remark}
The last result in \cite{DI}, namely Corollary 3.13, claims to recover our Theorem \ref{SLP condition}. This is not quite true.  Their result makes an assumption on the relation between the number of points, $2d+1$, and the degree $d$ where SLP fails in range 2, as well as an assumption that not too many of the points of $Z$ are collinear (because they need the points to impose independent conditions on forms of degree $d$). We know that if there is an unexpected curve of degree $j+1$ then {\em as a consequence} $h_Z(t_Z) = |Z|$, i.e. the points of $Z$ impose independent conditions on forms of degree $t_Z$. But there is no guarantee that $j+1 \geq t_Z$. We only know (Theorem \ref{mainThm1} and Lemma \ref{m, t and u}) that $a_Z < t_Z$ and hence $t_Z \leq u_Z$. 
\end{remark}

\begin{corollary} \label{explain DIV}
Let $\mathcal A(f)$ be a line arrangement in $\mathbb P^2$ with splitting type $(a,b)$, where $2 \le a \le b$. Let $f = L_1\cdots L_d$ and assume that the ideal generated by the $(a+1)$-st partial derivatives of $f$ is artinian. 
Then   $b-a \geq 2$ if and only if  $R/(L_1^{a+1},\dots,L_d^{a+1})$ fails the SLP at range 2 in degree $a-1$. 
\end{corollary}  

\begin{proof}
The condition on the ideal of partial derivatives guarantees that no  $a+2$ of the lines pass through any point of $\mathbb P^2$. Thus no  $a+2$ of the dual points, $Z$, lie on a line, so we can apply Theorem \ref{thm:geomVersion} with $j=a=m_Z$. Then the result  follows from Theorem \ref{SLP condition}.
\end{proof}

\subsection{Terao's Conjecture}

It is natural to wonder to what extent numerical invariants of a line arrangement are determined by its combinatorial properties. The latter are captured by the \emph{incidence lattice} of the arrangement. It consists of all intersections of lines, ordered by reverse inclusion.  For example, if $\mathcal{A} (f)$ and $\mathcal{A} (g)$  are two line arrangements in $\PP^2$ with the same incidence lattice, then it follows that the Jacobian ideals of $f$ and $g$ have the same degree. 

One of the main open problems is to decide whether freeness of hyperplane arrangements is a combinatorial property. It is open even for line arrangements. 

\begin{conjecture}[Terao]
    \label{conj:Terao} 
Freeness of a line arrangement  depends  only on its incidence lattice. 
\end{conjecture}

The connection between Terao's conjecture and the multiplication by the square of a general linear form on certain quotient algebras was first studied in \cite{DIV}.
Here we want to use our earlier results to state an equivalent version of this conjecture. At the same time we remark on the relevant results and assertions of \cite{DIV}. We need some preparation. 

Consider a vector bundle $\mathcal{E}$ on $\PP^2$ of rank two. As pointed out above, its restriction to a general line $L$ has the form 
$\mathcal{O}_L(-a) \oplus \mathcal{O}_L(-b)$ for some integers $a\leq b$. The pair $(a, b)$ is the (generic) \emph{splitting type} of $\cE$. If $\mathcal E$ splits as a direct sum of line bundles,  then $c_2 (\cE) = ab$, where $c_2 (\cE)$ denotes the second Chern class of $\cE$. The converse is true as well.

\begin{theorem}[{\cite[Theorem 1.45]{Y}}]
  \label{thm:splitt crit}
 For every rank two vector bundle $\cE$ on $\PP^2$ with generic splitting type $(a, b)$,  one has $c_2 (\cE) \ge ab$.  Furthermore, equality is true if and only if $\cE$ splits as a direct sum of line bundles
\end{theorem} 

Recall that the derivation bundle $\mD (f)$ of a line arrangement $\mA (f)$ is the sheafification of the module $D(f)$, defined by the exact sequence 
\[
0 \to D(f) \to R^3 \to \Jac (f) (\deg f -1) \to 0. 
\]
It follows that 
\begin{equation}
   \label{eq:Chern degree} 
   c_2 (\mD (f) ) = (\deg f - 1)^2 -  \deg \Jac (f). 
\end{equation}

We are ready to establish the following result, which is implicitly used in \cite{DIV}.  

\begin{proposition}
    \label{prop:split type free} 
Let $\mA (f)$ and $\mA (g)$ be two line arrangements with the same incidence lattice. Assume $\mA (f)$ is free with splitting type $(a, b)$. Then one has: 
\begin{itemize}
\item[(a)] $\mA (g)$ is free if and only if $\mD (g)$ has the same splitting type as $\mD (f)$.    

\item[(b)] If $\mA (g)$ is not free, then the splitting type of $\mD (g)$ is $(a - s, b+s)$ for some positive integer $s$. 
\end{itemize}
\end{proposition} 

\begin{proof}
Set $d = \deg f$. By Theorem \ref{thm:splitt crit}, since $\mathcal A(f)$ is free we get  $c_2 (\mD (f)) = ab$.  Since the arrangements have the same incidence lattice, Equation \eqref{eq:Chern degree}  gives $c_2 (\mD (g)) = c_2 (\mD (f))$.  Combining, we obtain $c_2 (\mD (g)) = ab$. 

Since $f$ and $g$ have the same degree, the sum of the integers in the splitting type for $\mD(f)$ must be equal to the sum for $\mD(g)$, i.e. the splitting type for $\mD(g)$ is $(a-s,b+s)$ for some integer $s$, where $a-s \leq b+s$.
Combined with Theorem \ref{thm:splitt crit}, and using the fact that $a+b+1=d$, we obtain 
\[
0 \le c_2 (\mD (g)) - (a-s) (b+s)  = a b - (a-s) (b+s) = a (d-1-a) - (a-s) (d-1-a+s). 
\]
Since the function $h (t) = t (d-1-t)$ is strictly increasing on the interval $(-\infty, \frac{d-1}{2}]$, and since both $a$ and $a-s$ lie in this interval, we conclude that  $s \ge 0$ and that $c_2 (\mD (g)) - (a-s) (b+s)  = 0$ if and only if $s = 0$. Hence Theorem \ref{thm:splitt crit} gives that $\mD (g)$ is free if and only if $s = 0$ and that $s > 0$ otherwise. 
\end{proof}

As an immediate consequence we get: 

\begin{corollary} 
     \label{cor:type of free comb}
If the splitting type of a line arrangement is a combinatorial property, then Terao's conjecture is true for line arrangements. 
\end{corollary} 

Thus we propose the following question:

\begin{question}
Is the splitting type a combinatorial invariant for arbitrary arrangements?
\end{question}

Using a Lefschetz-like property, we give a statement that is equivalent to  Terao's conjecture. 

\begin{proposition} 
   \label{prop:Terao equiv} 
The following two conditions are equivalent: 
\begin{itemize}

\item[(a)] Terao's conjecture is true. 

\item[(b)] If $\mA (f)$ is any free line arrangement with splitting type $(a, b)$, then, for every line arrangement $\mA (g)$ with the same incidence lattice as $\mA (f)$, the multiplication map 
\[
[R/J]_{b-2} \stackrel{\times L^2}{\longrightarrow} [R/J]_b
\]
is surjective, where $J = (\ell_1^b,\dots, \ell_{a+b+1}^b, L_1^b,\ldots,L_{b-a}^b)$ with $g = \ell_1 \cdots \ell_{a+b+1}$ and general linear forms $L, L_1,\ldots,L_{b-a} \in R$. 
\end{itemize}
\end{proposition} 

\begin{proof}
Let $\mA (f)$ be a free line arrangement with splitting type $(a, b)$, and let $\mA (g)$ be a line 
arrangement with the same incidence lattice as $\mA (f)$. By Proposition \ref{prop:split type free}, 
the splitting type of $\mA (g)$ is $(a-s, b+s)$ for some integer $s \ge 0$. Let 
$L, L_1,\ldots,L_{b-a} \in R$ be general linear forms, and set $h = L_1 \cdots L_{b-a}$. 
Proposition  \ref{conj:change of dZ} and Lemma \ref{m, t and u}(a) 
give  that $\mA (gh)$ has splitting type $(b-s, b+s)$. Denote by $Z$   the set of points in $\PP^2$ 
that is dual to $\mA (g h)$. It  has multiplicity index $m_Z = b -s$. 
The cokernel of the 
multiplication map 
\[
[R/J]_{b-2} \stackrel{\times L^2}{\longrightarrow} [R/J]_b
\]
is $[R/(J, L^2)]_b$. By Theorem \ref{thm:inverse-system}, this is isomorphic to $[I_{Z + (b-1)P}]_b$, where  $P \in \PP^2$  is the point that is dual to $L$.  It follows that the 
above map is surjective if and only if $m_Z = b$, that is, $s = 0$, which means that   $\mA (g)$ has 
the same splitting type as $\mA (f)$. By Proposition \ref{prop:split type free}, the latter is equivalent 
to $\mA (g)$ being free, which concludes the argument. 
\end{proof} 

Similar arguments give a sufficient condition. 

\begin{corollary} 
   \label{cor:suff Terao}
Consider the following condition
\begin{itemize}
\item[(*)] Let $f = \ell'_1 \cdots \ell'_{2k+1}$ and $g =   \ell_1 \cdots \ell_{2k+1}$ be  products of $2k+1$ linear forms in $R$, and let $L \in R$ be a general linear form. Assume that the multiplication map 
\[
[R/I]_{k-2} \stackrel{\times L^2}{\longrightarrow} [R/I]_k
\]
is surjective, where $I = (\ell'^k_1,\dots, \ell'^k_{2k+1})$.  

If the line arrangements ${\mathcal A} (f)$ and ${\mathcal A} (g)$ have the same incidence lattices, then the 
multiplication map 
\[
[R/J]_{k-2} \stackrel{\times L^2}{\longrightarrow} [R/J]_k
\]
is also surjective, where $J =   (\ell_1^k,\dots, \ell_{2k+1}^k)$. 
\end{itemize} 
If Condition (*) is true for any two sets of $2 k+1$ linear forms, then Terao's conjecture is true. 
\end{corollary} 

\begin{proof} 
Adopt the notation of the proof of Proposition \ref{prop:Terao equiv}. In particular, let  $\mA (f)$ and 
$\mA (g)$ be two line arrangements with the same incidence lattice, where $\mA (f)$ is free with 
splitting type $(a, b)$.   Let $\ell'_1,\ldots,\ell'_{a+b+1}$ be linear forms such that $f = \ell'_1 \cdots 
\ell'_{a+b+1}$. We will use Condition (*) by considering the ideal 
$I = (\ell'^b_1,\dots, \ell'^b_{a+b+1}, L_1^b,\ldots,L_{b-a}^b)$. Indeed, the arrangement 
$\mA (f h)$ has splitting type $(b, b)$. Hence the multiplication map 
\[
[R/I]_{b-2} \stackrel{\times L^2}{\longrightarrow} [R/I]_b
\]
is surjective. Since $L_1,\ldots,L_{b-a}$ are general linear forms, the arrangements $\mA (f h)$ and $\mA (g h)$ also have the same incidence lattice. Therefore, Condition (*) gives that  the  multiplication map 
\[
[R/J]_{b-2} \stackrel{\times L^2}{\longrightarrow} [R/J]_b
\]
is surjective, where $J =  (\ell^b_1,\dots, \ell^b_{a+b+1}, L_1^b,\ldots,L_{b-a}^b)$. As above, it follows that $\mA (g)$ must be a free arrangement, as desired. 
\end{proof}

\begin{remark}\label{DIadjustmentRemark}
(i) In \cite{DIV} the authors conjecture that the above Condition (*) is always satisfied if one replaces surjectivity of the multiplication maps by maximal rank. An assumption on collinearity for the dual points was added in \cite{DI}. Moreover, they claim that this modification of Condition (*) is equivalent to Terao's conjecture,  whereas we claim only one direction.

(ii) We have seen in Example \ref{ctrex to DIV}  that injectivity of the multiplication map is not enough to draw a conclusion on the splitting type. One  needs surjectivity as stated in Condition~(*). However, it is not clear (to us) whether Condition (*) is in fact equivalent to Terao's conjecture. 
\end{remark}

Returning to sets of points, we conclude with the dual version of Corollary \ref{cor:suff Terao}. 

\begin{corollary}
     \label{cor:points to Terao}
If, for sets of  $2k + 1$ points of $\PP^2$, having (maximal) multiplicity index $k$ is a combinatorial property, then  Terao's conjecture is true for line arrangements.   
\end{corollary}

%%%%%%%%%%%%%%%%%%%%%%%%%%%%%%%%%%%%%%%%%%%%%%%%%%%%%%

\section*{Appendix}

\setcounter{theorem}{0}
\setcounter{equation}{0}
\renewcommand{\thesection}{A}

We derive some facts on line configurations for which we could not find a reference in the literature. 

Define the submodule 
$D(Z)\subset R\frac{\partial}{\partial x}\oplus R\frac{\partial}{\partial y}\oplus R\frac{\partial}{\partial z}\cong R^3$ 
to be the $K$-linear derivations $\delta$ such that $\delta(f)\in Rf$.
In particular,  $D(Z)$ contains the Euler derivation 
$\delta_E =x\frac{\partial}{\partial x}+y\frac{\partial}{\partial y}+z\frac{\partial}{\partial z}$,
and $\delta_E$ generates a submodule $R\delta_E \cong R(-1)$. 
We can now define the quotient $D_0(Z)=D(Z)/R\delta_E$. 

Define the \emph{Jacobian ideal} of $f \in R$ as $J = \Jac (f) = (f, f_x, f_y, f_z)$. Let $J' = (f_x,f_y,f_z)$. For $\delta \in D(Z)$, we may view $\delta$ as a triple $(g_1,g_2,g_3)^T$ of  polynomials such that $(g_1 \frac{\partial}{\partial x} + g_2 \frac{\partial}{\partial y} + g_3 \frac{\partial}{\partial z} ) (f) = h f$, for some  $h \in R$ (possibly zero) that depends on $\delta$.  Then the module $D(Z)$ can be described by the exact sequence
\begin{equation} \label{LES}
0 \longrightarrow D(Z) \to R^3  \stackrel{\varphi}{\longrightarrow} (R/f R) (d-1)  \longrightarrow (R/J)(d-1) \longrightarrow 0, 
\end{equation}
where $\varphi ((g_1, g_2, g_3)^T) = g_1 f_x + g_2 f_y + g_3 f_z\!\! \mod f$. Notice that the image of $\varphi$ is $J/fR (d-1)$. 
Using this, consider the commutative diagram
\[
\begin{array}{cccccccccc}
&&&& 0 && 0 \\
&&  && \downarrow && \downarrow \\
&& 0 && E && D(Z)(1-d) \\
&& \downarrow && \downarrow && \downarrow \\
0 & \rightarrow & R(-d) & \rightarrow & R(-d) \oplus R(1-d)^3 & \rightarrow & R(1-d)^3 & \rightarrow & 0 \\
&& \downarrow && \downarrow && \downarrow \\
0 & \rightarrow & R(-d) & \stackrel{\times f}{\longrightarrow} & J & \rightarrow & J/fR & \rightarrow & 0 \\  
&& \downarrow && \downarrow && \downarrow \\
&& 0 && 0 && 0
\end{array}
\]
where $E$ is the syzygy module associated to $J$. Then from the Snake Lemma we see that $D(Z)$ is isomorphic to a twist of the syzygy module of $J$. In particular, $D(Z)$ is reflexive. 
Its sheafification $\widetilde{D(Z)}$ is a locally free sheaf of rank three. Since the Euler derivation corresponds to a global non-vanishing section, it follows that the sheafification of $D_0(Z)$  is a locally free sheaf on $\PP^2$ of rank two, which we will denote by $\mathcal D_Z$.  We  call $\mathcal D_Z$ the {\em derivation bundle of~$Z$}.

Moreover, from the commutative diagram
\[
\begin{array}{cccccccccccc}
&& 0 && 0 \\
&& \downarrow && \downarrow \\
0 & \rightarrow & R \delta_E & \rightarrow & R \delta_E & \rightarrow & 0 \\
&& \downarrow && \downarrow && \downarrow \\
0 & \rightarrow & D(Z) & \rightarrow & R^3 & \stackrel{\varphi}{\longrightarrow} & J/fR (d-1) & \rightarrow & 0 \\
&& \downarrow && \downarrow && \downarrow \\
&& D_0(Z) && R^3/R \delta_E && J/fR(d-1) \\
&& \downarrow && \downarrow && \downarrow \\
&& 0 && 0 && 0
\end{array}
\]
we get the exact sequence
\[
0 \longrightarrow D_0 (Z) \to R^3/R \delta_E  \longrightarrow (J/fR)(d-1) \longrightarrow 0. 
\]
Notice that the sheafification of $R^3/R \delta_E$ is isomorphic to the tangent bundle, $\mathcal T_{\mathbb P^2}$, of $\PP^2$ twisted by $(-1)$. Thus, $\mathcal D_Z$  is a  subbundle of this twisted tangent bundle. We now compute its first Chern class. 

Sheafifying the above exact sequence, we obtain
\[
0 \to \mathcal O_{\mathbb P^2}(-1) \to \mathcal J (d-1) \to (\mathcal J/f \mathcal O_{\mathbb P^2}) (d-1) \to 0
\]
where $\mathcal J$ is the sheafification of $J$.
Thus, we get 
\[
c_1 ((\mathcal J/f \mathcal O_{\mathbb P^2}) (d-1)) = d-1 - (-1) = d. 
\]
Hence, the sequence 
\[
0 \longrightarrow D_0 (Z) \to R^3/R \delta_E  \longrightarrow (J/fR)(d-1) \longrightarrow 0 
\]
gives, after sheafifying, 
\[
c_1 (\mathcal D_Z) = c_1 (\mathcal T_{\mathbb P^2}(-1)) - c_1 (\mathcal J/f \mathcal O_{\PP^2} (d-1)) = 1 - d.
\]

Now let $J' = (f_x, f_y,f_z)$. Let $E' = Syz(J')(d-1)$ be the twisted syzygy module of $J'$, which is reflexive of rank 2. Consider the following commutative diagram. 
\[
\begin{array}{cccccccccccc}
&&&& 0&& 0 \\
&&&&\downarrow && \downarrow \\
&&&& E && D_0(Z) \\
&&&& \downarrow &&  \downarrow \\
0 & \rightarrow & R(-1) & \stackrel{[x \ y \ z]^T}{\longrightarrow} & R^3 & \rightarrow & R^3/R \delta_E & \rightarrow & 0 \\
&& \phantom{\alpha} \downarrow { \alpha} && \phantom{\beta} \downarrow {\beta} && \downarrow \\
0 & \rightarrow & R(-1) & \stackrel{\cdot f}{\longrightarrow} & J(d-1) & \rightarrow & (J/fR)(d-1) & \rightarrow & 0 \\
&&&&&& \downarrow \\
&&&&&& 0
\end{array}
\]
where $\alpha$ is multiplication by $d$ and $\beta$ is the presentation matrix for $R/J'$. When $\hbox{char}(K)$ does not divide $d$, we have that $\alpha$ is an isomorphism and $J = J'$. It follows that $D_0(Z) \cong E$.  When $\hbox{char}(K)$ does  divide $d$, $\alpha$ is the zero map and we obtain 
\[
\begin{array}{cccccccccccc}
&& 0 && 0&& 0 \\
&& \downarrow &&\downarrow && \downarrow \\
&& R(-1) && E && D_0(Z) \\
&& \downarrow && \downarrow &&  \downarrow \\
0 & \rightarrow & R(-1) & \stackrel{[x \ y \ z]^T}{\longrightarrow} & R^3 & \rightarrow & R^3/R \delta_E & \rightarrow & 0 \\
&& \phantom{\cdot 0} \downarrow { \cdot 0} && \phantom{\beta} \downarrow {\beta} && \downarrow \\
0 & \rightarrow & R(-1) & \stackrel{\cdot f}{\longrightarrow} & J(d-1) & \rightarrow & (J/fR)(d-1) & \rightarrow & 0 \\
&& \downarrow && \downarrow && \downarrow \\
&&R(-1) && J/J' && 0 \\
&& \downarrow && \downarrow \\
&& 0 && 0
\end{array}
\]
so the Snake Lemma gives the long exact sequence
\[
0 \rightarrow R(-1) \rightarrow E \rightarrow D_0(Z) \rightarrow R(-1) \rightarrow J/J' \rightarrow 0.
\]

These calculations produce the following lemma.

\begin{lemma} 
          \label{syzygy bundle}
Let $Z$ be a  set of $d$ points dual to a line arrangement defined by a product, $f$, of linear forms. 
Let $J' = (f_x,f_y,f_z)$ and $J = (f_x,f_y,f_z,f)$, and let $\mathcal D_Z$ be the associated derivation 
bundle. Let $L$ be a general line. Then $\mathcal D_Z|_L$ splits as a direct sum 
$\mathcal O_{\mathbb P^1}(-a_Z) \oplus \mathcal O_{\mathbb P^1}(-b_Z)$ with $a_Z+b_Z = d-1$. 
Furthermore, if $\mathcal E = \widetilde{E}$ is the syzygy bundle of $J'(d-1)$, then $\mathcal D_Z$ 
is isomorphic to $\mathcal E$ if and only if $\hbox{char}(K)$ does not divide $d$. If 
$\hbox{char}(K)$ does divide $d$ then $\mathcal E$ and $\mathcal D_Z$ are related by the exact sequence
\[
0 \rightarrow \mathcal O_{\mathbb P^1}(-1) \rightarrow \mathcal E \rightarrow \mathcal D_Z \rightarrow \mathcal O_{\mathbb P^1}(-1) \rightarrow \widetilde{J/J'} \rightarrow 0.
\]
\end{lemma}

\begin{definition}
We shall call the ordered pair $(a_Z,b_Z)$, with $a_Z \leq b_Z$,  the {\em splitting type of~$Z$}. 
\end{definition}

\begin{remark}
When $\operatorname{char}(K)$ does not divide $\deg(f)$, so $J = J'$, we can see the identification of $Syz(J)$ with $\{ \delta \in D(Z) \ | \ \delta(f) = 0 \}$ more directly.  Indeed, it is not hard to show  that that we have an isomorphism of $R$-modules
\[
D(Z) \rightarrow R \delta_E \oplus [Syz(\Jac(f))](d-1)
\]
defined by 
\[
\delta = (g_1,g_2,g_3) \mapsto \frac{1}{d} h \delta_E + \left (\delta - \frac{1}{d} h \delta_E \right )
\]
(with $h$ defined as above in terms of $\delta$). It follows that $D_0 (Z)  \cong Syz (J) (d-1)$.
Notice that the isomorphism is defined if and only if if the degree $d$ is a unit of $R$.
We thus have the exact sequence of sheaves
\[
0 \rightarrow \mathcal D_Z \rightarrow \mathcal O_{\PP^2}^3 \rightarrow \mathcal J(d-1) \rightarrow 0
\]
where $\mathcal J$ is the sheafification of $J$. This identification of $\mathcal D_Z$ with the syzygy bundle of $J$ is often very useful.
\end{remark}


\begin{thebibliography}{BDHHSS}

\bibitem[BDHHSS]{BDHHSS}
Th.\ Bauer, S.\ Di Rocco, B.\ Harbourne, J.\ Huizenga, A.\ Seceleanu, T.\ Szemberg,
{\it Negative curves on symmetric blowups of the projective plane, resurgences and Waldschmidt constants}, 
Preprint 2016, arXiv:1609.08648.

\bibitem[BGM]{BGM}
A.\ Bigatti, A.\ V. Geramita and J.\ Migliore, 
{\it Geometric consequences of extremal behavior in a theorem of Macaulay}, 
Trans.\ Amer.\ Math.\ Soc.  {\bf 346} (1994),  203--235.


\bibitem[Ca]{Camp} 
G.\ Campanella, 
{\it Standard bases of perfect homogeneous polynomial ideals of height 2},
J.\  Algebra {\bf  101} (1986) 47--60.



\bibitem[CM]{refCM}
C.\ Ciliberto, R.\ Miranda, 
{\it The Segre and Harbourne--Hirschowitz conjectures}, in:
Applications of Algebraic Geometry to Coding Theory, Physics and Computation (Eilat, 2001), 
NATO Sci.\ Ser.\ II Math.\ Phys.\ Chem.\ {\bf 36}, Kluwer Academic, Dordrecht (2001), 37--51.


\bibitem[CoCoA]{cocoa}   CoCoA: a system for doing  Computations in Commutative Algebra, Available at {\tt http://cocoa.dima.unige.it.}

\bibitem[Da]{Da}
E.\ D.\ Davis, 
{\it 0-dimensional subschemes of $\PP^2$: New application of Castelnuovo's function}, 
 Ann.\ Univ.\ Ferrara Sez.\ VII (N.S.) {\bf 32} (1986), 93--107. 


\bibitem[DGM]{DGM} 
E.\ D.\ Davis, A.\ V. Geramita and P.\  Maroscia,  
{\it Perfect Homogeneous Ideals: Dubreil's Theorems Revisited}, 
Bull.\ Sci.\ Math.\ (2) {\bf 108} (1984),  143--185.


\bibitem[DI]{DI} R.\ Di Gennaro and G.\ Ilardi, {\it More on ``singular hypersurfaces characterizing the Lefschetz properties''}, Preprint, 2016.

\bibitem[DIV]{DIV} 
R.\ Di Gennaro, G.\ Ilardi and J.\ Vall\`es,  
{\it Singular hypersurfaces characterizing the Lefschetz properties}, 
J.\ London Math.\ Soc. (2) {\bf 89} (2014), no.\ 1, 194--212 (arXiv:1210.2292).


\bibitem[EI]{EI} 
J.\ Emsalem and A.\ Iarrobino, 
{\it Inverse system of a symbolic power $I$}, 
J.\ Algebra {\bf 174} (1995), 1080--1090.

\bibitem[FV]{FV2} 
D.\ Faenzi and J.\ Vall\`es, 
{\it Logarithmic bundles and line arrangements, an approach via the standard construction}, 
J.\ London Math.\ Soc.\ (2) {\bf 90} (2014), 675--694.

\bibitem[GHM]{GHM} 
A.V.\ Geramita, B.\ Harbourne and J.\ Migliore, 
{\it Star configurations in $\mathbb P^n$}, 
J.\ Algebra {\bf 376} (2013), 279--299. 

\bibitem[G]{G}
A.\ Gimigliano,  {\it On linear systems of plane curves}, Thesis, Queen's University, Kingston, 1987.

\bibitem[GM]{GM}
H.\ Grauert  and G.\ M\"ulich, {\it Vektorb\"undel vom Rang 2 \"uber dem $n$-dimensionalen 
komplex-projektiven Raum},  Manuscripta Math. {\bf 16} (1975), 75--100.


\bibitem[Ha1]{Ha1} 
B.\ Harbourne, 
{\it  The geometry of rational surfaces and Hilbert functions of points in the plane}, Proceedings of
the 1984 Vancouver Conference in Algebraic Geometry, CMS Conf. Proc. {\bf 6},  Amer. Math. Soc., Providence,
RI (1986),  95--111.

\bibitem[Ha2]{Ha} 
B.\ Harbourne, 
{\it Anticanonical rational surfaces}, 
Trans.\ Amer.\ Math.\ Soc.\ {\bf 349} (1997) 1191--1208. 

\bibitem[Ha3]{brianlect}
B. Harbourne,
{\em Asymptotics of linear systems with connections to line arrangements}, Lecture notes, Spring, 2016 miniPAGES. Preprint.



\bibitem[H1]{Hrt} 
R.\  Hartshorne,  {\it Algebraic Geometry},  Springer-Verlag, New York, 1977.

\bibitem[H2]{H} 
R.\  Hartshorne, {\it Stable reflexive sheaves},  Math. Ann. {\bf 254} (1980), 121--176.

\bibitem[Hi]{Hi}
A.\ Hirschowitz, {\it Une conjecture pour la cohomologie des diviseurs sur les surfaces rationelles g\'en\'eriques}, J.
Reine Angew. Math. {\bf 397} (1989), 208--213.

\bibitem[I]{ilardi}
G. Ilardi, {\em Jacobian ideals, arrangements and the Lefschetz properties}, preprint.

%\bibitem[Kl]{Kl}
%S.\ Kleiman, {\it  Bertini and his two fundamental theorems. Studies in the history of modern mathematics, III},  Rend. Circ. Mat. Palermo (2) Suppl. No. {\bf 55} (1998), 9--37.


\bibitem[M2]{M2}
D.\ R.\ Grayson and M.\ E.\ Stillman, 
{\it Macaulay2, a software system for research in algebraic geometry},
Available at \url{http://www.math.uiuc.edu/Macaulay2/}. 



\bibitem[OSS]{OSS}
C.\ Okonek, M. Schneider, and H.\  Spindler, {\it Vector bundles on complex projective spaces},  Progress in Mathematics {\bf 3},  Birkh\"auser, Boston, Mass., 1980.

\bibitem[OT]{OT} 
P.\ Orlik and H.\ Terao, {\it Arrangement of hyperplanes}, Grundlehren
der Mathematischen Wissenschaften {\bf 300}, 
Springer-Verlag, Berlin, 1992.

\bibitem[S]{hal} H.\ Schenck,  
{\it Elementary modifications and line configurations in $\mathbb P^2$}, 
Comment.\ Math.\ Helv.\ {\bf 78} (2003), 447--462. 

\bibitem[Se]{segre}
B.\ Segre, {\it Alcune questioni su insiemi finiti di punti in geometria algebrica},  Atti del Convegno Internazionale di Geometria Algebrica (Torino, 1961),  Rattero, Turin, 1962, 15--33. 

\bibitem[Su]{refSuciu}
A.\ Suciu, {\it Fundamental groups, Alexander invariants, and cohomology jumping loci}, pp.\ 179--223, in:
Contemporary Math.\ {\bf 538}, Topology of Algebraic Varieties and Singularities, editors
J.\ I.\ Cogolludo-August\'in, E.\ Hironaka, 2011.

\bibitem[U]{U}
G.\ A.\ Urz\'ua, 
{\it Arrangements of curves and algebraic surfaces},
Thesis (Ph.D.) University of Michigan,  2008. 

\bibitem[Y]{Y}
M.\ Yoshinaga, {\em Freeness of hyperplane arrangements and related topics}, 
Ann.\ Fac.\ Sci.\ Toulouse Math.\ (6) {\bf 23} (2014), 483--512.


\end{thebibliography}
\end{document}